\documentclass[11pt]{article}
\usepackage{amsfonts,amssymb, latexsym, amsmath, amsthm,mathrsfs, verbatim}
\usepackage{graphics}
\usepackage{url,color}
\usepackage{enumerate}
\usepackage{esint}
\usepackage{tikz}
\usepackage{pgfplots}
\usepackage{pifont}
\usepackage{upgreek}
\usepackage{times}
\usepackage{calligra}
\usepackage{graphicx}
\usepackage{caption}
\usepackage{appendix}
\usepackage[margin=1in]{geometry}

\def\R{\mathbb R}

\def\E{\mathcal{E}}
\def\D{\mathcal{D}}
\def\be{\begin{equation}}
\def\ee{\end{equation}}
\def\bea{\begin{eqnarray}}
\def\eea{\end{eqnarray}}
\def\beas{\begin{eqnarray*}}
\def\eeas{\end{eqnarray*}}

\def\g{\partial}

\def\l{\lambda}
\def\pa{\partial }
\def\Energy{\mathcal E}
\def\S{\mathcal S}

\newtheorem{theorem}{Theorem}[section]
\newtheorem{definition}[theorem]{Definition}
\newtheorem{proposition}[theorem]{Proposition}

\newtheorem{lemma}[theorem]{Lemma}
\newtheorem{remark}[theorem]{Remark}

\title{Nonlinear stability of 
expanding star solutions in the radially-symmetric mass-critical Euler-Poisson system}

\author{Mahir Had\v zi\'c\thanks{Department of Mathematics, King's College London, London, UK. Email: mahir.hadzic@kcl.ac.uk.} \ and Juhi Jang\thanks{Department of Mathematics, University of Southern California, Los Angeles, CA 90089, USA. Email: juhijang@usc.edu.}}

\date{}
\begin{document}

\maketitle

\abstract{We prove nonlinear stability of 
compactly supported expanding star-solutions of the mass-critical gravitational Euler-Poisson system. These special solutions were discovered by Goldreich and Weber in 1980. The expanding rate of such solutions can be either self-similar or non-self-similar (linear), and we treat both types. 
An important outcome of our stability results is the existence of a new class of global-in-time radially symmetric solutions, which are not homogeneous and therefore not encompassed by the existing works.
Using Lagrangian coordinates we reformulate the associated free-boundary problem as a degenerate quasilinear wave equation on a compact spatial domain.
The problem is mass-critical with respect to an invariant rescaling and the analysis is carried out 
in similarity variables. 
}

\tableofcontents
\section{Introduction}
One of the most fundamental models of a Newtonian star is given by the compressible Euler-Poisson system.
Here the star is idealized as a self-gravitating fluid/gas, kept together by the self-consistent gravitational force-field.
We assume that the fluid density function $\rho:\R^3\to\R$ is initially supported on a compact domain $\Omega_0\subset \R^3$.
As the system evolves the fluid support $\Omega(t)$ changes in time and we are naturally led to a moving (free) boundary problem.
The remaining unknowns in the problem are the fluid pressure $p:\R^3\to\R,$ the fluid velocity $\mathbf{u}:\R^3\to\R^3,$ the gravitational potential $\Phi:\R^3\to\R,$
and the fluid support $\Omega(t).$
Coupling the fluid evolution to the Newton's gravitational theory we
arrive at the Euler-Poisson system given by 

\begin{subequations}
\label{E:EULERPOISSON}
\begin{alignat}{2}
\g_t\rho + \text{div}\, (\rho \mathbf{u})& = 0 &&\ \text{ in } \ \Omega(t)\,;\label{E:CONTINUITYEP}\\
\rho\left(\g_t  \mathbf{u}+ ( \mathbf{u}\cdot\nabla) \mathbf{u}\right) +\nabla p &=-\rho \nabla\Phi&&\ \text{ in } \ \Omega(t)\,;\label{E:VELOCITYEP}\\
\Delta \Phi  = 4\pi\rho, \ \lim_{|x|\to\infty}\Phi(t,x) & = 0&& \ \text{ in } \ \R^3 \,; \label{E:POISSON}\\
p&=0&& \ \text{ on } \ \partial\Omega(t) \,;\label{E:VACUUMEP} \\
\mathcal{V}_{\partial\Omega(t)}&= \mathbf{u}\cdot \mathbf{n}(t)  && \ \text{ on } \ \partial\Omega(t)\,;\label{E:VELOCITYBDRYEP}\\
(\rho(0,\cdot),  \mathbf{u}(0,\cdot))=(\rho_0,  \mathbf{u}_0)\,, & \ \Omega (0)=\Omega _0&&\,.\label{E:INITIALEP}
\end{alignat}
\end{subequations}
Here $\partial \Omega(t)$ denotes the boundary of the moving domain, $\mathcal{V}_{\partial\Omega(t)}$ denotes the normal velocity of $\partial\Omega(t)$, 
and $\mathbf{n}(t) $ denotes the outward unit normal vector to $\partial\Omega(t)$. 
We refer to the system of equation~\eqref{E:EULERPOISSON} as the EP-system.

Equation~\eqref{E:CONTINUITYEP} is the well-known continuity equation, while~\eqref{E:VELOCITYEP} expresses the conservation of momentum.
Equation~\eqref{E:POISSON} is the Poisson equation equipped with a suitable asymptotic boundary condition. 
Boundary condition~\eqref{E:VACUUMEP} is the vacuum boundary condition, while~\eqref{E:VELOCITYBDRYEP} is the kinematic boundary condition stating that the boundary
movement is tangential to the fluid particles.

In this article we shall only consider ideal barotropic fluids where the pressure is a function of the fluid density, expressed through the following equation of state
\be\label{E:EQUATIONOFSTATE}
p = \rho^{\gamma}, \ \ 1<\gamma<2.
\ee
The EP-system~\eqref{E:EULERPOISSON} with the polytropic equation of state~\eqref{E:EQUATIONOFSTATE} will be referred to as the EP$_{\gamma}$-system.

The most famous class of special solutions of the EP$_\gamma$-system are the Lane-Emden steady states.
They are spherically symmetric steady states of the form 
\[
(\rho,{\bf u}) \equiv (\rho_*,{\bf 0}).
\]
Plugging this  ansatz in~\eqref{E:EULERPOISSON} the problem reduces to a single ordinary differential equation (ODE) for the enthalpy $w:=\rho_*^{\gamma-1}$: \begin{align}\label{E:LANEEMDEN}
\pa_{rr}w+\frac2r\pa_rw+\pi w^{\frac{1}{\gamma-1}} = 0.
\end{align}
It is well-known~\cite{Ch, ZeNo, BiTr, Jang2014} that 
for any $\gamma\in(\frac65,2)$ there exits a compactly supported steady solution to~\eqref{E:LANEEMDEN} with the boundary conditions $w'(0)=w(1)=0.$
The associated steady state density $\rho_*$ has finite mass  and compact support and therefore represents a steady star. Moreover, classical linear stability arguments \cite{lin}
give the following dichotomy in the stability behavior of the above family of steady stars:
\begin{align}
&\text{if} \ \ \frac65<\gamma<\frac43, \ \text{ steady state $(\rho_*,{\bf 0})$ is linearly unstable}; \notag \\
&\text{if} \ \ \frac43\le\gamma<2 , \ \text{ steady state $(\rho_*,{\bf 0})$ is linearly stable}.  \notag 
\end{align}
The case of $\gamma=\frac43$ admits 0 as the first eigenvalue and it is often referred to as neutrally stable. 
Under the assumption that a global-in-time solution exists, nonlinear stability of Lane-Emden steady stars in the range $\frac43<\gamma<2$ has been shown by variation arguments in an energy-based topology, see~\cite{Rein, LuSm}.
Since the result is conditional upon the existence of a solution, it remains an important open problem to prove or disprove the full nonlinear stability.  
Nonlinear instability in the range $\frac65\le\gamma<\frac43$ has been rigorously established by the second author~\cite{J0,Jang2014}. In this case, the instability is induced by the existence of 
a growing mode in the linearized operator, while there are no such modes when $\frac43\le\gamma<2$. As we shall see in Section~\ref{S:SPECIALSOLUTIONS},  Lane-Emden stars in the case $\gamma=\frac43$ are in fact
nonlinearly unstable despite the absence of growing modes in the linearized operator.

The special significance of the values $\gamma=\frac65$ and $\gamma=\frac43$ is better illustrated by the following scaling analysis.
If $(\rho,\mathbf{u})$ is a solution of the EP$_\gamma$-system, so is the pair $(\tilde\rho,\tilde {\mathbf{u}})$ defined by
\begin{align}
\rho(t,x) & = \lambda^{-\frac{2}{2-\gamma}} \tilde{\rho}(\frac{t}{\lambda^{\frac{1}{2-\gamma}}}, \frac{x}{\lambda}) , \label{E:SSDENSITY}\\
\mathbf{u}(t,x) & = \lambda^{-\frac{\gamma-1}{2-\gamma}} \tilde{\mathbf{u}}(\frac{t}{\lambda^{\frac{1}{2-\gamma}}}, \frac{x}{\lambda}) , \label{E:SSVELOCITY}
\end{align}
for some $\l>0.$
The associated pressure $\tilde p$ and the gravitational potential $\tilde\Phi$ relate to $p$ and $\Phi$ via:
\begin{align*}
p(t,x) & = \lambda^{-\frac{2\gamma}{2-\gamma}} \tilde{p}(\frac{t}{\lambda^{\frac{1}{2-\gamma}}}, \frac{x}{\lambda}), \\
\Phi(t,x) & = \lambda^{-\frac{2\gamma-2}{2-\gamma}} \tilde{\Phi}(\frac{t}{\lambda^{\frac{1}{2-\gamma}}}, \frac{x}{\lambda}).
\end{align*}

The self-similar rescaling~\eqref{E:SSDENSITY}--\eqref{E:SSVELOCITY} is at the heart of our analysis. Two fundamental conserved quantities associated 
with the EP$_\gamma$-system are the total mass
\be\label{E:MASS}
M(\rho) := \int_{\mathbb R^3} \rho \, dx,
\ee
and energy
\be\label{E:ENERGY}
E(\rho,\mathbf{u}) :=  \int_{\mathbb R^3} \left[ \frac12 \rho |\mathbf{u}|^2+ \frac12 \rho \Phi + \frac1{\gamma-1}\rho^{\gamma}\right]\,dx. 
\ee
It is easy to examine their behavior with respect to the rescaling~\eqref{E:SSDENSITY}--\eqref{E:SSVELOCITY}:
\begin{align}
M(\rho) = \lambda^{\frac{4-3\gamma}{2-\gamma}}M(\tilde\rho), \ \ E(\rho,\mathbf{u}) = \lambda^{\frac{6-5\gamma}{2-\gamma}} E(\tilde\rho,\tilde {\mathbf{u}}).
\end{align}
The total mass $M$ is invariant under the self-similar rescaling~\eqref{E:SSDENSITY}--\eqref{E:SSVELOCITY} when $\gamma=\frac43$ and the energy $E$ when $\gamma=\frac65$, 
so we refer to the cases $\gamma = \frac43$ and $\gamma = \frac65$ as the {\em mass-critical} and the {\em energy-critical} case respectively.

As all the new results in this work pertain to the case of spherical symmetry, we next formulate the EP$_\gamma$-system in radial symmetry.
Let
$r=|x|,$ $\mathbf{u}(t,x)=v(t,r)\frac{x}{r},$ and 
\[
\lambda(t) := \text{ radius of the support of fluid}.
\]
Then the system~\eqref{E:EULERPOISSON} takes the form:
\begin{subequations}
\begin{alignat}
\partial_t\rho + v\partial_r \rho + \rho\left(\partial_r v + \frac 2r v\right) &= 0, \ \ r\le\lambda(t), \label{E:CONTINUITYRADIAL}\\
\rho\left(\partial_t v+ v \partial_r v \right) + \partial_r p + \rho \partial_r\phi & = 0, \ \ r\le\lambda(t),\\
\partial_{rr}\Phi + \frac2r \partial_r \Phi  = 4\pi \rho, \ \ \lim_{r\to\infty}\Phi(t,r)& =0, \ \ r\ge0, \label{E:POISSONRADIAL}\\
\rho(t,\lambda(t)) &= 0, \label{E:VACUUM}\\
\dot\lambda(t) & = v(t,\lambda(t)),\label{E:DYNAMICRADIAL}\\
\lambda(0) = \lambda_0, \rho(0,r) = \rho_0(r), v(0,r) &= v_0(r).\label{E:INITIALRADIAL}
\end{alignat}
\end{subequations}
We note that under the assumption~\eqref{E:EQUATIONOFSTATE}, conditions~\eqref{E:VACUUM} and~\eqref{E:VACUUMEP} are equivalent.

The mass-critical regime $\gamma = \frac43$ will be the focus of this article. Exponent $\gamma = \frac43$ in the polytropic gas pressure law~\eqref{E:EQUATIONOFSTATE} is also referred to as the radiation case~\cite{Ch}. 
Combining the ideal gas law and the Stefan-Boltzmann radiation law the effective formula for the pressure takes the form $p \sim \rho \Theta + \Theta^4$, where $\Theta$ is the temperature of the gas. The two terms on the right-hand side are in balance when $\Theta$ is approximated by $\rho^{\frac13}$ and this precisely leads to the radiative case $\gamma=\frac43$.

The case $\gamma=\frac43$ is particularly appealing in the astrophysics community as it lends itself to a simple scenario of stellar collapse/expansion:
a compact fluid body can shrink or expand in a self-similar manner by cascading between different scales $t\mapsto \lambda(t)$ thereby preserving the total mass.
When $\gamma = \frac43$ we have $\frac1{2-\gamma} = \frac32$ and therefore the self-similar expansion/collapse would
heuristically correspond to the rates
\be\label{E:SSASYMPTOTICS}
\bar\lambda_{\text{expansion}}(t) \sim_{t\to\infty} k_1 t^{2/3}, \ \  \ \bar\lambda_{\text{collapse}}(t) \sim_{t\to T} k_2(T- t)^{2/3} , 
\ee
for some positive constant $k_1,k_2>0$ and a collapse time $T>0.$ 
Here $\bar\lambda_{\text{expansion}}(t)$ and $\bar\lambda_{\text{collapse}}(t)$ represent the radii of an expanding and collapsing star respectively.
A remarkable feature of the EP$_{\frac43}$-system is that special solutions exhibiting the self-similar behavior~\eqref{E:SSASYMPTOTICS} can be explicitly constructed! 
They were first discovered in the work of Goldreich and Weber~\cite{GoWe} in 1980 and their energy is exactly zero.  
A wider class of Type I collapsing and expanding solutions, although implicitly contained in~\cite{GoWe},
was discovered independently by Makino~\cite{Makino92} in 1992 and Fu \& Lin~\cite{FuLin} in 1998.
Solutions from~\cite{Makino92} and~\cite{FuLin} either expand or collapse at a 
{\em linear} rate 
\be\label{E:MAKINORATES}
\tilde\lambda _{\text{expansion}}(t) \sim_{t\to\infty} \tilde k_1 t , \ \ \ \tilde\lambda _{\text{collapse}}(t) \sim_{t\to\tilde T} \tilde k_2 (\tilde T - t),  
\ee
for some positive constant $\tilde k_1,\tilde k_2>0$ and a collapse time $\tilde T>0.$ 
Rates~\eqref{E:MAKINORATES} are not self-similar as opposed to~\eqref{E:SSASYMPTOTICS} and as is shown in Section~\ref{S:SPECIALSOLUTIONS}
their energy is always different from zero. A synthetic treatment of the two types of special homogeneous solutions using the Lagrangian coordinates is given in Section~\ref{S:SPECIALSOLUTIONS}.


\subsection{Special expanding solutions}

To describe the special homogeneous\footnote{Our use of the word "homogeneous" is only justified if we interpret these special solutions in Lagrangian coordinates, see Section~\ref{S:SPECIALSOLUTIONS}.} expanding solutions in Eulerian coordinates, we need to specify three parameters:
\[
(\lambda_0, \ \lambda_1, \ \delta) \in\mathbb R_+\times\mathbb \R\times [\delta^*,\infty),
\]
where $\delta^*<0$ is a real number defined in Section~\ref{S:SPECIALSOLUTIONS}. 
The expanding radius satisfies the ordinary differential equation:
\be\label{E:ODE}
\ddot\lambda \lambda^2 =\delta,
\ee
with the initial conditions
\be\label{E:BClambda}
\lambda(0)=\lambda_0>0, \quad \dot{\lambda}(0) =\lambda_1.
\ee
The associated density and velocity field are given by:
\be\label{profile}
\rho(t,r) = \lambda(t)^{-3} w^3(\frac r{\lambda(t)}), \ \ v(t,r) = \frac{\dot{\lambda}(t)}{\lambda(t)} r, 
\ee
where 
$w:[0,1]\to\R_+$ is a non-negative function called {\em enthalpy} solving the generalized Lane-Emden equation:
\be\label{E:gLE0}
w'' + \frac 2z w' + \pi w^3 =-\frac34 \delta, \ \ \text{ in } \ \ [0,1], \ \ w'(0)=0, \ \  w(1)=0 .
\ee
The conserved energy $E$ associated with these solutions takes the form (see Lemma \ref{lem:E}):
\begin{align}\label{E:ENERGYSPECIAL}
E(\lambda,\dot\lambda ) = \left(\dot\lambda^2+\frac{2\delta}{\lambda}\right)  \int_0^{1}2\pi w^3z^4 dz .
\end{align}
Here we have written $E(\lambda,\dot\lambda )=E(\rho,v)$ where $\rho$ and $v$ are given in \eqref{profile}. Throughout the article, $E$ will represent the physical energy but we will abuse the notation for the arguments in order to emphasize the dependence of the corresponding dynamical quantities for different formulations. 


\subsubsection{Self-similar solutions}

It can be checked  (see Section~\ref{S:SPECIALSOLUTIONS}) that for any $\delta\in[\delta^*,0)$ and any $(\lambda_0,\lambda_1)$ satisfying
\be\label{E:INITIALCONSTRAINT0}
E(\lambda_0,\lambda_1) =\left( \lambda_1^2+\frac{2\delta}{\lambda_0}\right)  \int_0^{1}2\pi w^3z^4 dz  = 0
\ee
there exists a self-similar solution $\bar\lambda(t)$ given explicitly by:
\be\label{E:SSRATE0}
\bar\lambda(t)=\left(\lambda_0^{3/2} + \frac32\lambda_0^{1/2}\lambda_1 t \right)^{2/3}
\ee
with the density $\rho$, velocity $v$, and the enthalpy $w$ given by~\eqref{profile}--\eqref{E:gLE0}.
Note that the cases $\lambda_1>0$ and $\lambda_1<0$ correspond to expansion and collapse respectively.


\begin{definition}[Self-similar expanding homogeneous solutions]\label{D:SSEXPANSION}
The family of solutions given by~\eqref{E:SSRATE0} with $\bar\rho= \bar\lambda(t)^{-3} w^3(\frac r{\bar\lambda(t)})$, $v(t,r) = \frac{\dot{\bar\lambda}(t)}{\bar\lambda(t)} r$ and~\eqref{profile}--\eqref{E:gLE0} is denoted by
\[
(\bar\lambda,\bar\rho,\bar v)_{\lambda_0,\lambda_1,\delta}
\]
with $(\lambda_0,\lambda_1,\delta)\in\mathbb R_+\times\mathbb R_+\times [\delta^*,0)$. We refer to such solutions as {\em self-similar expanding solutions of the EP$_\frac43$-system.}
\end{definition}




\subsubsection{Linearly expanding solutions}

If we assume that either
\be\label{E:1}
\delta>0,
\ee 
or
\be\label{E:2}
\delta=0, \ \ \lambda_1>0,
\ee
or
\be\label{E:3}
\delta^*<\delta<0, \ \ \lambda_1>\lambda_1^*=\frac{\sqrt{2|\delta|}}{\lambda_0},
\ee 
then a solution $\tilde\lambda $ to~\eqref{E:ODE}--\eqref{E:BClambda} exists globally-in-time and expands indefinitely at a linear rate, i.e. there exists a constant $c>0$ such that 
\be\label{E:EXPANSIONRATE}
\lim_{t\to\infty}\dot{\tilde\lambda }(t)=c.
\ee
The density $\rho$, velocity $v$, and the enthalpy $w$ are given by~\eqref{profile}--\eqref{E:gLE0}. We remark that all of the solutions above have strictly positive energy, i.e.
$
E(\tilde\lambda , \dot{\tilde\lambda })>0
$ 
(see~\eqref{E:ENERGYSPECIAL}). 


\begin{definition}[Linearly expanding homogeneous solutions]\label{D:LINEAREXPANSION}
The family of solutions satisfying~\eqref{E:EXPANSIONRATE} for some $c>0$ with $\tilde\rho= \tilde\lambda (t)^{-3} w^3(\frac r{\tilde\lambda (t)})$, $v(t,r) = \frac{\dot{\tilde\lambda }(t)}{\tilde\lambda (t)} r$ is denoted by
\[
(\tilde\lambda ,\tilde\rho,\tilde v)_{\lambda_0,\lambda_1,\delta}
\]
with parameters $(\lambda_0,\lambda_1,\delta)$ satisfying either~\eqref{E:1}, or~\eqref{E:2}, or~\eqref{E:3}.
We refer to such solutions as {\em linearly expanding solutions of the EP$_\frac43$-system.}
\end{definition}


\subsection{Main results}

The main motivation for this article is the question of nonlinear stability of the above described family of homogeneous expanding solutions of the EP$_\frac43$-system, given by Definitions~\ref{D:SSEXPANSION} and~\ref{D:LINEAREXPANSION}.
The problem of asymptotic stability of this family is a first necessary step in justifying the qualitative behavior~\eqref{E:SSASYMPTOTICS},~\eqref{E:MAKINORATES}  as a credible scenario of stellar expansion. 

Since the family of self-similar solutions $(\bar\lambda,\bar\rho,\bar v)_{\lambda_0,\lambda_1,\delta}$ (subject to the constraint~\eqref{E:INITIALCONSTRAINT0}) embeds in the family of linearly expanding ones, it is obvious that the self-similar behavior is not stable.
Nevertheless, our first result is to show that if we limit our perturbations to a surface of zero-energy perturbations, then the self-similar behavior is indeed nonlinearly stable.

\begin{theorem}[Codimension-one stability of self-similar expanding solutions -- {\em informal statement}]\label{T:SSINFORMAL}
The self-similar expanding solutions of the EP$_\frac43$-system specified in Definition~\ref{D:SSEXPANSION} are globally-in-time nonlinearly stable with respect to spherically symmetric perturbations $(\rho_0,{v}_0)$ with vanishing energy
$
E(\rho_0,{v}_0) = 0.
$
Any such perturbation converges to a nearby self-similar expanding homogeneous solution.
\end{theorem}

A rigorous version of Theorem~\ref{T:SSINFORMAL} is stated in Theorem~\ref{T:MAIN} using the Lagrangian coordinates.


\begin{remark}
The asymptotic attractor $(\bar\lambda,\bar\rho,\bar {v})_{\lambda_0,\lambda_1,\delta}$ is characterized by the requirements
\begin{align}
E(\rho_0,{ v}_0) = E(\bar\rho,\bar {v}) = 0, \ \ M(\rho_0) = M(\bar\rho), \ \ \text{supp}(\rho_0) = [0,\lambda_0].
\end{align}
In particular, our stability result is a codimension-one stability result, as we must restrict our perturbations to the zero-energy surface of admissible stellar configurations. 
It shows that only the directions transversal to the zero-energy surface are responsible for the loss of a self-similar behavior.
\end{remark}


Our second main result concerns the stability of linearly expanding homogeneous solutions. Unlike self-similar solutions, we do not limit our perturbations to the same energy level as the background solution.


\begin{theorem}[Nonlinear stability of the linearly expanding homogeneous solutions -- {\em informal statement}]\label{T:LINEARINFORMAL}
There exists an $\tilde\varepsilon>0$ such that for any $\delta>-\tilde\varepsilon$ the linearly expanding homogeneous solutions  $(\tilde\lambda ,\tilde\rho,\tilde v)$ from Definition~\ref{D:LINEAREXPANSION} are nonlinearly stable with respect to small spherically symmetric perturbations.  
\end{theorem}

A rigorous version of Theorem~\ref{T:LINEARINFORMAL} is stated in Theorem~\ref{T:LINEAREXPANSION} using the Lagrangian coordinates.

\begin{remark}
Theorem~\ref{T:LINEARINFORMAL} states that the behavior in the light blue shaded region in Figure~\ref{F:BIFURCATION} is dynamically stable, as long as $\delta>-\tilde\varepsilon$.  In contrast to Theorem~\ref{T:SSINFORMAL}, Theorem \ref{T:LINEARINFORMAL} does not specify the asymptotic attractor for the perturbation nor is it expected to behave like a nearby member of the family of homogeneous solutions.
Instead, we show that the perturbed solution remains nearby in a suitable sense and that its support grows linearly. 
Question of asymptotic behavior is more subtle, see Remark \ref{rem0}. 
\end{remark}

Theorems~\ref{T:SSINFORMAL} and~\ref{T:LINEARINFORMAL}  are to the best of our knowledge the first results producing a family of global-in-time non-homogeneous classical solutions to the 
EP$_{\frac43}$-system. In particular, we show that the expansion of the background homogeneous solution counteracts the possibility of shock formation and produces a certain damping effect.
This effect is visible only in suitably rescaled similarity variables, and the role of the critical scaling invariance is a fundamental ingredient of our proof.

An important feature of expanding solutions is that the boundary of the domain (fluid support) is moving with the fluid. A simple comparison between perturbed solutions and the asymptotic attractor in Eulerian coordinates is not feasible. The problem is better understood in material coordinates by following the fluid trajectories. We will therefore formulate the problem in suitably 
rescaled Lagrangian coordinates, which will allow us to pull the problem back onto a fixed domain.


\section{Lagrangian formulation}

\subsection{Physical vacuum and local-in-time well-posedness}

One of the fundamental  difficulties associated with the well-posedness theory for~\eqref{E:EULERPOISSON} is tied to the concept of a {\em physical vacuum boundary condition}.
It is clear that the  homogeneous solutions  described above satisfy  $\dot\lambda(t)\neq 0$ and therefore the support of the fluid is  moving. 
On the other hand, the background enthalpy profile $w$ satisfying \eqref{E:gLE0} vanishes at the boundary $z=1$ while it is strictly positive on $[0,1).$ A simple calculation based on~\eqref{E:gLE0} reveals that the strict inequality
$w'(1)<0$ has to hold. This in turn implies that $\rho(t,\lambda(t)) =0$ and $\g_r\rho^{1/3}(t,\lambda(t))<0$. 
Such a boundary behavior is a typical feature of compactly supported stellar configurations: 

\begin{definition}[Physical vacuum condition]\label{D:PHYSICALVACUUM}
We say that the system~\eqref{E:EULERPOISSON} satisfies the physical vacuum boundary condition if 
\be\label{physical_v}
-\infty < \frac{\g c^2}{\g {n}} <0 \quad \text{on } \Gamma(t)
\ee
where 
\[
 c^2={\frac{dp}{d\rho}} = \gamma \rho^{\gamma-1}  \ \ (\text{$c$ is known as the speed of sound} )
\] 
and $\frac{\g}{\g n}$ denotes the outward normal derivative. 
\end{definition}

We note that the physical vacuum condition is analogous to Rayleigh-Taylor sign condition arising in the free boundary problems of the incompressible fluid dynamics (see~\cite{CoSh2007} for a detailed discussion and references therein). 
The problem of moving vacuum boundaries characterized by \eqref{physical_v} has for a long time created various analytical and conceptual difficulties in addressing the question of well-posedness due to the degeneracy caused by 
the physical vacuum requirement. Note that by Definition~\ref{D:PHYSICALVACUUM}, an enthalpy profile $c^2\sim\rho^{\gamma-1}$ satisfying the physical vacuum condition is {\em not} smooth across the moving boundary and it is a priori not even clear 
what is the suitable functional analytic framework to address the question of local-in-time well-posedness.

In 2010 the basic advance was achieved by Coutand \& Shkoller~\cite{CoSh2011,CoSh2012} (see \cite{CLS} for the a priori estimates) and independently by Jang \& Masmoudi~\cite{JaMa2009,JaMa2015} where the authors developed a well-posedness theory in high-order weighted Sobolev spaces 
for the compressible Euler equations satisfying the physical vacuum condition. By a straightforward extension, the same framework can be applied to the free boundary Euler-Poisson system, 
as the effects of the added gravitational field are of lower order from the point-of-view of well-posedness theory~\cite{Jang2014, Jang2015, LXZ}. We refer to \cite{JM1,JM2012} for more detailed discussion on the physical vacuum and other vacuum states.

The crux of the approach in~\cite{CoSh2012, JaMa2015, Jang2014} is the use of Lagrangian coordinates. They allow us to pull-back the  EP$_\frac43$-system onto a fixed compact domain. More importantly, 
in spherical symmetry the Lagrangian formulation of the EP$_\frac43$-system reduces to a degenerate quasilinear wave equation for the Lagrangian flow-map, wherein the initial density profile $\rho_0$
features as a coefficient inside the nonlinear wave operator. To see this  we introduce a Lagrangian flow map $\eta:\Omega_0\to\Omega(t)$ as a solution of: 
\begin{align}
\dot\eta(t,x) &= \mathbf{u}(t,\eta(t,x)), \label{E:FLOWMAP} \\
\eta(0,x) &= \eta_0(x), \label{E:ETAINITIAL}
\end{align}
where $\eta_0:\Omega_0\rightarrow \Omega(0)$ is a diffeomorphism with positive Jacobian determinant.  
Since the problem is radially symmetric, we make the ansatz: 
\be\label{E:ETARADIAL}
\eta(t,x) = \chi(t,z) x, \ \ z=|x|
\ee
which leads to the following evolution equation for $\chi$~\cite{Jang2014} : 
\be\label{E:XIEQUATION0}
 \chi_{tt} + F_{w}[\chi] = 0
\ee
with the initial conditions:
\be\label{E:INITIALCONDITIONS}
\chi(0) = \chi_0, \ \ \dot\chi(0) = \chi_1.
\ee
Here, the nonlinear operator $F_{w}$ is given by 
\be\label{E:NONLINEARITY0}
F_{w}[\chi] : = \frac{\chi^2}{ w^3 z}\g_z\left(w^{4}\left[\chi^2\left(\chi+z\g_z\chi\right)\right]^{-\frac{4}{3}}\right) 
+\frac{1}{\chi^2z^3}\int_0^z4\pi w^{3} s^2\,ds
\ee
where 
\be\label{wrho_0}
w^3:= \rho_0 \chi_0^2(\chi_0+ z\g_z\chi_0). 
\ee
A remarkable feature of the formulation~\eqref{E:XIEQUATION0}--\eqref{E:INITIALCONDITIONS} is that the EP system~\eqref{E:EULERPOISSON} in the new variables takes the form of a \textit{quasi-linear
degenerate wave equation} in a bounded domain.

To formulate a well-posedness result we need  weighted Sobolev spaces.

\begin{definition}[Weighted spaces and energy]\label{Def:H^j}
For any $k\in\mathbb N\cup\{0\}$ and a function $w:[0,1]\to\mathbb R_+$, we define weighted spaces $L^2_{w,k}$ as a completion of the space $C_c^\infty([0,1])$ with respect to the norm $\|\cdot\|_{w,k}$ 
generated by the inner product 
\be\label{E:WEIGHTEDINNERPRODUCT}
(\chi_1,\chi_2)_{w,k} : = \int_0^1 \chi_1\chi_2 w^{3+k}z^4\,dz.
\ee
We will set $\|\cdot\|_{w,0}:=\|\cdot\|_{w}$. We also define the following Hilbert space
\be
{H}^j_w:= \{ \varphi \in L^2_{w,0}: \g_z^k \varphi \in L^2_{w,k} \  \text{ for all } \ 0\leq k\leq j  \}. 
\ee
The function space for initial data is defined by 
\begin{align}\label{E:INITIALDATANORM}
\mathcal H_{w}: = \left\{(\chi_0,\chi_1) \big| \, \|(\chi_0,\chi_1)\|_{\mathcal H_{w}} < \infty\right\},
\end{align}
where
\[
\|(\chi_0,\chi_1)\|_{\mathcal H_{w}}^2 : = \sum_{k=0}^{8}\|\pa_z^k\chi\|_{w,k}^2 + \sum_{k=0}^{7}\|\pa_z^{k+1}\chi_1\|_{w,k+1}^2 . 
\]
We also introduce the following weighted space-time norm: 
\begin{align}\label{E:INITIALDATASPACE}
\mathcal E(t) : = \sum_{j=0}^7\sum_{k=0}^j   \left(\|\g_s^{j-k+1}\g_z^k \chi\|_{w,k}^2 +  \|\g_s^{j-k}\g_z^{k+1} \chi\|_{w,k+1}^2  +\|\g_s^{j-k}\g_z^k \chi\|_{w,k}^2\right).
\end{align} 
\end{definition}

We are now ready to state the following local-in-time well-posedness theorem~\cite{CoSh2012,Jang2014,JaMa2015,LXZ}: 


\begin{theorem}\label{T:LOCAL}
Let the initial density $\rho_0$ satisfy the physical vacuum boundary condition~\eqref{physical_v}. Assume additionally that 
$
(\chi_0,\chi_1)\in \mathcal H_{w}
$. 
Then there exists a constant $C>0$ and a unique solution to the initial value problem~\eqref{E:XIEQUATION0}--\eqref{E:INITIALCONDITIONS} on some finite time interval $[0,T], T>0,$ such that the map 
\[
[0,T]\ni t \to \E(t)
\]
is continuous, and the solution $(\chi,\chi_t)$ satisfies the energy bound
\[
\sup_{0\le t \le T} \E(t) \le C \|(\chi_0,\chi_1)\|_{\mathcal H_{w}}^2.
\]
If $\mathcal T=\sup\{0\le t\le\infty :\,\text{solution exists on $[0,t)$ and }\, \E(t)<\infty\}$ is the maximal time of existence, then $\mathcal T<\infty$ implies
$
\lim_{t\to\mathcal T^-}\E(t) = \infty.
$
\end{theorem}


\begin{remark}\label{R:ALTERNATIVE}
The norm used to state our well-posedness theorem suggests an interesting scaling structure in the problem. As the number of spatial derivatives increases
the weights get more singular as it is clear from~\eqref{E:INITIALDATASPACE}.
An alternative to the norm $\E$ is a norm based on purely spatial derivatives. In fact, the methodology developed in~\cite{JaMa2015} ensures the local-in-time well-posendess of solutions in weighted Sobolev spaces generated by spatial derivatives in which suitable a priori estimates are established. This principle applies to our problem as well: based on the a priori estimates (for instance, see \eqref{E:MAINBOUND2}), we have a local existence theorem analogous to Theorem \ref{T:LOCAL} by using the norm $\tilde \E$ defined in \eqref{E:ENERGYLINEAR} and the corresponding function space introduced in Definition \ref{D:DELTAJSPACES}. 
\end{remark}



The conservation of mass~\eqref{E:MASS} is embedded in the new formulation of the problem, as the continuity equation~\eqref{E:CONTINUITYEP} is used fundamentally in the derivation of~\eqref{E:XIEQUATION0}, see~\cite{Jang2014}.
The only surviving non-trivial conservation law is the energy conservation law. In the new variable $\chi$, the energy~\eqref{E:ENERGY} takes the form:
\begin{align}\label{E:ENERGYLAGRANGIAN}
E(\chi,\partial_t\chi)(t)  =& 2\pi \|\partial_t\chi\|_w^2 +12\pi \| \left( z^6\chi^2(\chi+z\partial_z\chi)\right)^{-1/6}\|_w^2  - 4\pi \Big\| \left( \chi^{-1} z^{-3}\int_0^z 4\pi w^3 s^2 ds \right)^{1/2}\Big\|_w^2. 
\end{align}
It can be checked by a direct computation that along the smooth solutions of~\eqref{E:XIEQUATION0} the energy $E(\chi,\partial_t\chi)(t)$ remains constant. In particular, the solutions obtained by Theorem \ref{T:LOCAL} satisfy $E(\chi,\partial_t\chi)(t)=E(\chi_0,\chi_1)$ for $0\leq t\leq T$. 

\subsection{Homogeneous solutions of the EP$_{\frac43}$-system} \label{S:SPECIALSOLUTIONS}

In this section, we unify the treatment of the existence of self-similar expanding/collapsing solutions and show that the special solutions discovered by Goldreich and Weber~\cite{GoWe}, Makino~\cite{Makino92}, Fu and Lin~\cite{FuLin} 
naturally correspond to {\em homogeneous} solutions expressed in the Lagrangian coordinates introduced above. 
For the more general non-isentropic flows in arbitrary dimensions the existence of such solutions was shown by Deng, Xiang, and Yang in~\cite{DengXiangYang}.

A homogeneous solution of the EP$_{\frac43}$-system possesses {\em by definition} a flow map $\eta$ of the form
\[
\eta(t,x) = \lambda(t) x.
\]
Equivalently, we look for the solutions to~\eqref{E:XIEQUATION} of the form 
$
\chi(t,z) = \lambda(t),
$
thus justifying the use of the word ``homogeneous".
In the physics literature~\cite{GoWe,RB, Yahil1983} one also refers to such solutions as {\em homologous} solutions.
The nonlinear operator $F_w$ is easily checked to satisfy the homogeneity property $F_w[\l] = \lambda^{-2} F_w[1]$ and therefore equation~\eqref{E:XIEQUATION} reduces to
\be
 \lambda^2 \ddot \lambda  + F_w[1]=0. 
\ee 
Since $\l$ is a function of $t$ only and $w$ depends only on $z$, both terms must be constant: 
\be\label{E:LAMBDAODE}
 \lambda^2 \ddot \lambda = \delta
 \ee
 \be\label{D:w}
 F_w[1] =\frac{4 w'}{ z} 
+\frac{1}{ z^3}\int_0^z4\pi w^{3} s^2\,ds= -\delta
\ee
for some $\delta \in \R$. The existence of the steady state solutions of \eqref{E:XIEQUATION} amounts to the existence of $\lambda(t)$ (the radius of the star), $w(z)$ (the star configuration) satisfying \eqref{E:LAMBDAODE}--\eqref{D:w}. In fact, the above two equations are essentially the ones obtained by Goldreich and Weber \cite{GoWe}, Makino \cite{Makino92}, Fu and Lin  \cite{FuLin} starting with the ansatz $r=\lambda(t)z$, $\rho(t,r)=\lambda^{-3}(t)w^3(z)$ and $v(t,r)=\dot\lambda (t)z$ in  Eulerian coordinates. 
Note that a true self-similar ansatz would have the relationship $v(t,r)=\dot\lambda (t)z$ replaced by $v(t,r)=\lambda^{-\frac12}(t)z$ thus honoring
the self-similar rescaling~\eqref{E:SSDENSITY}--\eqref{E:SSVELOCITY} when $\gamma=\frac43$. The solvability of  \eqref{E:LAMBDAODE} and \eqref{D:w} with suitable initial, boundary conditions and the behavior of the solutions are discussed in detail in \cite{Makino92, FuLin}. In what follows, we shall summarize some of these results, and state the corresponding regularity properties necessary for 
the nonlinear analysis.

To describe the solutions to~\eqref{D:w} it is convenient to work with the second order formulation
\be\label{E:LE}
w''+ \frac{2}{z} w' + \pi w^3 =- \frac{3\delta}{4},
\ee
which follows easily from~\eqref{D:w}.
Note that $\delta=0$ corresponds to the classical Lane-Emden equation with index $3$~\cite{Ch}. 
We refer to~\eqref{E:LE} as the {\em generalized Lane-Emden equation}. 
The following result guarantees the existence of solutions with suitable boundary conditions, thereby showing that the value of $\delta$ cannot be arbitrary. 
 
\begin{proposition}[\cite{FuLin}]\label{FuLin} There exists a negative constant $\delta^\ast <0$ such that for any $\delta\geq \delta^\ast$ there exists  
 $w=w(z)$ satisfying \eqref{E:LE} such that $0<w<\infty$ in $[0,1)$ and $w$ satisfies the boundary conditions  
\be\label{BCy}
w'(0)=0, \quad w(1)=0. 
\ee
\end{proposition}

Proposition \ref{FuLin} is essentially contained in~\cite{FuLin}. We point out only one difference. Equation \eqref{E:LE} is customarily solved with the boundary conditions $w(0)=1$, $w'(0)=0$ by a shooting method. The first zero denoted by $\bar z=\bar z(\delta)$ is then the radius of a star configuration. This solution is transformed to our desired solution $w$ in
Proposition \ref{FuLin}  through a similarity transformation. Namely, for any $\beta>0$ we note that $w_\beta(z):= \beta w(\beta z)$ solves \eqref{E:LE} with $\delta$ replaced by $\delta\beta^3$ and it satisfies $w_\beta(0)=\beta,$ $w_\beta'(0)=0$, and $\bar z_\beta= \beta^{-1} \bar z(\delta)$. 

The following lemma concerns the regularity of $w$. 

\begin{lemma}\label{lem:w} Let $w$ be a solution obtained in Proposition \ref{FuLin}. Then $w\in C^\infty(0,1)$  and $w$ is analytic near $z=0$ as well as $z=1.$ Moreover 
\begin{enumerate}
\item[(i)] $w(z)=A_1+ A_2 z^2 +O(z^4)$, $z\sim 0$ for some constants $A_1$ and $A_2$ and $w^{(2k+1)}(0)=0$ for any nonnegative integer $k\geq 0$; 
\item[(ii)] $w$ satisfies the physical vacuum condition, i.e. strict inequalities $-\infty<w'(1)<0$ hold.
\end{enumerate}
\end{lemma}

\begin{proof} It follows from a minor modification of Lemma 3.3 in \cite{Jang2014}. We omit the details. 
\end{proof}

Self-similar solutions of~\eqref{E:LAMBDAODE} are by definition the ones satisfying
\be\label{E:BDEFINITION}
b: = -\frac{\lambda_s}{\lambda} = \text{ \ const.}
\ee
where $\frac{ds}{dt} = \lambda(t)^{-\frac32}.$ This translates in the assumption $\dot\lambda(t) \lambda^{\frac12}(t)=\text{const.}$
Upon specifying the initial conditions
\be\label{BClambda}
\lambda(0)=\lambda_0>0, \quad \dot\lambda(0) =\lambda_1,
\ee
it is straightforward to check that for any $\delta\in(\delta^\ast,0)$
there exists a self-similar solution $\lambda(t)$ of the form:
\be\label{E:SSRATE}
\lambda(t)=\left(\lambda_0^{3/2} + \frac32\lambda_0^{1/2}\lambda_1 t \right)^{2/3}
\ee
satisfying~\eqref{E:LAMBDAODE} and~\eqref{E:BDEFINITION} with the additional initial data constraint
\be\label{E:INITIALCONSTRAINT}
\lambda_1^2 + \frac{2\delta}{\lambda_0}=0.
\ee


%

Of course there exist other non-self-similar solutions of \eqref{E:LAMBDAODE} with more general initial data. The following result discusses the behavior of the solutions of \eqref{E:LAMBDAODE} and \eqref{BClambda}. 

\begin{proposition}[\cite{Makino92,FuLin}]\label{prop1} Let $\lambda(t)$ be the solution of \eqref{E:LAMBDAODE} and \eqref{BClambda}. 
\begin{enumerate}
\item[(1)] If $\delta>0$, then $\lambda(t)>0$ for all $t>0$, $\lim_{t\to\infty}\lambda(t)=\infty$ and moreover, there exist $c_1,c_2>0$ such that 
\be\label{E:LINEARBEHAVIOR}
\lambda(t)\sim_{t\to\infty} c_1(1+c_2 t).
\ee
\item [(2)] If $\delta=0$, then $\lambda(t)=\lambda_0+\lambda_1 t$. 
\item [(3)] If $\delta<0$ let 
\be\label{E:SSCONDITION}
\lambda_1^\ast=\sqrt{\frac{2|\delta| }{\lambda_0}}. 
\ee
\begin{enumerate}
\item
If $\lambda_1= \lambda_1^\ast$, then  $\lambda(t)>0$ for all $t>0$ and it is explicitly given by the formula: 
\be\label{E:SSRATE}
\lambda(t)=\left(\lambda_0^{3/2} + \frac32\lambda_0^{1/2}\lambda_1 t \right)^{2/3}, \ \ t\ge0.
\ee
\item
If $\lambda_1>\lambda_1^*$  then  $\lambda(t)>0$ for all $t>0$, $\lim_{t\to\infty}\lambda(t)=\infty$, and there exist $c_1,c_2>0$ such that $\lambda(t)$ satisfies the asymptotic relation~\eqref{E:LINEARBEHAVIOR}.
\item
If $\lambda_1<\lambda_1^\ast,$ then there is a time $0<T<\infty$ such that $\lambda(t)>0$ in $(0,T)$ and $\lambda(t)\rightarrow 0$ as $t\rightarrow T^-$. Moreover, there exist constants $k_1,k_2>0$ such that 
\be\label{E:COLLAPSERATE}
\lambda(t)\sim_{t\to T^-} k_1(T-k_2t)^{\frac23}.
\ee
\end{enumerate}
\end{enumerate}
\end{proposition}



\begin{proof}
The only statements not discussed in~\cite{Makino92,FuLin} are the precise rates of expansion stated above. Rate~\eqref{E:SSRATE} is easily obtained by explicitly solving~\eqref{E:LAMBDAODE} with 
the initial condition~\eqref{E:SSCONDITION},
whereas rates~\eqref{E:LINEARBEHAVIOR} and~\eqref{E:COLLAPSERATE} require a little more work, but follow from classical ODE arguments.
\end{proof}

Propositions~\ref{FuLin} and~\ref{prop1}  in particular imply that for any given $(\delta,\lambda_0,\lambda_1)\in [\delta^\ast,\infty)\times (0,\infty)\times (-\infty,\infty)$ satisfying~\eqref{E:INITIALCONSTRAINT},  $\chi\equiv \lambda(t)$ is a 
solution of~ \eqref{E:XIEQUATION0} with $\l$ 
given by~\eqref{E:SSRATE} and $w$ a solution of~\eqref{D:w} and~\eqref{BCy}. 
In Eulerian coordinates, these solutions are given by
\be\label{profile1}
\rho(t,r) = \lambda(t)^{-3} w^3(\frac r{\lambda(t)}), \ \ v(t,r) = \frac{\dot\lambda(t)}{\lambda(t)} r. 
\ee

We next examine the physical energy and total mass for the homogeneous solutions given by \eqref{profile1}. 

\begin{lemma}\label{lem:E} The energy $E$ of the homogeneous solutions is given by 
\be
E= \left(\lambda_1^2+\frac{2\delta}{\lambda_0} \right) \int_0^{1}2\pi w^3z^4 \,dz . 
\ee
\end{lemma}

\begin{proof}
Recall the total energy: 
\be
E= \int_0^{\lambda(t)}\left( \frac{1}{2}\rho v^2 + 3 \rho^{\frac43} \right)4\pi  r^2 \, dr -  \int_0^{\lambda(t)}  4\pi \rho r \left(\int_0^r 4\pi \rho s^2 ds\right) \, dr. 
\ee 
We first compute the mass in the ball of radius $r$: 
\[
\begin{split}
\int_0^r 4\pi \rho s^2 ds& = 4\pi \int_0^r \frac{1}{\lambda^3} w^3(\frac{s}{\lambda}) s^2 \, ds =
4\pi \int_0^{\frac{r}{\lambda}} w^3(z) z^2 \, dz\\
 &=4 \left(- (\frac{r}{\lambda})^2 w'( \frac{r}{\lambda})- (\frac{r}{\lambda})^3 \frac{\delta}{4} \right)
\end{split}
\]
and hence the potential energy can be written as 
\[
\begin{split}
\int_0^{\lambda(t)}  4\pi \rho r \left(\int_0^r 4\pi\rho s^2 ds\right) \,dr &= \int_0^{\lambda} \frac{16\pi }{\lambda^3} w^3  r  \left(- (\frac{r}{\lambda})^2 w'( \frac{r}{\lambda})- (\frac{r}{\lambda})^3 \frac{\delta}{4} \right)  \,dr \\
&=\frac{12\pi}{\lambda^4}\int_0^{\lambda } w^4 r^2 \,dr -\frac{4\pi\delta}{\lambda^6}\int_0^{\lambda } w^3 r^4 \,dr . 
\end{split}
\]
Therefore, the total energy is written as 
\[
\begin{split}
E&=  \frac{2\pi \dot\lambda ^2}{ \lambda^5}  \int_0^{\lambda}w^3 r^4 \,dr +\frac{12\pi}{ \lambda^4} \int_0^{\lambda}  w^4 r^2  \,dr 
-\left( \frac{12\pi}{\lambda^4}\int_0^{\lambda } w^4 r^2 \,dr -\frac{4\pi\delta}{\lambda^6}\int_0^{\lambda } w^3 r^4 \,dr\right).\\
\end{split}
\]
Now we use the equation for $\l$ to obtain the total energy 
\[
E= \frac{\lambda_1^2+\frac{2\delta}{\lambda_0} }{ 2\lambda^5}  \int_0^{\lambda}4\pi w^3r^4 \,dr =
 \frac{\lambda_1^2+\frac{2\delta}{\lambda_0} }{2 }  \int_0^{1}4\pi w^3z^4 dz  . 
\]
\end{proof}

Lemma \ref{lem:E} and Proposition \ref{prop1} reveal an interesting relationship between the energy of self-similar expansions and the energy of non self-similar expansions.
The self-similar stellar expansion is associated with zero-energy stellar configurations, i.e.
\be\label{E:ZEROENERGY}
E(\bar\lambda(t),\dot{\bar\lambda}(t)) =0.
\ee
For all other linearly expanding homogeneous solutions, the associated physical energy is strictly positive.  On the other hand, the energy of the collapsing solutions can be positive, negative, or zero.

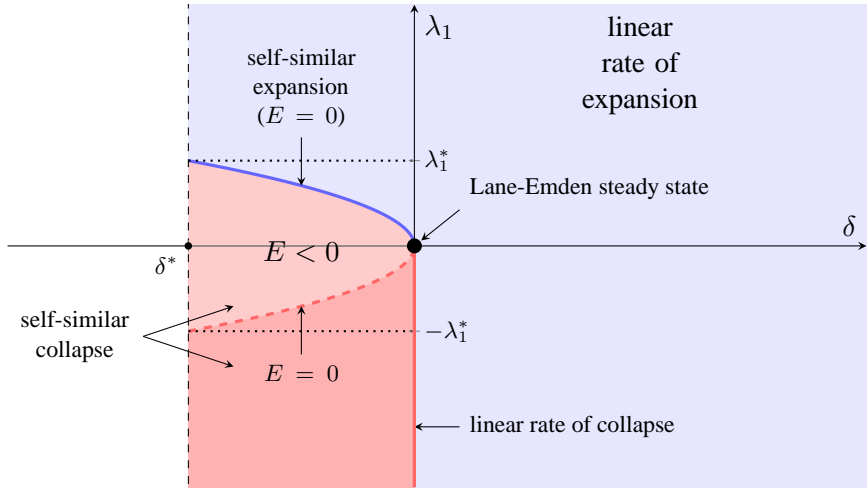
\begin{figure}[htb]
\centering
\begin{tikzpicture}[>=stealth]
\begin{axis}[width=13cm, height=8cm,
  xmin=-1.8, xmax=2, ymin=-4, ymax=4,
  axis lines=center, axis on top=true, xlabel=$\delta$, ylabel=$\lambda_1$,
  xtick={-1}, xticklabels={},
  ytick={-1.414,1.414}, yticklabels={},
  after end axis/.code={
    \draw[black!30] (axis cs: 0, 0) -- (axis cs: -1, 0);
    \draw[very thick, red!60] (axis cs: 0, 0) -- (axis cs: 0, -4);
    \node[circle, fill=black, inner sep=2pt] at (axis cs: 0, 0) {};
    \node[circle, fill=black, inner sep=1pt] at (axis cs: -1, 0) {};
    \node at (axis cs: -0.5, -0.1) {$E < 0$};
  }]

\draw[fill=blue!10, draw=none] (axis cs: -1, -5) rectangle (axis cs: 2, 5);
\draw[fill=red!30, draw=none] (axis cs: -1, 0) rectangle (axis cs: 0, -5);
\addplot [domain=-1:0, samples=1000, fill=red!20, draw=none] {sqrt(-2*x)} \closedcycle;
\addplot [domain=-1:0, samples=1000, fill=red!20, draw=none] {-sqrt(-2*x)} \closedcycle;

\addplot [domain=-1:0, samples=1000, very thick, draw=blue!60] {sqrt(-2*x)};
\addplot [domain=-1:0, samples=1000, very thick, draw=red!60, dashed] {-sqrt(-2*x)};

\draw[dashed] (axis cs: -1, -5) -- (axis cs: -1, 5);
\draw[dotted, thick] (axis cs: -1, 1.414) -- (axis cs: 0, 1.414);
\draw[dotted, thick] (axis cs: -1, -1.414) -- (axis cs: 0, -1.414);

\node[anchor=west, font=\footnotesize] at (axis cs: 0, 1.414) {$\lambda_1^*$};
\node[anchor=west, font=\footnotesize] at (axis cs: 0, -1.414) {$-\lambda_1^*$};
\node[anchor=north east, font=\footnotesize] at (axis cs: -1, 0) {$\delta^*$};

\draw[<-] (axis cs: 0, -3) -- (axis cs: 0.2, -3) node [anchor=west, font=\footnotesize] {linear rate of collapse};
\draw[<-] (axis cs: 0.03, 0.09) -- (axis cs: 0.2, 0.6) node [anchor=south west, font=\footnotesize] {Lane-Emden steady state};

\coordinate (P) at (axis cs: -1.2, -1.5);
\node [anchor=east, text width=1.8cm, align=center, inner sep=0pt, font=\footnotesize] at (P) {self-similar collapse};
\draw[->] (P) -- (axis cs: -0.8, -1);
\draw[->] (P) -- (axis cs: -0.8, -2);

\draw[<-] (axis cs: -0.5, 1) -- (axis cs: -0.5, 1.8) node [anchor=south, text width=2cm, align=center, font=\footnotesize] {self-similar expansion ($E=0$)};
\draw[<-] (axis cs: -0.5, -1) -- (axis cs: -0.5, -1.8) node [anchor=north, text width=2cm, align=center, font=\footnotesize] {$E=0$};
\node [text width=2cm, align=center] at (axis cs: 1, 3) {linear rate of expansion};
\end{axis}
\end{tikzpicture}

\caption{Bifurcation diagram, $\lambda_0=1$ \label{F:BIFURCATION}}
\end{figure}



A simple computation based on the generalized Lane-Emden equation~\eqref{E:LE} shows that the total mass of a homogeneous solution is given by 
\be\label{E:MASSHOMO}
M=M(\delta)=\int_0^{\lambda (t)} 4\pi \rho r^2 \, dr = -4 w'(1) -\delta 
\ee
which obviously depends only on $\delta$. Since $\delta\in[\delta_\ast,\infty)$ and $-\infty<w'(1)<0$, $M(\delta)$ cannot assume arbitrary positive values. Theorem 3.3 in \cite{FuLin} asserts that for any $\delta\in (\delta_\ast, \infty)$ 
\[
\frac{d M(\delta)}{d\delta}<0, 
\]
which implies that the admissible total mass lies in $(0,M(\delta_\ast)]$. This fact supports the following rather appealing interpretation of Propositions~\ref{prop1} and~\ref{FuLin}: 
if $M< M(0)$, then the star will expand and the density will eventually go to zero. If $M\geq M(0)$ and if the initial velocity is below the escape velocity, the star will collapse toward the center in  finite time. 
We note that $M(0)$ is the mass of the steady Lane-Emden configuration.
The notion of  critical mass separating homogeneous collapse from expansion is discussed for the non-isentropic Euler-Poisson in general dimensions  in~\cite{DengXiangYang}.



\begin{remark}
It is obvious from Figure~\ref{F:BIFURCATION} that the Lane-Emden steady state ($\gamma=\frac43$) at the origin of the graph is dynamically unstable.
This is clear from~\cite{GoWe,Makino92,FuLin,DengXiangYang} and was specifically pointed out by Rein~\cite{Rein}, who remarked that  positive energy configuration can lead to an indefinite expansion of the fluid support \cite{DLYY}. 
Note that Figure~\ref{F:BIFURCATION} shows that we can perturb the steady state even with negative energy configuration leading to an eventual collapse of the star, or with positive energy configurations that 
also lead to a collapse! The phase space around the steady state is intricate and the nature of instability (expansion vs. collapse) cannot be decided based solely on the
sign of the energy of the perturbation. 
\end{remark}

\subsection{Nonlinear stability in similarity variables}\label{S:SSGLOBAL}

Our goal is to study the nonlinear stability of the expanding star families $(\bar\lambda,\bar\rho,\bar v)$ and $(\tilde\lambda ,\tilde\rho,\tilde v)$ given by Definitions~\ref{D:SSEXPANSION} and~\ref{D:LINEAREXPANSION} respectively.
The general strategy will be to pass to suitable similarity coordinates,
where the time-dependent expanding homogeneous solution transforms into a steady state of the new system of equations. Because of different rates of expansion, our stability results are presented in two separate subsections. We start with self-similar expansion.

\subsubsection{Nonlinear stability for self-similarly expanding homogeneous solutions}

Let $(\bar\lambda_*,\bar\rho_*,\bar v_*)_{\lambda_0^*,\lambda_1^*,\delta^*}$ be a given self-similar expanding homogeneous solution of the EP$_{\frac43}$-system given by Definition~\ref{D:LINEAREXPANSION}.
We consider a perturbation of $(\bar\lambda_*,\bar\rho_*,\bar v_*)_{\lambda_0^*,\lambda_1^*,\delta^*}$ 
\[
\rho(0,r)=\rho_0(r), \ \ {v}(0,r) = {v}_0(r).
\]
Without loss of generality we shall assume that $\rho_0:[0,1]\to\mathbb R$ satisfies
\be\label{E:DENSITYINITIAL}
\rho_0(r)>0, \ \ r\in[0,1), \ \ \rho_0(1)=0, \ \ \overline{\text{supp}(\rho_0)}=[0,1].
\ee
We further assume that 
\be\label{E:MASSENERGYSS}
\left|M(\rho_0) - M(\bar\rho_*)\right| \ll 1, \ \ E(\rho_0,v_0) = E(\bar\rho^*,\bar v^*)=0.
\ee
The second condition means that we will restrict our stability analysis to a surface of zero-energy initial data.

In order to describe the asymptotic behavior of the solutions generated by $(\rho_0,v_0)$ 
we will use the conservation-of-mass law to identify
the background expanding homogeneous self-similar solution that we expect the perturbed solution to converge to.
We accomplish this by finding a $\delta$ such that 
\be
M(\delta) = M(\rho_0) = \int_{\mathbb R^3} \rho_0(x)\,dx.
\ee
In particular, this uniquely fixes a background solution
$
(\bar\lambda,\bar\rho,\bar{v})_{\lambda_0,\lambda_1,\delta}
$
with 
\[
\lambda_0 = 1, \ \ e(\bar\lambda,\bar\lambda_t):= \lambda_1^2 + \frac{2\delta}{\lambda_0} =0
\]
where $e$ is the coefficient appearing in $E$ \eqref{E:ENERGYSPECIAL} and we call it the effective energy.

We now recall  the Lagrangian formulation of the problem~\eqref{E:FLOWMAP}--\eqref{E:ETARADIAL} 
and assume that there exists a $\chi_0:[0,1]\to[0,1]$ such that
\begin{align}\label{E:XI0CHOICE}
\det D\eta_0 = \chi_0^2(\chi_0+z\pa_z\chi_0) = \frac{\rho_0\circ\chi_0}{\bar\rho(0,\cdot)}.
\end{align} 
For initial density $\rho_0$ exhibiting the same boundary behavior as $\bar\rho$ so that $\rho_0/\bar\rho$ is a smooth positive function, the existence of such a function $\chi_0$ follows from the Dacorogna-Moser theorem~\cite{DM}.
Since the energy is a dynamically conserved quantity, as long as smooth solutions to~\eqref{E:XIEQUATION0}--\eqref{E:INITIALCONDITIONS} exist, we have the identity
\begin{align*}
E(\chi,\chi_t)(t) = E(\chi_0,\chi_1). 
\end{align*}
By the second assumption in~\eqref{E:MASSENERGYSS} we note that
\be\label{E:ENERGYISZERO}
E(\chi_0,\chi_1) =0.
\ee

To understand the nonlinear asymptotics, we
introduce a new unknown $\xi$ by setting 
\[
\xi(t,z) : = \frac{\chi(t,z)}{\bar\lambda(t)}.
\]
As a consequence of~\eqref{E:XIEQUATION0} we obtain an evolution equation for $\xi$ : 
\be\label{E:XIEQUATION}
\bar\lambda^2(\bar\lambda \xi)_{tt} + F_{w_\delta}[\xi] = 0,
\ee
equipped with the initial conditions
\be\label{E:XIINITIAL}
\xi(0,z) = \chi_0(z), \ \ \xi_t(0,z) = \xi_1(z), \ \ z\in[0,1],
\ee
where 
$
\xi_1 = \chi_1 - \chi_0\bar\lambda_1
$
and  the nonlinear operator $F_{w_\delta}$ is given by~\eqref{E:NONLINEARITY0}.
With respect to the new variable $\xi$ the energy $E$ takes the form:
\begin{align}\label{E:ENERGYXI0}
E(\xi,\xi_t) & = 2\pi \|\xi_t \bar\lambda + \xi\bar\lambda_t\|_\delta^2 + \frac{4\pi\delta}{\bar\lambda} \|\xi^{-1/2}\|_\delta^2 \notag \\
& \ \ \ \ +\frac{12\pi}{\bar\lambda} \left(\int_0^1 (\xi^2(\xi+z\xi_z))^{-1/3}w_\delta^4 z^2\,dz + \frac13 \int_0^1 \xi^{-1}\pa_z(w_\delta^4) z^3\,dz\right).
\end{align}

Motivated by the self-similar rescaling~\eqref{E:SSDENSITY}--\eqref{E:SSVELOCITY} we introduce a self-similar time coordinate $s$ by setting
\be\label{E:SDEFINITION}
s =s(t) =  \int_0^t \frac{1}{\bar\lambda(\tau)^{\frac32}} \, d\tau, \ \ 
\ee
and the unknown $\psi$ in new coordinates $(s,z)$: 
\be\label{E:PSIDEFINITION}
 \psi(s,z): = \xi(t,z).
\ee

Using~\eqref{E:XIEQUATION},~\eqref{E:SDEFINITION}, and~\eqref{E:PSIDEFINITION} a straightforward calculation gives a
quasilinear wave equation for $\psi$
\be\label{E:PSIEQUATION}
\psi_{ss} - \frac12 b\psi_s + \delta \psi + F_{w_\delta}[\psi] = 0,
\ee
where 
\begin{align}
b = -\frac{\bar\lambda_s}{\bar\lambda} = -\sqrt{2|\delta|}<0.
\end{align} 
We supply~\eqref{E:PSIEQUATION} with initial conditions:
\be\label{E:PSIINITIALCONDITIONS}
\psi(0) = \psi_0=\xi_0, \ \ \psi_s(0) = \psi_1 = \xi_1,
\ee
where we have used $\bar\lambda_0=1$. 
Observe that conditions~\eqref{E:PSIINITIALCONDITIONS} are naturally inherited from~\eqref{E:XIINITIAL}.
Equation~\eqref{E:PSIEQUATION} admits a steady state
\[
\bar\psi\equiv 1
\]
which is a self-similar realization of the  background expanding solution $(\bar\lambda,\bar\rho,\bar{v})_{\lambda_0,\lambda_1,\delta}$
given by~\eqref{E:SSRATE}. We note that $\bar\lambda$ in the new time coordinate $s$ takes the following form
\be\label{E:BARLAMBDA}
\bar\lambda(s)= e^{-bs},
\ee
and thus grows exponentially in $s$ (recall that $b<0$ corresponds to the expansion case).
The energy~\eqref{E:ENERGYXI0} in the new variables takes the form
\begin{align}\label{E:ENERGYXI1}
E(\psi,\psi_s) & = \frac{2\pi}{\bar\lambda} \|\psi_s + \frac{\bar\lambda_s}{\bar\lambda}\psi\|_\delta^2 + \frac{4\pi\delta}{\bar\lambda} \|\psi^{-1/2}\|_\delta^2 \notag \\
& \ \ \ \ +\frac{12\pi}{\bar\lambda} \left(\int_0^1 (\psi^2(\psi+z\psi_z))^{-1/3}w_\delta^4 z^2\,dz + \frac13 \int_0^1 \psi^{-1}\pa_z(w_\delta^4) z^3\,dz\right).
\end{align}

To stress the dependence of the background profile $w_\delta$ on the parameter $\delta$ we shall from now on use the notation
\[
(\cdot, \cdot)_{\delta,k}:=(\cdot, \cdot)_{w_\delta,k}, \ \ (\cdot, \cdot)_{\delta}:=(\cdot, \cdot)_{w_\delta,0}, \ \  
\]
\[
\|\cdot\|_{\delta,k} : = \|\cdot\|_{w_\delta,k}, \ \ \|\cdot\|_\delta: = \|\cdot\|_{w_\delta}, \ \ \mathcal H_\delta: = \mathcal H_{w_\delta},
\]
where the inner product $(\cdot, \cdot)_{\delta,k}$ and the norms $\|\cdot\|_{w,k}$ have been defined in~\eqref{E:WEIGHTEDINNERPRODUCT} and the high-order weighted space $\mathcal H_{w}$
is defined by~\eqref{E:INITIALDATASPACE}. 
For any perturbation $\phi=\psi-1$ from the steady state we define a high-order energy norm
\be\label{energy1}
\Energy(\phi,\phi_s)=\Energy(s) := \sum_{j=0}^7\sum_{k=0}^j   \|\g_s^{j-k+1}\g_z^k \phi\|_{\delta,k}^2 +  \|\g_s^{j-k}\g_z^{k+1} \phi\|_{\delta,k+1}^2  +\|\g_s^{j-k}\g_z^k \phi\|_{\delta,k}^2.
\ee


\begin{theorem}[Nonlinear stability of selfsimilar expanding stars]\label{T:MAIN}
There exist $\varepsilon,\epsilon>0$ such that for any $-\varepsilon<\delta<0$ and any initial datum $(\psi(0),\psi_s(0))=(\psi_0,\psi_1)$ with vanishing
initial energy
\begin{align}
E(\psi_0,\psi_1)  = 0,
\end{align}
satisfying the smallness condition
\begin{align}
\E(\psi_0-1,\psi_1)\le \epsilon
\end{align}
there exists a unique global solution to the initial value problem~\eqref{E:PSIEQUATION}--\eqref{E:PSIINITIALCONDITIONS}. Moreover, there
exists a constant $c=c(\delta)$ such that 
\begin{align}
\Energy(s)  \lesssim  \epsilon e^{-c s}, \ \ s\ge0.
\end{align}
\end{theorem}

\begin{remark}[Initial data and compatibility conditions]
Since our energy norm contains higher $s$-derivatives of $\psi$ at time $s=0$, we use the equation~\eqref{E:PSIEQUATION} to define $\g_s^j\psi\Big|_{s=0}$, $j\ge2$. 
Note that we do not have to impose any compatibility conditions at the vacuum boundary $z=1$.
\end{remark}

\begin{remark} \label{R:MAIN}
Since Theorem~\ref{T:MAIN} shows that $\eta(t,r)=\chi(t,r)r$ remains close to $\bar\eta(t,r)=\bar\lambda(t)r$ for all $t>0$, from the formula
\[
\rho(t,\eta(t,r)) = \bar\rho(0,r)\big[\det D\eta_0\big]^{-1}
\]
we may conclude that the Eulerian density $\rho(t,r)$ of the perturbed solution remains ``close" to the self-similarly rescaled Eulerian density
$
\bar\lambda^{-3}w_\delta^3(\frac{r}{\bar\lambda})
$
of the background Eulerian density.
A similar comparison holds between the Eulerian velocity $v$ of the perturbation and the rescaled background velocity $\bar\lambda^{-\frac12} \bar{ v}(\frac{r}{\bar\lambda})$.
\end{remark}




\subsubsection{Nonlinear stability of linearly expanding homogeneous solutions}\label{S:STABILITYLINEAR}


Let $(\tilde\lambda _*,\tilde\rho_*,\tilde v_*)_{\lambda_0^*,\lambda_1^*,\delta^*}$ be a given linearly expanding homogeneous solution of the EP$_{\frac43}$-system given by Definition~\ref{D:LINEAREXPANSION}.
We consider a perturbation of $(\tilde\lambda _*,\tilde\rho_*,\tilde v_*)_{\lambda_0^*,\lambda_1^*,\delta^*}$ 
\[
\rho(0,r)=\rho_0(r), \ \ {v}(0,r) = {v}_0(r).
\]
Without loss of generality we shall assume that $\tilde\lambda _0^*=1$ and that $\rho_0:[0,1]\to\mathbb R$ satisfies~\eqref{E:DENSITYINITIAL}.
We further assume that 
\[
\left|M(\rho_0) - M(\tilde\rho_*)\right| \ll 1, \ \ \left|E(\rho_0,v_0) - E(\tilde\rho^*,\tilde v^*)\right|\ll1.
\]
We may also assume that $E(\rho_0,v_0)>0$ since by definition $E(\tilde\rho^*,\tilde v^*)>0$. 
In order to describe the asymptotic behavior of the solutions generated by $(\rho_0,v_0)$, just like in previous section, 
we will use the conservation-of-mass and conservation-of-energy laws to uniquely fix a linearly expanding background solution
$
(\tilde\lambda ,\tilde\rho,\tilde{v})_{\lambda_0,\lambda_1,\delta}
$
satisfying
\[
\lambda_0 = 1, \ \ M(\tilde\rho)=M(\rho_0), \ \ E(\tilde\rho,\tilde v)= E(\rho_0,v_0). 
\]

Recalling the Lagrangian formulation of the problem~\eqref{E:FLOWMAP}--\eqref{E:ETARADIAL} 
we assume that there exists a $\chi_0:[0,1]\to[0,1]$ such that~\eqref{E:XI0CHOICE} holds with $\tilde\rho$ instead of $\bar\rho$.
To understand the nonlinear asymptotics, we
introduce a new unknown $\zeta$ by setting 
\[
\zeta(t,z) : = \frac{\chi(t,z)}{\tilde\lambda (t)}.
\]
The unknown $\zeta$ solves~\eqref{E:XIEQUATION} with $\zeta$ instead of $\xi$ and $\tilde\lambda $ instead of $\bar\lambda$, 
while the associated initial conditions read
\be\label{E:ZETAINITIAL}
\zeta(0,z) = \chi_0(z), \ \ \zeta_t(0,z) = \zeta_1(z), \ \ z\in[0,1],
\ee
where 
$
\zeta_1 = \chi_1 - \chi_0\lambda_1.
$
To reflect the expected linear growth rate of the perturbed solution we introduce a
new time variable $\tau=\tau(t)$:
\be\label{E:TAUDEFINITION}
\tau =\tau(t)= \int_0^t \frac{1}{\tilde\lambda (\sigma)}\,d\sigma.
\ee
Analogously to~\eqref{E:PSIDEFINITION} we introduce an unknown $\theta$ in new coordinates $(\tau,z)$: 
\be\label{E:THETADEFINITION}
 \theta(\tau,z): = \zeta(t,z).
\ee

Using~\eqref{E:XIEQUATION} and \eqref{E:TAUDEFINITION} we 
derive an equation for the unknown $\theta(s,z)=\zeta(t,z)$
\be\label{E:THETAEQUATION}
\tilde\lambda \theta_{\tau\tau} +\tilde\lambda _\tau \theta_\tau + \delta \theta + F_{w_\delta}[\theta] = 0,
\ee
with the initial data:
\be\label{E:LININITIAL}
\theta(0) = \theta_0 =\zeta_0, \ \ \theta_\tau(0) = \theta_1=\zeta_1.
\ee
In the $\tau$-coordinate, $\tilde\lambda $ satisfies the following ODE:
\[
\tilde\lambda _\tau = \tilde\lambda \sqrt{\tilde e-\frac{2\delta}{\tilde\lambda }},
\] 
where $\tilde e = \lambda_1^2+\frac{2\delta}{\lambda_0}$ is the effective energy of the linearly expanding homogeneous solution $(\tilde\rho,\tilde u, \tilde \l)$. 
It follows that for large values of  $\tau$ the asymptotic behavior of $\tilde\lambda $ is given by 
\be\label{E:LARGES}
\tilde\lambda (\tau)\sim_{\tau\to\infty} e^{\tilde\beta \tau},
\ee
where $\tilde\beta:=\sqrt{\tilde e}$. 
If $\delta\ge0$ it follows that 
\begin{align}\label{E:DELTA1}
\tilde\beta\ge\big|\frac{\tilde\lambda _\tau}{\tilde\lambda }\big| \ge \beta'
\end{align}
where 
\[
\beta': = \sqrt{\tilde e - 2\frac{\delta}{\lambda_0}}.
\]
Similarly, 
\be\label{E:DELTA2}
\beta'\ge\big|\frac{\tilde\lambda _\tau}{\tilde\lambda }\big| \ge \tilde\beta
\ee
if $\delta<0$. Let $\beta_1=\max\{\tilde\beta,\beta'\}$ and $\beta_2=\min\{\tilde\beta,\beta'\}$.
Then for all linearly expanding homogeneous solutions 
\begin{align}\label{E:TILDELAMBDABOUND}
e^{\beta_2\tau}\lesssim |\tilde\lambda (\tau)|+|\tilde\lambda _\tau(\tau)| \lesssim e^{\beta_1\tau}, \ \ \tau\ge0.
\end{align}

Equation~\eqref{E:THETAEQUATION} admits a steady state
\[
\tilde\theta\equiv 1
\]
which corresponds to the  background expanding solution $(\tilde\lambda ,\tilde\rho,\tilde{v})_{\lambda_0,\lambda_1,\delta}$ encompassed by Definition~\ref{D:LINEAREXPANSION} (i.e. belonging
to the light-blue region in Figure~\ref{F:BIFURCATION}.)
The total energy in the new variable takes the form
\begin{align}\label{E:ENERGYTHETA}
E(\theta,\theta_\tau) & = 2\pi \|\theta_\tau + \frac{\tilde\lambda _\tau}{\tilde\lambda }\theta\|_\delta^2 + \frac{4\pi\delta}{\tilde\lambda } \|\theta^{-1/2}\|_\delta^2 \notag \\
& \ \ \ \ +\frac{12\pi}{\tilde\lambda } \left(\int_0^1 (\theta^2(\theta+z\theta_z))^{-1/3}w_\delta^4 z^2\,dz + \frac13 \int_0^1 \theta^{-1}\pa_z(w_\delta^4) z^3\,dz\right).
\end{align}


A key differential operator in our analysis is given by
\be\label{E:SOPERATOR}
\mathcal{S}:= \g_z^2 +\frac{4}{z}\g_z.
\ee

It enters crucially in the definition of our high-order energy space. 
\begin{definition}[The space $\mathfrak{H}_{\delta,k}^j$]\label{D:DELTAJSPACES}
For any $j\in\mathbb N$ and $k\in \mathbb N$, we say that a function $\varphi$ belongs to the space $\mathfrak{H}_{\delta,k}^j$ if
$\S^\ell \varphi \in L^2_{\delta,2k+2\ell}([0,1])$ for all $0\le \ell \le j$. 
\end{definition}

For any perturbation $\phi=\theta-1$ from the steady state we define a high-order energy norm
\begin{align}\label{E:ENERGYLINEAR}
\tilde\E(\phi,\phi_\tau)= \tilde\E(\tau)=\sum_{j=0}^4\left(\tilde\lambda  \|\S^j\phi_\tau\|_{\delta,2j}^2 + \|\S^j\phi\|_{\delta,2j}^2 +\|\g_z\mathcal{S}^j\phi\|_{\delta,2j+1}^2 \right)
\end{align}


\begin{theorem}[Nonlinear stability of linearly expanding stars]\label{T:LINEAREXPANSION}
There exists $\epsilon, \varepsilon_0>0$ such that for any $-\epsilon<\delta<\infty$, any $0<\varepsilon\le\varepsilon_0$ and any initial data $(\theta_0,\theta_1)$ satisfying 
\[
\tilde\E(\theta_0-1,\theta_1) \le \varepsilon
\]
there exists a unique global-in-time solution to the initial value problem~\eqref{E:PHIEQUATION2}--\eqref{E:PHIINITIAL2}. Moreover, there exists a constant $C>0$ such that
\[
\sup_{\tau\ge0}\tilde\E(\tau)\le C \varepsilon
\]
and a $\tau$-independent function $\theta_\infty$ satisfying $\tilde{\mathcal E}(\theta_\infty,0)\le C\varepsilon$
such that 
\begin{align}\label{E:ESTIMATESFINAL}
\lim_{\tau\to\infty}\|\theta(\tau,\cdot) - \theta_\infty(\cdot)\|_{\mathfrak{H}_{\delta,0}^4}=0, \ \ \|\theta_\tau(\tau,\cdot)\|_{\mathfrak{H}_{\delta,0}^4} \lesssim \varepsilon {\tilde\lambda  (\tau)}^{-\frac12} \, . 
\end{align}
\end{theorem}


\begin{remark}\label{rem0}
The asymptotic attractor $\theta_\infty$ is not necessarily a steady state solution of~\eqref{E:THETAEQUATION}. 
Unlike the self-similar case (Theorem~\ref{T:MAIN}) we do not claim that the perturbed solution converges to a 
nearby member of the family of linearly expanding homogeneous solutions. 
 In fact we should not expect to get a decay estimate for the unknown $\phi$ and its spatial derivatives. A much simpler toy model ODE 
$$e^\tau \partial_{\tau\tau}x + e^\tau \partial_\tau x + x = 0$$ has a continuum of possible asymptotic states, although $x\equiv0$ is the only steady state. Due to the fast growing factors in front of $\partial_{\tau\tau}x$ and $\partial_{\tau}x$, the leading order decay of $\partial_{\tau}x$ is in general given by $\partial_{\tau}x= c \tau e^{-\tau}+O(e^{-\tau})$, where $c$ is a constant depending on the initial conditions.  
It can be checked that a nonzero $c$ leads to a nonzero asymptotic state for $x$. 
\end{remark}


\begin{remark} 
An important consequence of Theorem~\ref{T:LINEAREXPANSION} and the precise understanding of the asymptotic-in-time behavior of the Lagrangian map $\eta(t,r)=\chi(t,r)r=\tilde\lambda (t)\theta(t,r)r$
is that we can precisely track the growth of the support of the Eulerian density $\rho$ of the perturbation.
If $\lambda(t)$ denotes the radius of the support of the perturbed solution, then Theorem~\ref{T:LINEAREXPANSION} implies that there exists a constant $C\ge1$ such that 
\[
C^{-1}\tilde\lambda (t) \le \lambda(t) \le C\tilde\lambda (t), \ \ t\ge0,
\] 
where $\tilde\lambda (t)$ is the radius of support of the background Eulerian density $\tilde\rho$.
Just like in Remark~\ref{R:MAIN} we can also show that the Eulerian density and velocity remain ``close" to the suitably rescaled Eulerian density and velocity of the background
homogeneous solution $(\tilde\lambda ,\tilde\rho,\tilde{v})$.
\end{remark}



\subsection{Comments and methodology}\label{S:COMMENTS}

Questions of singularity formation in various critical problems have received a lot of attention in the past decades, especially in the context of nonlinear wave, Yang-Mills, nonlinear Schr\"odinger equations, generalized
Korteweg- deVries equation, corotational wave maps and various other hyperbolic problems. An exhaustive overview with many references can be found in~\cite{MaMeRa}. The possible blow-up phenomenology is rich, with rates often deviating from
the predicted self-similar rates, and such non self-similar regimes can be both stable and unstable. In the context of the gravitational Euler-Poisson problem, the catalogue of potential dynamical scenarios appears to be very rich and far from being understood.

The free-boundary gravitational Euler-Poisson system is a classical example of a system of hyperbolic balance laws where  we may expect shock singularities to occur in finite time. 
A global-in-time existence and uniqueness theory is not available even in the context of small initial data. The manifestly hyperbolic nature of the equations
is clearly seen in Lagrangian coordinates where the EP$_{\gamma}$-system takes a form of a degenerate quasilinear wave equation on a compact spatial domain. 
There are however two well-known classes of exact solutions: the Lane-Emden steady states when $\frac65<\gamma<2$ and a special class of homogeneous (also known as homologous) 
solutions in the case $\gamma=\frac43$ discovered and analyzed in~\cite{GoWe,Makino92,FuLin} (see Section~\ref{S:SPECIALSOLUTIONS}). Both classes of solutions share one important common feature - they  satisfy the so-called physical vacuum boundary condition at the star-vacuum interface.

In the mass-critical case $\gamma=\frac43$ Theorems~\ref{T:MAIN} and~\ref{T:LINEAREXPANSION} confirm that the two modes of expansion (marked by light-blue and dark-blue in Figure~\ref{F:BIFURCATION}) are nonlinearly dynamically stable.
It remains open to understand the asymptotic behavior in the vicinity of collapsing solutions (marked light- and dark-red in Figure~\ref{F:BIFURCATION}). Moreover, in the supercritical regime $\frac65<\gamma<\frac43$ it is not known whether there exist {\em any} examples of compactly supported collapsing solutions and this remains an important open question. Finally, despite the conditional stability results~\cite{Rein, LuSm} in the subcritical regime $\frac43<\gamma<2$, the question of nonlinear stability of Lane-Emden stars remains open, as it is not clear whether fluid shock singularities might develop from small perturbations.

When $\gamma\ge\frac43$ it was shown in~\cite{MaPe1990, DLYY} that {\em if} there exists a global-in-time strictly positive energy solution to the spherically-symmetric EP$_{\gamma}$-system, then the star support grows {\em at least} linearly-in-time. 
Theorem~\ref{T:LINEAREXPANSION} shows  that in the vicinity of a linearly expanding homogeneous solutions this result is optimal, i.e.  the star support grows {\em at most} linearly-in-time.
Notice that the sub linear growth  of the fluid support  stated in Theorem~\ref{T:MAIN} is associated only with the zero-energy configurations and therefore does not contradict~\cite{MaPe1990}.

Besides the work of Goldreich \& Weber~\cite{GoWe}, there exist numerous other physics works devoted to various notions of self-similar collapsing/expanding star solutions of the Euler-Poisson system, see~\cite{La,Pe,Yahil1983,BlBoCh,BoFeFiMu,RB} and references therein. The polytropic index $\gamma$ is often allowed to vary in the full range $1\le\gamma<2$. All of these other solutions (some of which lead to physically very interesting scenarios of stellar collapse~\cite{Yahil1983}) have an infinite support of the star density and therefore do not possess a star-vacuum free boundary. 

As an important consequence of our results, we show the existence of {\textit{global-in-time}}  \textit{large data solutions} to the EP$_\gamma$-system. To the best of our knowledge, this is the first work proving existence and uniqueness of {{global-in-time}} classical solutions of the free-boundary EP$_{\gamma}$-system that are neither homogeneous nor stationary. 
Our results provide a set of initial data and an energy space that lead to global-in-time solutions in a functional framework that allows us to track the regularity and behavior of the \textit{free boundary}. A key mechanism that precludes 
finite time singularity formation is  the expansion of the background star solution. The fluid density decreases sufficiently rapidly to counteract a possible shock formation. 
A related phenomenon, albeit in the absence of free boundaries, appears in the context of the Euler-Einstein system with positive cosmological constant~\cite{iRjS2012,jS2012,HaSp,cLjVK}, wherein the role of the expanding homogeneous solutions is replaced by the well-known Friedman-Lema\^itre-Robertson-Walker expanding spacetimes.

The presence of the physical vacuum condition, while natural from the physics point of view, causes fundamental difficulties in the analysis, as the first derivative of the fluid enthalpy is in general discontinuous across the fluid-vacuum interface.
The first step in handling both the free boundary and isolating the difficulties coming from the above mentioned degeneracy is to use the Lagrangian coordinates. In these coordinates, a natural functional-analytic framework of weighted Sobolev spaces emerges as the right setting for the well-posedness analysis. The initial fluid density is a natural weight in the problem. These tools have been first designed and successfully used to overcome the above mentioned difficulties in the context of compressible Euler equations in vacuum~\cite{CoSh2012, JaMa2015} and in the context of the Euler-Poisson system in~\cite{Jang2014, Jang2015}. A similar methodology was used to prove local-in-time existence for the Euler-Poisson system~\cite{LXZ}. In \cite{Makino2015}, Nash-Moser theory was used to establish the existence of smooth local solutions of the Euler-Poisson system approximating time periodic (linearized) profiles. 

The fluid densities associated with the expanding homogeneous stars also satisfy the physical vacuum condition, as they solve the generalized Lane-Emden equation.
To prove nonlinear stability of such solutions, we adapt the Lagrangian flow map to the expanding background, and control the deviation by suitably rescaling the time variable. 
In our approach, the initial density of the perturbation enters in the choice of the initial flow map. 

To prove Theorem~\ref{T:SSINFORMAL}
we introduce the self-similar time coordinate $s$ defined by
\be\label{E:SSTIME}
\frac{ds}{dt} = \frac 1{\bar{\lambda}(t)^{3/2}},  \ \ \text{i.e.} \ \ s\sim\log t.
\ee
The self-similar expanding homogeneous solutions transform into steady states of the newly rescaled equation given by
\begin{align}\label{E:PHIEQUATIONINTRO}
\phi_{ss} -\frac12 b\phi_s +3\delta\phi + \mathcal L_\delta\phi = N_\delta[\phi],
\end{align}
where $\mathcal L_\delta$ is a suitable second order self-adjoint differential operator, $N_\delta$ a nonlinearity and $b<0$ is a constant. 
Due to the scaling invariance of the problem, coefficients in front of the $\partial_s$-derivatives in~\eqref{E:PHIEQUATIONINTRO} are {\em constant}. 
A key mechanism for our stability result is a  new \textit{damping effect} induced by the term $ -\frac12 b\phi_s$ (recall that $b<0$) in~\eqref{E:PHIEQUATIONINTRO}. 
It exists as a consequence of the expansion of the background self-similar homogeneous solution.
However, due to a stretch effect coming from $3\delta\phi$ (note that $\delta<0$), the linear part admits a growing mode $\phi=e^{-bs}$, which makes the stability nontrivial and interesting. 
It is not a priori clear that this growing mode will not persist at the nonlinear level. However, the unstable mode is in fact a reflection of the conservation-of-energy law and the associated eigenvector is tangential to the manifold of constant energy solutions. 
Therefore to deal with such an instability we crucially use a nonlinear constraint inherited from the conservation of the physical energy, thereby reducing our stability result to a codimension-one set of initial data.

The stability of the homogeneous self-similar expanding solutions relies on proving an exponential decay-in-$s$ of small amplitude perturbations with zero energy.
Such a decay result follows from  the energy-dissipation inequalities for suitably defined high-order energy functionals $\mathcal E,\mathcal D$ taking the form 
\[
\mathcal{E}(s) +\int_0^s \mathcal{D}(\tau)d\tau  \lesssim \mathcal{E}(0)+  \int_0^s \mathcal{E}^{3/2}(\tau)d\tau , \quad s\geq 0 
\] 
where we have the full coercivity of the dissipation
\[
\mathcal{D} \gtrsim \mathcal{E}. 
\]
Our nonlinear estimates naturally split into two parts. First 
we prove the energy-dissipation inequalities for pure $\g_s$-derivatives of $\phi$ using the standard Hardy inequalities and embeddings between weighted Sobolev spaces. 
In our second step we provide the bounds for the full energy (involving purely spatial and mixed space-time derivatives) and this is accomplished by using~\eqref{E:PHIEQUATIONINTRO} and
a bootstrap procedure allowing us to express higher spatial derivatives in terms of $\g_s$-derivatives. We refer to this part of the analysis as the {\em elliptic estimates}. 
A crucial tool in the proof of the energy-dissipation bounds is the 
spectral gap of the linearized operator $\mathcal L_\delta$. Specifically $\mathcal L_\delta$ is a strictly positive operator as long as it acts on a space orthogonal to its null-space, i.e. 
as long as $(\phi,1)_\delta=0$. We therefore need to control $|(\phi,1)_\delta|$ separately. 
This is accomplished by exploiting the nonlinear constraint $E(\phi,\phi_s)=0$ which restricts the dynamics to the surface of zero-energy solutions. 
Specifically we show that the linearization of this condition yields a relationship of the schematic form
\[
(\phi,1)_\delta= \text{ const.}\times \, (\phi_s,1)_\delta + \text{ quadratic nonlinearity }
\]
which will allow us to show $|(\phi,1)_\delta|\lesssim \E$.

To prove Theorem~\ref{T:LINEARINFORMAL} we also rescale time. Note however that the linearly expanding homogeneous solutions $(\tilde\lambda ,\tilde\rho,\tilde v)$ do not expand at a self-similar rate. Therefore, we introduce a new time coordinate
$\tau$ by setting 
\be\label{E:LINEARTIME}
\frac{d\tau}{dt} = \frac 1{\tilde{\lambda}(t)},  \ \ \text{i.e.} \ \ \tau\sim\log t.
\ee
Like above, the homogeneous background solution $(\tilde\lambda ,\tilde\rho,\tilde v)$ transforms into a steady state, but unlike the self-similar case, the new equation for a perturbation $\phi$ in Lagrangian coordinates has $\tau$-dependent coefficients
in front of $\partial_\tau\phi$ and $\partial_{\tau\tau}\phi$:
\be\label{E:THETAEQUATIONINTRO}
\tilde\lambda \phi_{\tau\tau} +\tilde\lambda _\tau \phi_\tau + 3\delta \phi + \mathcal L_\delta\phi = N_\delta[\phi]. 
\ee
In the new coordinates $\tilde\lambda \sim e^{\beta\tau}$ for large values of $\tau$ where $\beta>0$ is some constant. Although \eqref{E:THETAEQUATIONINTRO} and \eqref{E:PHIEQUATIONINTRO} look structurally similar as they appear to be damped wave equations, they are fundamentally different. By contrast to~\eqref{E:PHIEQUATIONINTRO}, 
the exponentially growing coefficients $\tilde\lambda $ and $\tilde\lambda _\tau$ in \eqref{E:THETAEQUATIONINTRO} are inevitable since the linear-in-time expansion does not honor the invariance of the system and it creates severe analytical difficulties. As a consequence the high-order energy method developed for the self-similar case does not work since the operator $\g_\tau$ does not commute 
well with \eqref{E:THETAEQUATIONINTRO} any longer. Namely, some of the commutators contain $\tau$-dependent weights that grow faster than the natural energy and we cannot close the energy estimates. 
Instead we design a different high-order energy approach, based on differential operators containing only spatial derivatives. 
This idea goes back to the work by Jang \& Masmoudi \cite{JaMa2015} on free-boundary compressible Euler equations 
where the number of spatial derivatives changes the weight  in the energy space. 
In our problem a key elliptic operator 
\[
\mathcal L_{\delta,k}\psi = - \frac{4}{3w_\delta^{3+k}z^4}\g_z (w_\delta^{4+k} z^4\g_z\psi)
\]
captures precisely the degeneracy caused by the physical vacuum. In addition to this we have to deal with the coordinate singularity at $z=0$  because naive successive applications of the $\g_z$-operators may result in unfavorable negative powers of $z$ in our energy terms. To avoid this problem we commute the equation with high-order powers of the elliptic operator 
\[
 \mathcal{S}:= \g_z^2 +\frac{4}{z}\g_z
\]
to derive the energy estimates. It turns out that $\mathcal{S}$ commutes well with \eqref{E:THETAEQUATIONINTRO} at both boundaries. Together with $\mathcal L_{\delta,k}\psi $ it enables us to construct energy spaces without commuting the equation with $\g_\tau$ derivatives and close the energy estimates. The elliptic operators $\mathcal L_{\delta,k} $ and $\mathcal{S}$ have not explicitly appeared in the previous works and a careful and detailed analysis is carried out in Section~\ref{S:LINEAR}.

The paper proceeds as follows. Section~\ref{S:SS} is devoted to  the proof of Theorem~\ref{T:MAIN}: the nonlinear stability of self-similar expanding homogeneous solutions. A crucial coercivity of the energy, leading to the exponential decay of perturbations, is shown in Lemma \ref{L:COERCIVITY}. Nonlinear energy estimates (Theorem~\ref{thm1}) will be presented in detail. In Section~\ref{S:LINEAR}, we prove Theorem~\ref{T:LINEAREXPANSION}: the nonlinear stability of linearly expanding homogeneous solutions. The $\S$-energy method (Theorem~\ref{T:HIGHORDER}) will be developed based on various commutator estimates between $\S$ on one side and linear operator and nonlinear terms (Lemmas \ref{L:SLCOMMUTATOR},  \ref{L:QJ}, \ref{L:CJ}) on the other. We shall also prove comparison estimate 
between the $\S$ energy and the norm (Lemma \ref{L:NORMENERGY}). In the Appendix, we present the spectral theory of our key linearized operators, Hardy inequalities and weighted Sobolev spaces embedding results.

\section{Nonlinear stability of self-similar expanding homogeneous solutions}\label{S:SS}

To understand the stability properties of $\bar\psi\equiv1$, we linearize around this steady state.
Letting 
\[
\psi = \bar\psi + \phi= 1+\phi,
\]
using~\eqref{E:LE} we obtain
\[
\begin{split}
F_{w_\delta}[1+\phi]&= \frac{(1+\phi)^2}{ w_\delta^3 z}\g_z\left(w_\delta^{4}\left[ 1+\frac{1}{z^2}\g_z\left(z^3\left[\phi+\phi^2+\frac{\phi^3}{3}\right]\right) \right]^{-\frac{4}{3}}\right) \\
&\quad+\frac{1}{(1+\phi)^2}\left(-\delta - \frac{4 w_\delta'}{z} \right)\\
\end{split}
\]
Since 
$(1+y)^{-\frac43}= 1-\frac43 y + \frac{28}{9}\left(\int_0^1 (1-\theta)(1+\theta y)^{-\frac{10}{3}} d\theta\right) y^2$, by rearranging terms, we obtain the expansion
\be\label{F_wb}
\begin{split}
&F_{w_\delta}[1+\phi]= -\delta +2 \delta\phi - \frac{4}{3w_\delta^3 z^4}\g_z(w_\delta^4 z^4 \g_z \phi ) \\
&\quad +\frac{b^2}{2}\left(\frac{1}{(1+\phi)^2} - 1 +2\phi \right) + \frac{\g_zw_\delta^4 }{w^3 z} \left((1+\phi)^2 -\frac{1}{(1+\phi)^2} - 4\phi \right)\\
&\quad- \frac43\frac{(1+\phi)^2-1}{w_\delta^3 z}\g_z \left( w_\delta^4 \frac{1}{z^2}\g_z (z^3\phi) \right) \\
&\quad -\frac43 \frac{(1+\phi)^2}{w_\delta^3 z}\g_z \left( w_\delta^4 \frac{1}{z^2}\g_z (z^3(\phi^2+\frac{\phi^3}{3})) \right) \\
&\quad+\frac{28}{9}  \frac{(1+\phi)^2}{ w_\delta^3 z}\g_z\Big[(w_\delta^{4}  \left(\frac{1}{z^2}\g_z\left(z^3\left[\phi+\phi^2+\frac{\phi^3}{3}\right]\right) \right)^2 \\
&\qquad\qquad  \int_0^1 (1-\theta) \left[1+\theta  \left(\frac{1}{z^2}\g_z\left(z^3\left[\phi+\phi^2+\frac{\phi^3}{3}\right]\right) \right) \right]^{-\frac{10}{3}} d\theta   \Big].
\end{split}
\ee

It follows that the linearization of $F_{w_\delta}[\cdot]$ about $\bar\psi\equiv1$ is given by the operator
\be\label{E:LDELTA}
DF_{w_\delta}(\bar\psi)[\phi] = \mathcal L_\delta\phi +2\delta\phi, \ \ \mathcal L_\delta\phi : = -\frac{4}{3w_\delta^3z^4}\partial_z\left(w_\delta^4z^4\partial_z\phi\right).
\ee
Therefore,  the perturbation $\phi$ satisfies the following equation:
\begin{align}\label{E:PHIEQUATION}
\phi_{ss} -\frac12 b\phi_s +3\delta\phi + \mathcal L_\delta\phi = N_\delta[\phi],
\end{align}
equipped with the following initial conditions:
\begin{align}\label{E:PHIINITIAL}
\phi(0,z) =\phi_0(z) = \xi_0(z) - 1, \ \ \phi_s(0,z) = \phi_1(z) = \xi_1.
\end{align}
Here $N_\delta[\phi]$ is the nonlinear remainder defined by the relationship
\be\label{E:NONLINEARITY}
F_{w_\delta}[1+\phi] = -\delta + \mathcal L_\delta\phi + 2\delta\phi - N_\delta[\phi]. 
\ee
The explicit form of $N_\delta$ is obtained from the second through sixth lines of the expansion~\eqref{F_wb}. For the nonlinear estimates however, it will only be important that 
$N_\delta$ has the following structure:
\begin{align}
N_\delta[\phi] = &p_0(\phi) + \frac{\pa_zw_\delta}{z} p_1(\phi) + \frac{1}{w_\delta^3z}p_2(\phi)\pa_z\left(\frac{w_\delta^4}{z^2}\pa_z\left(z^3\phi\right)\right) \notag \\
& + \frac{1}{w_\delta^3z}p_3(\phi)\pa_z\left(\frac{w_\delta^4}{z^2}\pa_z\left(z^3p_4(\phi)\right)\right) \notag \\ 
& + \frac{1}{w_\delta^3z}p_5(\phi)\pa_z\left[w_\delta^4\left(\frac1{z^2}\pa_z\left(z^3p_6(\phi)\right)\right)^2 
\int_0^1(1-\theta)\left(1 +\frac\theta{z^2}\pa_z\left(z^3p_6(\phi)\right)\right)^{\frac{10}3}\,d\theta\right] , \label{E:NBSTRUCTURE}
\end{align}
where $p_i,$ $i=0,\dots,6$ are rational polynomials such that $p_i:[-c,c]\to\mathbb R$ is a $C^\infty$ function for some $c>0$ and $i=0,\dots, 7.$ 
Furthermore, polynomials $p_0,p_1,p_4$ are at least quadratic in $\phi$, and polynomials $p_2,p_6$ are at least linear in $\phi$ when expanded about zero.
In other words
\be\label{E:POLYNOMIALS}
p_i(0)\ \ i=0,1,2, 4,6,  \ \ p_i'(0)=0, \ \ i= 0,1,4.
\ee

It is easy to see that the operator $\mathcal L_\delta$  defined in~\eqref{E:LDELTA} is non-negative and symmetric with respect to $(\cdot,\cdot)_\delta:$
\[
(\mathcal L_\delta\phi,\phi)_\delta = (\phi,\mathcal L_\delta\phi)_\delta = \frac43\|\pa_z\phi\|_{\delta,1}^2
\]
for all $\phi\in C_c^\infty([0,1]) $. Moreover, following the same argument in \cite{Be1995, Makino2015}, we deduce that $\mathcal L_\delta$ has the Friedrichs extension, which is a self-adjoint operator in the weighted space $L^2_{\delta,0}$ whose spectrum is purely discrete. See Appendix \ref{A:A} for more detail on the analysis of the linearized operator. 
The following coercivity properties result from the spectral theory of $\mathcal L_\delta$, which will be crucial for the energy estimates in Section \ref{sec:EE}.

\begin{lemma}[Spectral gap for the linearized operator]\label{lem:core}  
There exists a sufficiently small $\varepsilon>0$ such that for any $-\varepsilon<\delta<0$ the following statements hold:
\begin{enumerate}
\item 
There exists a constant $\mu_1>0$ such that  for any $\varphi\in H^1_{\delta}$ satisfying $(\varphi,1)_\delta=0$
the following bound holds:
\be\label{E:SPECTRALGAP}
(\mathcal L_\delta\varphi,\varphi)_\delta \ge \mu_1 \|\varphi\|_\delta^2 
\ee
\item There exists $\mu_2>0$ such that for any $\varphi\in H^1_{\delta}$ satisfying $(\varphi,1)_\delta=0$
the following bound holds:
\be \label{E:SPECTRALGAP2}
\left((\mathcal L_\delta +3\delta)\varphi,\varphi\right)_\delta \geq \mu_2\left( \|\partial_z\varphi\|_{\delta,1}^2+\|\varphi\|_\delta^2 \right) 
\ee
\end{enumerate}
Here $\mu_1$ and $\mu_2$ can be chosen uniformly in $\delta$. 
\end{lemma}

The proof of Lemma \ref{lem:core} will be given in Appendix \ref{A:A}.

The conserved energy expressed in terms of the perturbation $\phi$ takes the form
\begin{align}\label{E:ENERGYXI}
E(1+\phi,\phi_s) & = \frac{2\pi}{\bar\lambda} \|\phi_s + \frac{\bar\lambda_s}{\bar\lambda}(1+\phi)\|_\delta^2 + \frac{4\pi\delta}{\bar\lambda} \|(1+\phi)^{-1/2}\|_\delta^2 \notag \\
& \ \ \ \ +\frac{12\pi}{\bar\lambda} \big(\int_0^1 ((1+\phi)^2(1+\phi+z\phi_z))^{-1/3}w_\delta^4 z^2\,dz  \notag \\
& \ \ \ \ + \frac13 \int_0^1 (1+\phi)^{-1}\pa_z(w_\delta^4) z^3\,dz\big).
\end{align}
Linearizing the relation
\be\label{E:VANISHINGENERGY}
E(1+\phi,\phi_s) = E(1,0)=0,
\ee
we conclude that 
\begin{align}\label{E:ENERGYCONDITION}
\frac32 b(\phi,1)_\delta - (\phi_s,1)_\delta = \mathcal J[\phi],
\end{align}
where $\mathcal J[\phi]$ is a quadratic nonlinear expression given by 
\begin{align}\label{E:JDEFINITION}
\mathcal J[\phi] &= \frac{\bar\lambda}{4\pi b}\left(E(1+\phi,\phi_s) - E(1,0) - \partial_\phi E(1,0) \phi\right)\notag \\
=&\frac{1}{2b} \Big(  \|\phi_s -b \phi\|_\delta^2 + 2 \int_0^1 \left(\delta + \frac{4w_\delta'}{z}\right) \frac{\phi^2}{1+\phi}  w^3_\delta z^4dz\notag \\
&+ 6\int_0^1  ( ((1+\phi)^2(1+\phi+z\phi_z))^{-1/3} - 1 +\frac {3\phi + z\phi_z}{3} ) w^4_\delta z^2 dz \Big). 
\end{align}

In addition to the total energy $\E$ defined in~\eqref{energy1} we also introduce an
auxiliary energy $\mathscr E(s):$
\be\label{E:ENERGY2}
\mathscr E(s) : = \| \phi\|_{\delta}^2+ \sum_{j=0}^7 \left[\|\g_s^{j+1}\phi\|_{\delta}^2 +  \|\g_s^{j}\g_z \phi\|_{\delta,1}^2\right]. 
\ee

\begin{theorem}[High-order energy estimate]
\label{thm1} Let $\delta<0$ be given and let $|\delta|$ be sufficiently small. 
Let $\phi$ solve the degenerate wave equation \eqref{E:PHIEQUATION}--\eqref{E:PHIINITIAL} on a time interval $[0,S]$ Assume that for a small but fixed $M>0$ the following a priori bound holds:
\be\label{E:APRIORI}
\sup_{s\in[0,S]}\Energy(s) \le M. 
\ee
There exists a constant $\varepsilon$ such that for any initial data $(\phi(0),\phi_s(0))=(\phi_0,\phi_1)$ 
satisfying the initial bound:
\be\label{E:INITIALPHI}
\Energy(0) \le \varepsilon, 
\ee
and satisfying the vanishing energy assumption
\[
E(1+\phi_0,\phi_1)=0,
\]
there exist  constants $C_1,C_2,C_3>0$ 
such that the following energy bound holds: 
\be\label{E:MAINBOUND}
\sup_{s\in[0,S]} \Energy(s) + C_1\int_0^s\Energy(\sigma)\,d\sigma \le C_2\Energy(0) + C_3\int_0^s \Energy(\sigma)^{\frac32}\,d\sigma, \ \ s\in[0,S]. 
\ee
\end{theorem}

Before we prove Theorem \ref{thm1} we collect the following key estimates.

\subsection{Nonlinear energy estimates}\label{sec:EE}

\begin{proposition}[Energy equivalence] \label{P:EQUIVALENCE}
Assume the same as in Theorem \ref{thm1}. 
Then the total energy ${\Energy (s)}$ is bounded by the instant energy $\mathscr E(s).$ 
\be\label{elliptic}
\Energy(s) \lesssim \mathscr E(s), \ \ s\ge0.
\ee
\end{proposition}


\begin{proof}
The equivalence relation~\eqref{elliptic} is a non-trivial consequence of the fact that $\phi$ solves the equation~\eqref{E:PHIEQUATION}:
\be\label{elliptic1}
-\mathcal L_\delta \phi =  \phi_{ss} -\frac12 b\phi_s +3 \delta \phi -  N_\delta[\phi].
\ee
By exploiting the elliptic structure of $\mathcal L_\delta$, we will establish the following estimates: 
\be\label{E:ELLIPTIC2}
\Energy(s) \lesssim \mathscr E(s) + \Energy(s)^{2}, \ \ s\ge0.
\ee
Then the claim follows from~\eqref{E:ELLIPTIC2} for sufficiently small $\Energy(s).$ The proof of~\eqref{E:ELLIPTIC2} is analogous to the proof of Proposition 4.1 from~\cite{Jang2014} and it proceeds by induction.  Here we briefly discuss how to obtain \eqref{E:ELLIPTIC2}. We start with $\|\g_z^2\phi\|^2_{\delta,2}$. In order to see that $\|\g_z^2\phi\|^2_{\delta,2}$ is bounded by the right-hand side of \eqref{E:ELLIPTIC2}, square \eqref{elliptic1} and take the inner product with 1: 
\be\label{elliptic2}
(1,(\mathcal L_\delta \phi)^2)_\delta \lesssim \|\phi_{ss}\|_\delta^2 + \|\phi_s\|_\delta^2 +\| \phi \|_\delta^2+ \| N_\delta[\phi]\|_\delta^2. 
\ee
By using the $L^\infty$ estimates and Hardy inequalities, one can deduce that 
\[
 \| N_\delta[\phi]\|_\delta^2 \lesssim  \mathcal{E} \left(\|\phi\|_\delta^2 + \|\g_z\phi\|^2_{\delta,1}+ \|\g_z^2\phi\|^2_{\delta,2}\right). 
\] 
The left-hand-side of \eqref{elliptic2} can be rewritten as 
\begin{align*}
(1,(\mathcal L_\delta \phi)^2)_\delta&=\frac{16}{9}\int_0^1 (w_\delta \phi_{zz} + \frac{4w_\delta\phi_z}{z}+ 4w_\delta' \phi_z )^2 w_\delta^3 z^4 dz\\
&= \frac{16}{9}\Big[ \int_0^1(\phi_{zz})^2 w_\delta^{5}z^4 dz + 4\int_0^1 (\phi_z)^2 w_\delta^5 z^2dz  \\
&\quad\quad\quad - 4\int_0^1 w_\delta'' (\phi_z)^2 w_\delta^4z^4 dz -4 \int_0^1 (\phi_z)^2 w_\delta' w_\delta^4z^3 dz \Big]
\end{align*}
where we have integrated by parts. Since $w_\delta'=O(z)$, the last two terms are bounded by $\|\g_z\phi\|_{\delta,1}^2$. Hence we deduce that 
\be\label{elliptic3}
\|\g_z^2\phi\|^2_{\delta,2}  \lesssim  \mathscr E(s) + \Energy(s)^{2}
\ee
The bound for $\|\g_s^k\g_z^2\phi\|^2_{\delta,2}$, $1\leq k\leq 6$ can be obtained from the time differentiated equation of \eqref{elliptic1} in the same fashion. 

To illustrate the induction procedure, we now discuss the case of $\|\g_z^3\phi\|^2_{\delta,3}$. To this end, we multiply \eqref{elliptic1} by $z$ and differentiate with respect to $z$: 
\begin{align*}
\frac43(w_\delta z\phi_{zzz} + 5w_\delta'z\phi_{zz}+5w_\delta\phi_{zz}) = -\frac43(4w_\delta''z\phi_z + 8w_\delta' \phi_z) 
+z\phi_{ssz}+\phi_{ss} \\-\frac12 bz\phi_{sz}-\frac12b\phi_s +3 \delta z\phi_z+3\delta\phi - z \g_zN_\delta[\phi] -N_\delta[\phi].
\end{align*}
Dividing this equation by $z$, squaring it, and then taking $(,)_{\delta,1}$ inner product, we obtain the inequality similar to \eqref{elliptic2}. Notice that the right-hand side consists of quartic terms and quadratic terms involving the instant energy $\mathscr E$ and $\mathcal{E}$ containing two spatial derivatives only, which have been estimated at the previous step. The quartic terms can be estimated by the embedding inequalities and Hardy inequalities. The left-hand side gives rise to $\|\g_z^3\phi\|_{\delta,3}^2$ and remaining terms. The remainder is bounded by $\|\g_z^2\phi\|_{\delta,2}^2$. Therefore, by \eqref{elliptic3}, we have the desired bound for $\|\g_z^3\phi\|_{\delta,3}^2$. The time derivatives of $\g_z^3\phi$ can be estimated in the same way. The estimation of higher order derivatives follows inductively. 
\end{proof}

We next state a technical coercivity lemma, that will be needed in the derivation of our high-order energy identity.
\begin{lemma}\label{L:COERCIVITY}
Let $(\phi,\phi_s)$ be a solution to the nonlinear degenerate wave equation~\eqref{E:PHIEQUATION}. Then there exists $c_0=c_0(|b|),|\delta|$ sufficiently small and constants $C_1,C_2>0$ so that 
the following coercivity bounds hold:
\begin{align}
&\sum_{j=0}^7\left(\|\g_s^j\phi_s\|_\delta^2  +( \mathcal L_\delta\g_s^j\phi +3\delta \g_s^j\phi,\g_s^j\phi )_\delta+c_0 (\g_s^j\phi_s,\g_s^j\phi)_\delta - \frac{c_0b}{4}  \| \g_s^j\phi\|_\delta^2 \right) \notag \\
& \ \ \ \ \ge C_1\mathscr E(s) -  C_2\sum_{j=0}^7\left(\pa_s^j\mathcal J[\phi]\right)^2,\label{E:COERCIVITY1}
\end{align}
\begin{align}
&\sum_{j=0}^7\left( -\frac12 (b +c_0)\|\g_s^j\phi_s\|_\delta^2 + \frac{c_0}{2} ( \mathcal L_\delta\g_s^j\phi +3\delta \g_s^j\phi,\g_s^j\phi  )_\delta\right) 
 \ge C_1 |b|\mathscr E(s) -  C_2|b|\sum_{j=0}^7\left(\pa_s^j\mathcal J[\phi]\right)^2,\label{E:COERCIVITY2}
\end{align}
where $\mathscr E$ and $\mathcal J$ are defined by~\eqref{E:ENERGY2} and~\eqref{E:JDEFINITION} respectively.
Here $b=-\sqrt{2|\delta|}$ and $C_1$ depends on $|\delta|$. 
\end{lemma}


\begin{proof}
Letting $g^j=\pa_s^j\phi,$ $j=0,1,\dots,7,$ we now want to prove the coercivity of the energy functional
\be\label{E:ENERGYPOSDEF1}
\mathcal F(g^j) = \|g^j_s\|_\delta^2 +( \mathcal L_\delta g^j + 3\delta g^j,g^j )_\delta+c_0 (g^j_s,g^j)_\delta- \frac{c_0b}{4}  \| g^j\|_\delta^2,
\ee
where $(\phi,\phi_s)$ is a local-in-time solution to~\eqref{E:PHIEQUATION} and $c_0$ is to be determined. 

We first decompose
\[
\|g^j\|_\delta^2  = \frac1{\|1\|_\delta^2}(g^j,1)_\delta^2 + \|\tilde g^j\|_\delta^2,  \ \ \tilde g^j = g^j - \frac1{\|1\|_\delta}(g^j,1)_\delta.
\] 
Since
\[
\mathcal L_\delta (1) = 0,
\]
and $\mathcal L_\delta$ is a self-adjoint operator with respect to $(\cdot,\cdot)_\delta,$
it is easy to check that 
\[
( \mathcal L_\delta g^j ,g^j )_\delta = ( \mathcal L_\delta \tilde g^j ,\tilde g^j )_\delta. 
\]
Therefore from~\eqref{E:ENERGYPOSDEF1} it follows that
\begin{align*}
\mathcal F(g^j)  & =  \|\tilde g^{j+1}\|_\delta^2 + \frac1{\|1\|_\delta^2}( g^{j+1},1)_\delta^2 + ( \mathcal L_\delta \tilde g^j + 3\delta \tilde g^j,\tilde g^j )_\delta
+ \frac{3\delta}{\|1\|_\delta^2}(g^j,1)_\delta^2  \\
& \ \ \ \ + c_0 (\tilde g^{j+1},\tilde g^j)_\delta  + \frac{c_0}{\|1\|_\delta^2}(g^{j+1},1)_\delta(g^j,1)_\delta -\frac{c_0b}{4}  \| g^j\|_\delta^2. 
\end{align*}

On the other hand, applying $\g_s^j$ to~\eqref{E:ENERGYCONDITION} we obtain the identity
\be\label{E:ENERGYCONDITIONJ}
\frac{3}{2}b(g^j,1)_\delta - (g^{j+1},1)_\delta =  \g_s^j \mathcal J[\phi], \ \ j = 0,\dots,7.
\ee
from which we infer that 
\begin{align*}
& \frac8{9\|1\|_\delta^2}( g^{j+1},1)_\delta^2 + \frac{3\delta}{\|1\|_\delta^2}(g^j,1)_\delta^2 \\
& = \frac1{\|1\|_\delta^2} \left(\frac89 \left(\frac32 b (g^j,1)_\delta- \pa_s^j\mathcal J[\phi]\right)^2+ 3\delta(g^j,1)_\delta^2\right) \\
& = \frac{b^2}{2\|1\|_\delta^2} (g^j,1)_\delta^2 - \frac{8}{3\|1\|_\delta^2}b(g^j,1)_\delta \pa_s^j\mathcal J[\phi] + \frac8{9\|1\|_\delta^2}\left(\pa_s^j\mathcal J[\phi]\right)^2,
\end{align*}
where we have used $\delta=-\frac12 b^2.$
By the spectral gap property~\eqref{E:SPECTRALGAP2},  
the previous two equalities, and the fact that $ g^j =\tilde g^j + \frac1{\|1\|_\delta}(g^j,1)_\delta$, we conclude that for sufficiently small $|\delta|$ and for any $k_1, k_2, k_3>0$
\begin{align}
\mathcal F(g^j) \ge & \|\tilde g^{j+1}\|_\delta^2 +\frac1{9\|1\|_\delta^2}( g^{j+1},1)_\delta^2+ \mu_2\left(\|\pa_z g^j\|_{\delta,1}^2 + \|\tilde g^j\|_\delta^2\right) \notag\\
&+ \frac{c_0|b|}{4}  \|\tilde g^j\|_\delta^2 +(  \frac{c_0|b|}{4}  +\frac{b^2}{2}) \frac{1}{\|1\|_\delta^2} (g^j,1)_\delta^2\notag \\
&  -\frac{    2b^2k_1}{3 \|1\|_\delta^2} (g^j,1)_\delta^2 - \frac8{3 k_1\|1\|_\delta^2}\left(\pa_s^j\mathcal J[\phi]\right)^2+ \frac8{9\|1\|_\delta^2}\left(\pa_s^j\mathcal J[\phi]\right)^2 \notag \\
&- k_2\|\tilde g^{j+1}\|_\delta^2- \frac{c_0^2}{4k_2} \|\tilde g^{j}\|_\delta^2 - \frac{k_3}{\|1\|_\delta^2}( g^{j+1},1)_\delta^2 -\frac{c_0^2}{4k_3\|1\|_\delta^2}( g^{j},1)_\delta^2 . 
\end{align}
We can choose $c_0, k_1,k_2,k_3>0$ such that 
the following inequality holds 
\begin{align}
\mathcal F(g^j) \ge c_1\left(\|g^{j+1}\|_\delta^2+\|\pa_z g^j\|_{\delta,1}^2 + \| \tilde g^j\|_\delta^2 \right) + c_2  \frac{1}{\|1\|_\delta^2} (g^j,1)_\delta^2  - \frac C{\|1\|_\delta^2}\left(\pa_s^j\mathcal J[\phi]\right)^2
\end{align}
where $c_1=c_1(\mu_2)$, $c_2=c_2(|\delta|)$. For instance, $k_1=k_2=1/2$, $k_3=1/18$, $c_0=|b|/6$ will work. 
Summing the above bound over all $j\in\{0,\dots,7\}$,  
we obtain the estimate 
\begin{align}\label{E:COERCIVE1}
\sum_{j=0}^7\mathcal F(g^j) \ge C_1\sum_{j=0}^7\left(\|\pa_s^{j+1}\phi\|_\delta^2+\|\pa_s^{j}\phi\|_\delta^2+\|\pa_s^{j}\pa_z\phi\|_{\delta,1}^2\right)
- C\sum_{j=0}^7\left(\pa_s^j\mathcal J[\phi]\right)^2
\end{align}
where $C_1=C_1(|\delta|)$. This concludes the proof of estimate~\eqref{E:COERCIVITY1}. 
Bound~\eqref{E:COERCIVITY2} follows in an analogous way.
\end{proof}


\subsection{Proof of Theorem~\ref{thm1}}

Evaluating the inner product of~\eqref{E:PHIEQUATION} and $\phi_s$ we obtain the fundamental energy dissipation law:
\be\label{E0}
\frac{1}{2}\frac d{ds}\left(\|\phi_s\|_\delta^2 +( \mathcal L_\delta\phi +3\delta \phi,\phi )_\delta\right) - \frac12 b\|\phi_s\|_\delta^2 = \left(N_\delta[\phi], \phi_s\right)_\delta. 
\ee
We can also evaluate the inner product of ~\eqref{E:PHIEQUATION} and $\phi$ to obtain
\[
(\phi_{ss},\phi )_\delta - \frac{b}4 \frac d{ds}  \|\phi\|_\delta^2   + ( \mathcal L_\delta\phi + 3\delta \phi,\phi  )_\delta = \left(N_\delta[\phi], \phi\right)_\delta. 
\]
Now we rewrite the first term as 
\[
(\phi_{ss},\phi )_\delta = \frac d{ds} (\phi_s,\phi)_\delta - \|\phi_s\|_\delta ^2. 
\]
By adding a small constant multiple of this identity to \eqref{E0} and integrating with respect to $s$ we obtain the following key energy identity: 
\begin{align}
& \frac{1}{2}\left(\|\phi_s\|_\delta^2 +( \mathcal L_\delta\phi + 3\delta \phi,\phi )_\delta+c_0 (\phi_s,\phi)_\delta- \frac{c_0b}{4}  \| \phi\|_\delta^2 \right)\Big|^s_0 \notag \\
& +\int_0^s \left( -\frac12 (b +c_0)\|\phi_s\|_\delta^2 + \frac{c_0}{2} ( \mathcal L_\delta\phi +3\delta \phi,\phi  )_\delta\right)\,d\sigma \notag \\
& \ \ 
= \int_0^s \, \left[\left(N_\delta[\phi], \phi_s\right)_\delta + \frac{c_0}{2} \left(N_\delta[\phi], \phi\right)_\delta \right]\, d\sigma. \label{E1}
\end{align}
Using the same calculation for the time differentiated problem we obtain 
\begin{align}
&\left(\|\g_s^j\phi_s\|_\delta^2 +( \mathcal L_\delta\g_s^j\phi +3\delta \g_s^j\phi,\g_s^j\phi )_\delta+c_0 (\g_s^j\phi_s,\g_s^j\phi)_\delta - \frac{c_0b}{4}  \| \g_s^j\phi\|_\delta^2 \right)\Big|^s_0 \notag \\
&+\int_0^s\left( -\frac12 (b +c_0)\|\g_s^j\phi_s\|_\delta^2 + \frac{c_0}{2} ( \mathcal L_\delta\g_s^j\phi +\delta \g_s^j\phi,\g_s^j\phi  )_\delta\right)\,d\sigma\notag \\
&=\int_0^s \, \Big[\left(\g_s^jN_\delta[\phi], \g_s^j\phi_s\right)_\delta + \frac{c_0}{2} \left(\g_s^jN_\delta[\phi], \g_s^j\phi\right)_\delta \Big] d\sigma,  \label{Ej}
\end{align}
for $0\leq j\leq 7$. 

Summing over $j\in\{0,\dots,7\}$ and using Lemma~\ref{L:COERCIVITY}
 we conclude that there exists some positive numbers $\kappa_0',\kappa_1'$ such that 
\begin{align}
\mathscr E(s) +  \kappa_0'\int_0^s\mathscr E(\tau)\,d\sigma
 \le & \, \kappa_1' \mathscr E(0) 
  + C\sum_{j=0}^7 \int_0^s\Big[\left(\g_s^jN_\delta[\phi], \g_s^j\phi_s\right)_\delta + \frac{c_0}{2} \left(\g_s^jN_\delta[\phi], \g_s^j\phi\right)_\delta   + C \left(\pa_s^j\mathcal J[\phi]\right)^2 \Big] \,d\sigma \notag \\
  & + C \left(\pa_s^j\mathcal J[\phi]\right)^2 . \label{E:IDENTITY1}
\end{align}

\noindent
{\em Estimates for the right-hand side of~\eqref{E:IDENTITY1}.}
We first recall the expression~\eqref{E:NBSTRUCTURE}:
\begin{align}
N_\delta[\phi] = &p_0(\phi) + \frac{\pa_zw_\delta}{z} p_1(\phi) + \frac{1}{w_\delta^3z}p_2(\phi)\pa_z\left(\frac{w_\delta^4}{z^2}\pa_z\left(z^3\phi\right)\right) \notag \\
& + \frac{1}{w_\delta^3z}p_3(\phi)\pa_z\left(\frac{w_\delta^4}{z^2}\pa_z\left(z^3p_4(\phi)\right)\right) \notag \\ 
& + \frac{1}{w_\delta^3z}p_5(\phi)\pa_z\left[w_\delta^4\left(\frac1{z^2}\pa_z\left(z^3p_6(\phi)\right)\right)^2 \int_0^1\left(1-\theta +\frac\theta{z^2}\pa_z\left(z^3p_6(\phi)\right)\right)^{\frac{10}3}\,d\theta\right] , \label{E:NBSTRUCTURE1}
\end{align}
where $p_i,$ $i=0,\dots,6$ are rational polynomials such that $p_i:[-c,c]\to\mathbb R$ is a $C^\infty$ function for some $c>0$ and $i=0,\dots, 7.$ 
Furthermore, polynomials $p_0,p_1,p_4$ are at least quadratic in $\phi$, and polynomials $p_2,p_6$ are at least linear in $\phi$ when expanded about zero.
In other words
\[
p_i(0)=0, \ \ i=0,1,2, 4,6,  \ \ p_i'(0)=0, \ \ i= 0,1,4.
\]

Fix $j\in\{0,\dots,7\}.$ From~\eqref{E:NBSTRUCTURE1} it is clear that 
\begin{align}\label{E:EST1}
 & \Big|\int_0^s\left(\g_s^j\left(p_0(\phi)+\frac{\pa_zw_\delta}{z} p_1(\phi)\right) , \, \g_s^{j+1}\phi\right)_\delta\,d\sigma\Big| \lesssim \int_0^s\left[M^{1/2}\Energy(\sigma) + \Energy^{3/2}(\sigma) \right] \, d\sigma,
\end{align}
where we used the $L^\infty$-bounds of Lemma~\ref{L:LINFINITYBOUNDS}, Cuchy-Schwarz inequality, Hardy inequality, 
the quadratic structure of $p_0$ and $p_1$, and the a priori assumption \eqref{E:APRIORI}. 
The third, fourth, and the fifth term on the right-hand side of~\eqref{E:NBSTRUCTURE1} contain two $\pa_z$-derivatives and are therefore rather subtle to estimate. 
We focus first on the third term and rewrite
\be\label{E:EST1'}
\pa_s^j\left(\frac{1}{w_\delta^3z}p_2(\phi)\pa_z\left(\frac{w_\delta^4}{z^2}\pa_z\left(z^3\phi\right)\right)\right) = 
\frac{1}{w_\delta^3z}p_2(\phi)\pa_z\left(\frac{w_\delta^4}{z^2}\pa_z\left(z^3\pa_s^j\phi\right)\right)  + \mathcal R^j,
\ee
where $\mathcal R^j$ is a lower order remainder arising due to the application of the Leibniz rule and it satisfies the energy bound
\begin{align}\label{E:RFINAL}
\Big|\int_0^s (\mathcal R^j,\pa_{s}^{j+1}\phi )_\delta \,d\sigma\Big| \lesssim \int_0^s\left[M^{1/2}\Energy(\sigma) + \Energy^{3/2}(\sigma) \right] \, d\sigma . 
\end{align}
Note that 
\begin{align}
& \left(\frac{1}{w_\delta^3z}p_2(\phi)\pa_z\left(\frac{w_\delta^4}{z^2}\pa_z\left(z^3\pa_s^j\phi\right)\right), \pa_s^{j+1}\phi\right)_\delta \notag \\
& =\int_0^1p_2(\phi)\pa_z\left(\frac{w_\delta^4}{z^2}\pa_z\left(z^3\pa_s^j(\phi)\right)\right) \pa_s^{j+1}\phi z^3\,dz \notag \\
&= - \frac12\pa_s\int_0^1 \frac{w_\delta^4}{z^2}p_2(\phi)|\pa_z(z^3\pa_s^j\phi)|^2\,dz + \frac12 \int_0^1 \frac{w_\delta^4}{z^2}\pa_s(p_2(\phi)) |\pa_z(z^3\pa_s^j\phi)|^2\,dz \notag \\
& \ \ \ \ - \int_0^1\pa_z\left(p_2(\phi)\right)w_\delta^4\pa_z\left(z^3\pa_s^j\phi\right) \pa_s^{j+1}\phi z\,dz \label{E:R2},
 \end{align}
where we integrated by parts.
Note that 
\begin{align}\label{E:TOPORDERONE}
\int_0^1 \frac{w_\delta^4}{z^2}|\pa_z(z^3\pa_s^j\phi)|^2 & \lesssim \int_0^1 w_\delta^4 z^2 (\pa_s^j\phi)^2\,dz + \int_0^1 w_\delta^4 z^4 (\pa_z\pa_s^j\phi)^2\,dz  \lesssim \mathscr E,
\end{align}
where we used the Hardy inequality in the last estimate to conclude that $\int_0^1 w_\delta^4 z^2 (\pa_s^j\phi)^2\,dz \lesssim \int_0^1 w_\delta^4 z^4 (\pa_s^j\phi)^2\,dz+ \int_0^1 w_\delta^4 z^4 (\pa_z\pa_s^j\phi)^2\,dz \le \mathscr E.$
Therefore from~\eqref{E:R2} we conclude that 
\begin{align}
&\Big| \int_0^s \left(\frac{1}{w_\delta^3z}p_2(\phi)\pa_z\left(\frac{w_\delta^4}{z^2}\pa_z\left(z^3\pa_s^j\phi\right)\right), \pa_s^{j+1}\phi\right)_\delta d\sigma\Big| \notag \\
& \lesssim \|p_2(\phi)\|_{L^\infty} \mathscr E(s) + \sup_{s\in[0,S]}\|p_2(\phi)\|_{L^\infty} \mathscr E(0) + \sup_{s\in[0,S]}\|\pa_sp_2(\phi)\|_{\infty} \int_0^s \mathscr E(\sigma)\,d\sigma  \notag \\
& \ \ \ \ + \sup_{s\in[0,S]}\|\pa_zp_2(\phi)\|_{\infty} \int_0^s \|\frac1{z^3}\pa_z(z^3 \pa_s^j\phi)\|_\delta\|\pa_s^{j+1}\phi\|_\delta \,d\sigma \notag \\
& \lesssim \Energy^{1/2} (\mathscr E(0)+\mathscr E(s)) + \sup_{s\in[0,S]}\Energy^{1/2}  \int_0^s \mathscr E(\sigma)\,d\sigma. \label{E:R2FINAL}
\end{align}
The estimate for the fourth term on the right-hand side of~\eqref{E:NBSTRUCTURE1} is entirely analogous to the previous bound and we conclude that
\begin{align*}
&\Big| \int_0^s\pa_s^j\left(\frac{1}{w_\delta^3z}p_3(\phi)\pa_z\left(\frac{w_\delta^4}{z^2}\pa_z\left(z^3p_4(\phi)\right)\right), \pa_s^{j+1}\phi\right)_\delta d\sigma\Big| \\
 & \ \ \ \ \lesssim \Energy^{1/2} (\mathscr E(0)+\mathscr E(s)) + \sup_{s\in[0,S]}\Energy^{1/2}  \int_0^s \mathscr E(\sigma)\,d\sigma.
\end{align*}
The last term on the right-hand side of~\eqref{E:NBSTRUCTURE1} appears slightly different in structure, but the top order estimate relies on the same integration-by-parts idea 
as in~\eqref{E:R2}--\eqref{E:TOPORDERONE}. Namely, introducing a shorthand notation $q(\phi,\phi_z)=\int_0^1(1-\theta)\left(1 +\frac\theta{z^2}\pa_z\left(z^3p_6(\phi)\right)\right)^{\frac{10}3}\,d\theta$, the term we need to estimate takes the form
\[
\Big| \int_0^s \left(\pa_s^j\left(\frac{1}{w_\delta^3z}p_5(\phi)\pa_z\left[w_\delta^4\left(\frac1{z^2}\pa_z\left(z^3p_6(\phi)\right)\right)^2 q(\phi,\phi_z)\right]\right)\, ,\, \pa_s^{j+1}\phi\right)_\delta \,d\sigma \Big|
\]
Commuting the operator $\pa_s^j$ all the way through in the above expression and recalling the definition of $(\cdot,\cdot)_\delta$,  we see that the highest order term takes the form
\[
\Big| \int_0^s \int_0^1p_5(\phi)\pa_z\left[w_\delta^4\left(\frac1{z^2}\pa_z\left(z^3p_6(\phi)\right)\right)^2 \pa_s^jq(\phi,\phi_z)\right] \pa_s^{j+1}\phi z^3\, dz d\sigma \Big|
\]
Integrating-by-parts with respect to $z$ we can rewrite the above term in the form
\begin{align}
&\Big| \int_0^s \int_0^1p_5(\phi)\pa_z\left[w_\delta^4\left(\frac1{z^2}\pa_z\left(z^3p_6(\phi)\right)\right)^2 \pa_s^jq(\phi,\phi_z)\right] \pa_s^{j+1}\phi z^3\, dz d\sigma \Big| \notag \\
&\le\Big| \int_0^s \int_0^1p_5(\phi)w_\delta^4\left(\frac1{z^2}\pa_z\left(z^3p_6(\phi)\right)\right)^2 \pa_s^jq(\phi,\phi_z) \pa_z\left(\pa_s^{j+1}\phi z^3\right)\, dz d\sigma \Big| \notag\\
& \ \ \ \ + \Big| \int_0^s \int_0^1\pa_zp_5(\phi)w_\delta^4\left(\frac1{z^2}\pa_z\left(z^3p_6(\phi)\right)\right)^2 \pa_s^jq(\phi,\phi_z) \pa_s^{j+1}\phi z^3\, dz d\sigma \Big|.\label{E:EST2}
\end{align}
The key property of the nonlinearity $q(\phi,\phi_z)$ is that at the top order 
\[
\pa_s^jq(\phi,\phi_z)=\frac{10}3\int_0^1\theta(1-\theta)\left(1 +\frac\theta{z^2}\pa_z\left(z^3p_6(\phi)\right)\right)^{\frac{7}3}\,\frac1{z^2}\pa_z\left(z^3\pa_s^jp_6(\phi)\right)d\theta + 
\text{ lower order terms}.
\]
The top order factor $\pa_s^j\phi_z$ enters linearly in the above expression and we may integrate-by-parts with respect to $s$
by the same reasoning as in~\eqref{E:EST1'}--\eqref{E:R2FINAL}. We thereby use Hardy inequalities and weighted Sobolev space embeddings to bound~\eqref{E:EST2} by 
$C\Energy(0) + C\int_0^s \Energy(\sigma)^{\frac32}\,d\sigma.$ 
All of the remainder terms are of lower order and they are estimated again by a systematic application of H\"older inequality in
conjunction with the bounds from Appendix~\ref{A:B}. Combining~\eqref{E:EST1}--\eqref{E:EST2} we conclude that 
\be\label{E:NBOUND1}
\sum_{j=0}^7 \Big| \int_0^s \left(\g_s^jN_\delta[\phi], \g_s^{j+1}\phi\right)_\delta\,d\sigma\Big|\lesssim \Energy(0) + \int_0^s \Energy(\sigma)^{\frac32}\,d\sigma.
\ee
Analogously to~\eqref{E:NBOUND1} we obtain (an easier estimate)
\be\label{E:NBOUND2}
\sum_{j=0}^7 \Big|\int_0^s \left(\g_s^jN_\delta[\phi], \g_s^j\phi\right)_\delta\,d\sigma\Big|\lesssim  \Energy(0) + \int_0^s \Energy(\sigma)^{\frac32}\,d\sigma.
\ee
Recalling the definition~\eqref{E:JDEFINITION} of $\mathcal J$ and using the Hardy inequalities and weighted Sobolev space embeddings of Appendix~\ref{A:B}, due to the
quadratic structure of $\mathcal J$ it is straightforward to check the bound
\begin{align}\label{E:JBOUND}
\int_0^s\left(\pa_s^j\mathcal J[\phi]\right)^2 \,d\sigma +  \left(\pa_s^j\mathcal J[\phi]\right)^2\lesssim  \Energy(0) + \int_0^s \Energy(\sigma)^{\frac32}\,d\sigma.
\end{align}
Therefore, from~\eqref{E:IDENTITY1} and bounds~\eqref{E:NBOUND1}--~\eqref{E:JBOUND}
it follows that 
\begin{align*}
& \mathscr E(s) +  \kappa_0'\int_0^s\mathscr E(\sigma)\,d\sigma  \le \kappa_1' \mathscr E(0)   + C\left(M^{1/2}+ \sup_{s\ge0}\Energy^{1/2} \right) \int_0^s \mathscr E(\sigma)\,d\sigma. 
 \end{align*}
By the a priori assumption~\eqref{E:APRIORI}, we can absorb the first term on the right-hand side into the left-hand side for sufficiently small $M$ to finally obtain
\begin{align}\label{E:ENERGYCRUCIAL0}
 \mathscr E(s) +  \kappa_0''\int_0^s\mathscr E(\sigma)\,d\sigma  \le \kappa_1'' \mathscr E(0). 
\end{align}
for some constants $\kappa_0'',\kappa_1''>0.$
Repeating the same argument we deduce that for any $0\le s' \le s <\infty$ the following energy bound holds:
\begin{align}\label{E:ENERGYCRUCIAL}
 \mathscr E(s) +  \kappa_0''\int_{s'}^s\mathscr E(\sigma)\,d\sigma \le\kappa_1'' \mathscr E(s').  
\end{align}
Theorem~\ref{thm1} now follows from Proposition~\ref{P:EQUIVALENCE}.



\subsection{Proof of Theorem~\ref{T:MAIN}}

Let $\epsilon<\frac M{2\kappa_1''}$ be given, where $\kappa_1''$ is the constant appearing in~\eqref{E:ENERGYCRUCIAL0}. We define
\[
\mathscr S: = \sup_{s\ge0}\left\{\sup_{0\le\sigma\le s}\Energy(\sigma)\le \epsilon \  \text{ and the solution to~\eqref{E:PHIEQUATION} exists on the interval $[0,s].$}\right\}
\]
From the local well-posedness theorem~\ref{T:LOCAL} it follows that $\mathscr S>0.$
It follows from~\eqref{E:ENERGYCRUCIAL0} that for a sufficiently small $\E(0)$ the solution to~\eqref{E:PHIEQUATION} has to exist globally-in-time, i.e. $\mathscr S=\infty.$ 
Since $\psi=1+\phi$ this implies that the solution to~\eqref{E:PSIEQUATION}--\eqref{E:PSIINITIALCONDITIONS} also exists globally-in-time.
Letting $s\to \infty$ in~\eqref{E:ENERGYCRUCIAL}  we conclude that 
\be\label{E:AUXBOUND}
\int_{s'}^\infty \mathscr E(\sigma)\,d\sigma \le  C\mathscr E(s'), \ \ s'\ge 0.
\ee
Defining $V(s'): = \int_{s'}^\infty \mathscr E(\sigma)\,d\sigma$ we conclude from~\eqref{E:AUXBOUND}
that 
\[
V'(s') = - \mathscr E(s') \le - \kappa_2 V(s'), \ \ s'\ge 0.
\]
By an elementary integration we obtain the decay 
\[
V(s') \le CV(0) e^{-\kappa_2 s'} \le C\mathscr E(0) e^{-\kappa_2 s'}, \ \ \text{ for all }s'\geq 0. 
\]
Now integrate~\eqref{E:ENERGYCRUCIAL} with respect to $s'$ over $[\frac s2,s]$ to get
\[
\mathscr E(s)\frac s2 \le C V(\frac s2)
\]
which in turn gives 
\[
\mathscr E(s) \le C e^{-\kappa s}, \ s\ge0, \ \ \text{ for some $0<\kappa<\frac{\kappa_2}{2}.$}
\]
Together with Proposition~\ref{P:EQUIVALENCE} the last bound completes the proof of Theorem~\ref{T:MAIN}.


\section{Nonlinear stability of linearly expanding homogeneous solutions}\label{S:LINEAR}

Linearizing~\eqref{E:THETAEQUATION} about the steady state $\tilde\theta=1$ and writing $\theta=1+\phi$ we obtain the equation satisfied by $\phi$
\begin{align}\label{E:PHIEQUATION2}
\tilde\lambda \phi_{\tau\tau} +\tilde\lambda _\tau\phi_\tau +3\delta\phi + \mathcal L_\delta\phi = N_\delta[\phi],
\end{align}
equipped with the initial conditions
\begin{align}\label{E:PHIINITIAL2}
\phi(0,z) =\phi_0(z) = \zeta_0(z) - 1, \ \ \phi_\tau(0,z) = \phi_1(z) = \zeta_1.
\end{align}
Operators $\mathcal L_\delta$ and $N_\delta$ 
are defined by~\eqref{E:LDELTA} and~\eqref{E:NBSTRUCTURE} respectively.

With respect to the unknown $\phi$ the energy $E$ takes the form
\begin{align}
E(1+\phi,\phi_\tau) & = 2\pi \|\phi_\tau + \frac{\tilde\lambda _\tau}{\tilde\lambda }(1+\phi)\|_\delta^2 + \frac{4\pi\delta}{\tilde\lambda } \|(1+\phi)^{-1/2}\|_\delta^2 \notag \\
& \ \ \ \ +\frac{12\pi}{\tilde\lambda } \int_0^1 ((1+\phi)^2(1+\phi+z\phi_z))^{-1/3}w_\delta^4 z^2\,dz \notag \\
& \ \ \ \  +\frac{4\pi}{\tilde\lambda } \int_0^1 (1+\phi)^{-1}\pa_z(w_\delta^4) z^3\,dz.
\end{align}
Expanding $E(1+\phi,\phi_\tau)$ around the steady state, we obtain 
\[
{E(1+\phi,\phi_\tau)}={E(1,0)}+4\pi\left(\tilde e - \frac{3\delta}{\tilde\lambda }\right)(\phi,1)_{\delta} +4\pi \frac{\tilde \lambda_\tau}{\tilde\lambda }(\phi_\tau,1)_{\delta} +4\pi \tilde{\mathcal J}[\phi]
\]
where $\tilde e = \lambda_1^2+\frac{2\delta}{\lambda_0}$ and the nonlinear remainder $\tilde{\mathcal J}[\phi]$ is given by
\begin{align}
\tilde{\mathcal J}[\phi] 
=&\frac12 \Big(  \|\phi_\tau + \frac{\tilde\lambda _\tau}{\tilde\lambda } \phi\|_\delta^2 + \frac2{\tilde\lambda } \int_0^1 \left(\delta + \frac{4w_\delta'}{z}\right) \frac{\phi^2}{1+\phi}  w^3_\delta z^4dz\notag \\
&+ \frac6{\tilde\lambda }\int_0^1 \left( ((1+\phi)^2(1+\phi+z\phi_z))^{-1/3} - 1 +\frac {3\phi + z\phi_z}{3} \right) w^4_\delta z^2 dz \Big). \label{E:TILDEJDEFINITION}
\end{align}
Since the energy is conserved i.e. $E(1+\phi,\phi_\tau)=E(1+\phi_0,\phi_1)$, we obtain the identity
\begin{align}\label{E:ENERGYCONS}
\left(\tilde e - \frac{3\delta}{\tilde\lambda }\right)(\phi,1)_{\delta} + \frac{\tilde \lambda_\tau}{\tilde\lambda }(\phi_\tau,1)_{\delta} = \kappa -\tilde{\mathcal J}[\phi]
\end{align}
where $\kappa$ is the physical energy deviation of the perturbation from the background state,
\be\label{kappa}
\kappa =\frac{ E(1+\phi_0,\phi_1) - E(1,0)}{4\pi} . 
\ee


As mentioned in Section~\ref{S:COMMENTS} the exponentially growing coefficients $\tilde\lambda $ and $\tilde\lambda _\tau$ in~\eqref{E:PHIEQUATION2} force us to avoid commuting~\eqref{E:PHIEQUATION2} with higher $\partial_\tau$-derivatives as we did previously in Section~\ref{S:SS}. Instead we develop  
a high-order energy method based on spatial derivatives solely. Inspired by~\cite{JaMa2015} we know that the more degenerate weights are needed for higher-order spatial derivatives in order to capture the physical vacuum singularity. In addition to that, we need to handle coordinate singularities at $z=0$ that arise with repeated application of spatial derivatives. 
To handle these difficulties we already introduced the operator $\mathcal S$ in \eqref{E:SOPERATOR} and we additionally define another carefully chosen elliptic operator.



\begin{definition}[A high-order version of the operator $\mathcal L_\delta$]\label{D:OPERATORS}
For any $k\in\mathbb N$ we define
\begin{align}
\mathcal L_{\delta,k}\psi:&= - \frac{4}{3w_\delta^{3+k}z^4}\g_z (w_\delta^{4+k} z^4\g_z\psi).
\end{align}
\end{definition}


Note that the operator $\mathcal L_{\delta,k}$ defined in Definition~\ref{D:OPERATORS} can be also written in the form
\begin{align}
\mathcal L_{\delta,k}\psi=-\frac43 (w_\delta \g_z^2\psi + (4+k)w'_\delta \g_z\psi +\frac{4w_\delta}{z}\g_z\psi ).
\end{align}

We first state the result for $\mathcal L_{\delta,k}$ analogous to Lemma \ref{lem:core} for $\mathcal L_\delta$. 

\begin{lemma}[Spectral gap for the linearized operator]\label{lem:core2}  
There exists a sufficiently small $\varepsilon>0$ such that for any $-\varepsilon<\delta\leq 0$ the following statements hold:
\begin{enumerate}
\item 
There exists a constant $\mu_{1,k}>0$ such that  for any $\varphi\in H^1_{\delta,k}$ satisfying $(\varphi,1)_{\delta,k}=0$
the following bound holds:
\be\label{E:SPECTRALGAP_2}
(\mathcal L_{\delta,k}\varphi,\varphi)_{\delta,k }\ge \mu_{1,k} \|\varphi\|_{\delta,k}^2 
\ee
\item There exists $\mu_{2,k}>0$ such that for any $\varphi\in H^1_{\delta,k}$ satisfying $(\varphi,1)_{\delta,k}=0$
the following bound holds:
\be \label{E:SPECTRALGAP2_2}
\left((\mathcal L_{\delta,k} +3\delta)\varphi,\varphi\right)_{\delta,k} \geq \mu_{2,k}\left( \|\partial_z\varphi\|_{\delta,k+1}^2+\|\varphi\|_{\delta,k}^2 \right) 
\ee
\end{enumerate}
\end{lemma}

We remark that $\mu_{1,k}$ and $\mu_{2,k}$ can be chosen uniformly in $\delta$. For the spectral theoretic properties of $\mathcal L_{\delta,k}$ and the proof Lemma \ref{lem:core2} 
see Appendix \ref{A:A}.

Recall the definition~\eqref{E:ENERGYLINEAR} of the high-order energy $\tilde{\mathcal E}$.
If $\delta>0$, it is clear that 
\be\label{Epos}
\sum_{j=0}^4\left(\tilde\lambda  \|\S^j\phi_\tau\|_{\delta,2j}^2 + \left((L_{\delta,2j} +3\delta) \mathcal{S}^j\phi ,\mathcal{S}^j\phi \right)_{\delta,2j} \right) \gtrsim \tilde\E. 
\ee
However, if $\delta\leq 0$, $\mathcal L_{\delta,k} +3\delta$ is not positive definite any longer. Nevertheless, we will show that if $\phi$ solves \eqref{E:PHIEQUATION2}--\eqref{E:PHIINITIAL2} and if $|\delta|$ sufficiently small, we have a desired coercivity for nonpositive $\delta$-s with a sufficiently small $|\delta|\ll1$. 


\begin{lemma}[Energy coercivity]\label{L:POSDEF}
Let $\delta\in(\delta^*,0]$ be given and let $\delta>-\tilde\varepsilon$ for some $\tilde\varepsilon>0$.
Assume that $(\phi,\phi_\tau)$ solves~\eqref{E:PHIEQUATION2}--\eqref{E:PHIINITIAL2} on some $\tau$-interval $[0,T]$, $T>0$. 
Then there exist constants $c_j>0$, $0\leq j\leq 4$ and $\tilde C_1,\tilde C_2>0$ such that 
\begin{align}\label{E:co}
\sum_{j=0}^4c_j\left(\tilde\lambda  \|\S^j\phi_\tau\|_{\delta,2j}^2 + 3\delta\|\S^j\phi\|_{\delta,2j}^2 +\frac43 \|\g_z\mathcal{S}^j\phi\|_{\delta,2j+1}^2 \right)&\\
 \ge \tilde C_1 \tilde\E (\tau)
- \tilde C_2 \left(\tilde\E(0) +\big|\tilde{\mathcal J}[\phi]\big|^2  \right)&, \ \ \tau\in[0,T]\nonumber
\end{align}
if $\tilde\varepsilon>0$ is sufficiently small.
\end{lemma}


\begin{proof}
To prove \eqref{E:co}, we first take advantage of~\eqref{E:ENERGYCONS}.
Since $\tilde e>0$ and $\tilde \lambda \sim e^{\sqrt{\tilde e}\tau}$, for sufficiently small $|\delta|$, $\tilde e - \frac{3\delta}{\tilde\lambda }>0$. Together with $1\lesssim\frac{\tilde{\lambda}_\tau}{\tilde{\lambda}}\lesssim1$, it follows that 
\begin{align}\label{E:phi_av}
\left|(\phi,1)_\delta\right| \lesssim \sqrt{\tilde\E(0)} + \left|(\phi_\tau,1)_\delta\right| + \big|\tilde{\mathcal J}[\phi]\big| \lesssim \sqrt{\tilde\E(0)} + \|\phi_\tau\|_{\delta} + \big|\tilde{\mathcal J}[\phi]\big| 
\end{align}
where we have used $|\kappa| \lesssim \sqrt{\tilde\E(0)}$. 

For any $j\ge1$ it is easy to check that 
\begin{align*}
(\S^j\phi,1)_{\delta,2j} & = - (3+2j)\int_0^1\pa_z\S^{j-1}\phi w_\delta^{2+2j}w_\delta'z^4\,dz \\
& =- (3+2j)\left(\pa_z\S^{j-1}\phi,w_\delta'\right)_{\delta,2(j-1)+1}.
\end{align*}
The Cauchy-Schwarz inequality and the physical vacuum condition $-\infty<w_\delta'(1)<0$ together imply that
\be\label{E:KJ}
\frac{\big|(\S^j\phi,1)_{\delta,2j} \big|}{\|1\|_{\delta,2j}} \le k_j \|\pa_z\S^{j-1}\phi\|_{\delta,2(j-1)+1},
\ee
for some $k_j>0$, $j\in\mathbb N$.
For any $j\in\mathbb N$ we have
\begin{align*}
&3\delta\|\S^j\phi\|_{\delta,2j}^2 +\frac43 \|\g_z\mathcal{S}^j\phi\|_{\delta,2j+1}^2  \\
&=\frac{4}{3} \|\g_z\mathcal{S}^j\phi\|_{\delta,2j+1}^2  + 3\delta\|\S^j\phi - \frac1{\|1\|_{\delta,2j}}(\S^j\phi,1)_{\delta,2j}\|_{\delta,2j}^2 + 3\delta \frac1{\|1\|_{\delta,2j}^2}(\S^j\phi,1)_{\delta,2j}^2.
\end{align*}
The first two terms can be rewritten as
\begin{align*}
&\frac{4}{3} \|\g_z\mathcal{S}^j\phi\|_{\delta,2j+1}^2  + 3\delta\|\S^j\phi - \frac1{\|1\|_{\delta,2j}}(\S^j\phi,1)_{\delta,2j}\|_{\delta,2j}^2  =\left((\mathcal{L}_{\delta,2j} +3\delta)\varphi_j, \varphi_j \right)_{\delta,2j}  
\end{align*}
where $\varphi_j = \S^j\phi - \frac1{\|1\|_{\delta,2j}}(\S^j\phi,1)_{\delta,2j}$. 
Since $(\varphi_j,1)_{\delta,2j}=0$, for $|\delta|$ sufficiently small, by Lemma \ref{lem:core2}, there exist $\tilde\mu>0$ such that
\begin{align*}
&\frac{4}{3} \|\g_z\mathcal{S}^j\phi\|_{\delta,2j+1}^2  + 3\delta\|\S^j\phi - \frac1{\|1\|_{\delta,2j}}(\S^j\phi,1)_{\delta,2j}\|_{\delta,2j}^2 \\
&\ge \tilde\mu\left(\|\g_z\mathcal{S}^j\phi\|_{\delta,2j+1}^2+\|\S^j\phi - \frac1{\|1\|_{\delta,2j}}(\S^j\phi,1)_{\delta,2j}\|_{\delta,2j}^2\right),  \ \ j\in\{0,1,2,3,4\}.
\end{align*}
Using~\eqref{E:KJ} it then follows that
\begin{align*}
&3\delta\|\S^j\phi\|_{\delta,2j}^2 +\frac43 \|\g_z\mathcal{S}^j\phi\|_{\delta,2j+1}^2\\
&\ge \tilde\mu\left(\|\g_z\mathcal{S}^j\phi\|_{\delta,2j+1}^2+\|\S^j\phi \|_{\delta,2j}^2\right)
-(3|\delta|+\tilde\mu )k_j^2\|\pa_z\S^{j-1}\phi\|_{\delta,2(j-1)+1}^2.
\end{align*}
Since the last negative term is indexed at $j-1$, we can select constants $c_j>0$ so that it can be absorbed into the linear combination of the left-hand side successively. Hence we deduce that 
\begin{align}
&\sum_{j=0}^4c_j\left(3\delta\|\S^j\phi\|_{\delta,2j}^2 +\frac43 \|\g_z\mathcal{S}^j\phi\|_{\delta,2j+1}^2\right) 
\ge \tilde C_1 \sum_{j=0}^4\left(\|\g_z\mathcal{S}^j\phi\|_{\delta,2j+1}^2+\|\S^j\phi \|_{\delta,2j+1}^2\right) 
- \tilde C_2 \big|(\phi,1)_\delta\big|^2.
\end{align}
By using \eqref{E:phi_av}, we obtain the desired inequality \eqref{E:co}. 
\end{proof}

In addition to the high-order energy $\tilde{\mathcal E}$ we also define a high-order dissipation functional by 
\begin{align}\label{dissipation}
\tilde{\mathcal D }= \tilde{\mathcal D }(\tau)=\sum_{j=0}^4 \frac{\tilde\lambda _\tau}{2} \|\S^j\phi_\tau\|_{\delta,2j}^2 . 
\end{align}


\begin{theorem}[High-order energy estimate]
\label{T:HIGHORDER} Let $-\tilde\varepsilon<\delta<\infty$ be given for a sufficiently small $\tilde\varepsilon>0$, where $\tilde\varepsilon$ appears in Lemma \ref{lem:core2}. 
Let $(\phi,\phi_\tau)$ be a solution to the degenerate wave equation \eqref{E:PHIEQUATION2}--\eqref{E:PHIINITIAL2} on a $\tau$-time interval $[0,T]$. Assume that for a small but fixed $\tilde M>0$ the following a priori bound holds:
\be\label{E:APRIORI}
\sup_{\tau\in[0,T]}\tilde\Energy(\tau) \le \tilde M. 
\ee
There exists a constant $\varepsilon>0$ such that for any initial data $(\phi(0),\phi_s(0))=(\phi_0,\phi_1)$ 
satisfying the initial bound:
\be\label{E:INITIALPHI}
\tilde\Energy(0) \le \varepsilon, 
\ee
there exist  constants $\tilde C_1,\tilde C_2, \tilde C_3>0$ 
such that the following energy bound holds: 
\be\label{E:MAINBOUND2}
\sup_{\tau'\in[0,\tau]} \tilde\Energy(\tau') + \tilde C_1\int_0^\tau\tilde\D(\sigma)\,d\sigma \le \tilde C_2\tilde\Energy(0) + \tilde C_3\int_0^\tau e^{-\beta_2\sigma}\tilde\Energy(\sigma)\,d\sigma, \ \ \tau\in[0,T]
\ee
where $\beta_2$ is defined in Section~\ref{S:STABILITYLINEAR}.
\end{theorem}

 
\begin{lemma}[Algebraic properties of $\mathcal{S}$] 
Assuming that $\Phi,\Psi\in H^2_{\text{loc}}([0,1])$ the following identities hold: 
\begin{enumerate}
\item 
\be
\mathcal{S}(c_1\Phi+c_2\Psi) =c_1 \mathcal{S}\Phi + c_2 \mathcal{S}\Psi
\ \text{ for } \ c_1,c_2\in \mathbb{R}.
\ee
\item
\be\label{E:PRODUCTRULE}
\mathcal{S}(\Phi\Psi) =( \mathcal{S}\Phi) \Psi + \Phi (\mathcal{S}\Psi) + 2 \g_z\Phi \g_z\Psi.
\ee
\item 
\be\label{E:CHAINRULE}
\mathcal{S}(p(\Phi))= p'(\Phi) \mathcal{S}\Phi + p''(\Phi) |\g_z\Phi|^2.
\ee
\end{enumerate}
\end{lemma}

\begin{proof}
The proof of the lemma is a straightforward application of the product and the chain rule.
\end{proof}



\begin{lemma}[Commutator of $\S$ and $\mathcal L_{\delta,k}$]\label{L:SLCOMMUTATOR}
For any $k\in\mathbb N$ the following commutation formula holds:
\begin{align}
\S \mathcal L_{\delta, 2k}\psi =  \mathcal L_{\delta, 2k+2} \mathcal{S} \psi + a_{k+1}(z)\mathcal{S}\psi + b_{k+1}(z) \g_z\psi,
\end{align}
where
\begin{align} 
a_{k}(z) &: = -\frac43 \left [(5+4k)w_\delta''(z) +\frac{4w_\delta'(z)}{z}\right ], \label{E:AK}\\
 b_k(z)&:=-\frac43\left( (2+2k)w_\delta'''(z) -\frac{4(2+2k)w_\delta''(z)}{z} + \frac{4(2+2k)w_\delta'(z)}{z^2} \right).\label{E:BK}
\end{align}
\end{lemma}



\begin{proof}
The proof relies on a direct calculation.
\be
\begin{split}
\g_z\mathcal L_{\delta, 2k}\psi=  -\frac43 (&w_\delta \g_z^3\psi + (5+2k)w_\delta' \g_z^2\psi +\frac{4w_\delta}{z}\g_z^2\psi 
+ [ (4+2k)w_\delta'' +\frac{4w_\delta'}{z} -\frac{4w_\delta}{z^2}] \g_z\psi )
\end{split}
\ee

\be
\begin{split}
\g_z^2\mathcal L_{\delta, 2k}\psi=  -\frac43 (&w_\delta \g_z^4\psi + (6+2k)w_\delta' \g_z^3\psi +\frac{4w_\delta}{z}\g_z^3\psi 
 + [ (9+4k)w_\delta'' +\frac{8w_\delta'}{z} -\frac{8w_\delta}{z^2}] \g_z^2\psi \\
&+ [ (4+2k)w_\delta''' +\frac{4w_\delta''}{z} - \frac{8w_\delta'}{z^2} + \frac{8w_\delta}{z^3}] \g_z\psi )
\end{split}
\ee

Hence we obtain 

\be
\begin{split}
\mathcal{S} \mathcal L_{\delta, 2k} \psi&=(\g_z^2+\frac{4}{z}\g_z)\mathcal L_{\delta, 2k}\psi\\
&=  -\frac43 (w_\delta \g_z^4\psi + (6+2k)w_\delta' \g_z^3\psi +\frac{8w_\delta}{z}\g_z^3\psi 
 + [ (9+4k)w_\delta'' +\frac{4(7+2k)w_\delta'}{z} +\frac{8w_\delta}{z^2}] \g_z^2\psi \\
&\quad+ [ (4+2k)w_\delta''' +\frac{4(5+2k)w_\delta''}{z} + \frac{8w_\delta'}{z^2} - \frac{8w_\delta}{z^3}] \g_z\psi )
\end{split}
\ee

On the other hand, 

\be
\begin{split}
\mathcal L_{\delta, 2k+2} \mathcal{S} \psi&=  - \frac{4}{3w_\delta^{5+2k}z^4}\g_z (w_\delta^{6+2k} z^4\g_z[ (\g_z^2+\frac{4}{z}\g_z)\psi ])\\
&=  - \frac43( w_\delta\g_z^4 \psi  + (6+2k)w_\delta' \g_z^3\psi +\frac{8w_\delta}{z}\g_z^3\psi \\
&\quad + \frac{4(6+2k)w_\delta'}{z}\g_z^2\psi + \frac{8 w_\delta}{z^2}\g_z^2\psi -\frac{4(6+2k)w_\delta'}{z^2}\g_z\psi -\frac{ 8w_\delta}{z^3}\g_z\psi)
\end{split}
\ee

Therefore, 

\be\label{SL}
\begin{split}
\mathcal{S} \mathcal L_{\delta, 2k} \psi&=\mathcal L_{\delta, 2k+2} \mathcal{S} \psi -\frac43 (  [(9+4k)w_\delta'' +\frac{4w_\delta'}{z}]\g_z^2\psi  \\
& \quad+[  (4+2k)w_\delta''' +\frac{4(5+2k)w_\delta''}{z} + \frac{4(8+2k)w_\delta'}{z^2}  ]\g_z\psi )\\
&= \mathcal L_{\delta, 2k+2} \mathcal{S} \psi -\frac43   [(9+4k)w_\delta'' +\frac{4w_\delta'}{z}]\mathcal{S}\psi  \\
& \quad-\frac43[  (4+2k)w_\delta''' -\frac{4(4+2k)w_\delta''}{z} + \frac{4(4+2k)w_\delta'}{z^2}  ]\g_z\psi \\
&= \mathcal L_{\delta, 2k+2} \mathcal{S} \psi + a_{k+1}(z)\mathcal{S}\psi + b_{k+1}(z) \g_z\psi,
\end{split}
\ee
where $a_{k+1}(z)$ and $b_{k+1}(z)$ are given by~\eqref{E:AK} and~\eqref{E:BK} respectively.
\end{proof}


\begin{lemma}\label{L:AKSMOOTH}
For any $k\in\mathbb N$ functions $a_k$ and $b_k$ are smooth and $b_k$ possess the following Taylor expansions
about $z=0:$ $a_{k}(z)=a_{k,0}+a_{k,2}z^2 + O(z^4)$, $\pa_z^{2\ell+1}a_k\big|_{z=0} = 0$,
$b_{k}(z)=b_{k,1} z + b_{k,3} z^3 + O(z^5)$, $\pa_z^{2\ell}b_k\big|_{z=0} = 0$, $\ell\in\mathbb N$.
Moreover, for any $i,j,k\in\mathbb N$  there exists a $c_{ijk}>0$ such that 
\[
\sum_{\ell=0}^j\left(\|\pa_z^{2\ell}\zeta\|_{L^\infty([0,1])} + \|\S^\ell \zeta\|_{L^\infty([0,1])} \right) \le c_{ijk}, \ \ \zeta= a_k, z^{2i+1}b_k.
\]
\end{lemma}

\begin{proof}
Since $w_\delta\in C^\infty(0,1)$, it suffices to check the smoothness of $a_k$ and $b_k$ near $z=0$. From Lemma \ref{lem:w} we have $w_\delta(z)=A_1+A_2z^2+O(z^4)$, $z\sim 0$. Then $w_\delta'/z=2A_2+O(z^2)$ and hence $a_k$ is smooth. For $b_k$, first note that $w_\delta''/z=2A_2/z+O(z)$ and $w_\delta'/z^2=2A_2/z+O(z)$. The last two terms in \eqref{E:BK} are smooth: $-4(2+2k) w_\delta''/z + 4(2+2k) w_\delta'/z^2=O(z)$. It is straightforward to see that only odd powers of $z$ remain in the expansion of $b_k$ and only even powers in the expansion of $a_k$ around $z=0$. The claimed inequality follows easily from the already established Taylor expansions of $a_k$ and $b_k$.
\end{proof}


\subsection{The energy norm and weighted Sobolev norms}


The next lemma allows us to estimate the higher $\pa_z$-derivatives of $\varphi$ in terms of the high-order $\S$-derivatives of $\varphi$
in the weighted spaces $\mathfrak{H}_{\delta,k}^j$ (Definition~\ref{D:DELTAJSPACES}).

\begin{lemma}[From the energy to the norm]\label{L:NORMENERGY}
 Let $k\in \mathbb N$ be given.
\begin{enumerate}
\item
For any $\varphi\in \mathfrak{H}_{\delta,k-1}^1$ the following bound holds:
\begin{align}\label{E:COERCIVITYFIRST}
 \|\g_z^2\varphi\|_{\delta,2k}^2 + \|\frac{\g_z \varphi}{z}\|_{\delta,2k-1}^2 \lesssim \|\S \varphi\|_{\delta, 2k}^2.
\end{align}
\item
Let $\varphi\in \mathfrak{H}_{\delta,k-1}^j,$ $j\in\mathbb \{1,2,3,4\}$. 
Then there exists a constant $C>0$ such that 
\be
\sum_{\ell =1}^{j}\left(\|\g_z^{2\ell }\varphi\|_{\delta,2k+2\ell-2 }^2 + \|\frac{\g_z^{2\ell -1}\varphi}{z}\|_{\delta,2k+2\ell-3}^2+\|\frac{\g_z^{2\ell -2}\varphi}{z^2}\|_{\delta,2k+2\ell -4}^2\right)
\le C \sum_{\ell =1}^j \|\S^\ell \varphi\|_{\delta,2k+2\ell-2}^2.
\ee
\end{enumerate}
\end{lemma}


\begin{proof}
We recall that $w_\delta$ satisfies $w_\delta'<0$, $w_\delta'\sim cz$ near $z=0$, and $w_\delta'=O(1)$ (see Lemma \ref{lem:w}). 
We first examine the case $\ell=1$. Note that 
\be
\begin{split}
(\mathcal{S}\varphi,\mathcal{S}\varphi)_{\delta,2k} &= \int_0^1 w_\delta^{3+2k} z^4 ( \g_z^2\varphi +\frac{4}{z}\g_z\varphi )^2 dz\\
 &=  \int_0^1 w_\delta^{3+2k} z^4 ( (\g_z^2\varphi)^2 + \frac{16}{z^2} (\g_z\varphi )^2 )dz + 4\int_0^1 w_\delta^{3+2k} z^3\g_z(\g_z\varphi)^2    dz\\
 &=  \int_0^1 w_\delta^{3+2k} z^4 ( (\g_z^2\varphi)^2 + \frac{4}{z^2} (\g_z\varphi )^2 )dz - 4(3+2k)\int_0^1 w_\delta^{2+2k} w'z^3(\g_z\varphi)^2    dz. 
\end{split}
\ee
 Each term is positive and it contains the desired square norms for $\g_z^2\varphi$ and $\frac{\g_z\varphi}{z}$: 
\be
\begin{split}
(\mathcal{S}\varphi,\mathcal{S}\varphi)_{\delta,2k}& \gtrsim \|\g_z^2\varphi\|_{\delta,2k}^2 
+ \| \frac{w_\delta\g_z\varphi}{z} \|_{\delta,2k-2}^2+ \|\g_z \varphi\|_{\delta,2k-1}^2 \\
&\gtrsim  \|\g_z^2\varphi\|_{\delta,2k}^2 + \|\frac{\g_z \varphi}{z}\|_{\delta,2k-1}^2, 
\end{split}
\ee
which proves~\eqref{E:COERCIVITYFIRST}.
By Rellich inequality or Hardy inequality, we also obtain 
\[
(\mathcal{S}\varphi,\mathcal{S}\varphi)_{\delta,2k} + \|\varphi\|^2_{\delta,2k-2} \gtrsim   \|  \frac{\varphi}{z^2} \|_{\delta, 2k-2} 
\]


By a similar argument, we obtain 
\be\label{S^2}
(\mathcal{S}^2\varphi,\mathcal{S}^2\varphi)_{\delta,2k+2} \gtrsim \|\g_z^2\mathcal{S}\varphi\|_{\delta,2k+2}^2 
 + \|\frac{\g_z \mathcal{S}\mathcal{\varphi}}{z}\|_{\delta,2k+1}^2 ,
\ee
as well as 
\be
(\mathcal{S}^2\varphi,\mathcal{S}^2\varphi)_{\delta,2k+2} + \|\mathcal{S}\varphi\|^2_{\delta,2k} \gtrsim \|  \frac{\mathcal{S}\varphi}{z^2} \|_{\delta,2k}^2. 
\ee

Next we will show how to recover higher order $\varphi$-norms from $\mathcal{S}\varphi$. Recall that 
\be
\|  \frac{\mathcal{S}\varphi}{z^2} \|_{\delta,2k}^2 = \int_0^1 w_\delta^{3+2k} ( \g_z^2\varphi +\frac{4}{z}\g_z\varphi )^2 dz
\ee
The finiteness of the integral leads to $\g_z\phi|_{z=0}=0$ because letting $f=\g_z^2\varphi +\frac{4}{z}\g_z\varphi $, we have 
\[
|\g_z\varphi |= |\frac{1}{z^4}\int_0^z z^4f dz | \leq z^{\frac12} (\int_0^z f^2 dz)^{\frac12}
\]
which implies $\lim_{z\rightarrow 0^+}\g_z\phi=0$. We deduce that 
$
\lim_{z\rightarrow 0^+}\frac{(\g_z\phi)^2}{z} = 0. 
$
Therefore by integration-by-parts we obtain 
\be\label{bound1}
\|  \frac{\mathcal{S}\varphi}{z^2} \|_{\delta,2k}^2 \gtrsim \|\frac{\g_z^2\varphi }{z^2}\|^2_{\delta, 2k} +\|\frac{\g_z\varphi }{z^3}\|^2_{\delta, 2k-1} . 
\ee

Now we rewrite the second term in the right-hand side of \eqref{S^2} as
\[
\begin{split}
 \| \frac{\g_z\mathcal{S}\varphi}{z} \|_{\delta,2k+1}^2 &= \int_0^1 w_\delta^{4+2k} z^2 | \g_z(\g_z^2 \varphi +\frac{4}{z} \g_z \varphi) |^2 dz \\
 &= \int_0^1 w_\delta^{4+2k} z^2 | \g_z^3 \varphi +\frac{4}{z} \g_z^2 \varphi  - \frac{4}{z^2}\g_z\varphi |^2 dz . 
\end{split}
\]
Due to \eqref{bound1}, we deduce that 
\be
\int_0^1 w_\delta^{4+2k} z^2 |\g_z^3 \varphi|^2 dz \leq \| \frac{\g_z\mathcal{S}\varphi}{z} \|_{\delta,2k+1}^2 + \|  \frac{\mathcal{S}\varphi}{z^2} \|_{\delta,2k}^2
\ee
The first term of the right-hand side of \eqref{S^2} gives 
\[
\|\g_z^2\mathcal{S}\varphi\|_{\delta,2k+2}^2 = \int_0^1 w_\delta^{5+2k} z^4  | \g_z^2(\g_z^2 \varphi +\frac{4}{z} \g_z \varphi) |^2 dz 
\]
By previous bounds, we obtain 
\be
\begin{split}
 \int_0^1 w_\delta^{5+2k} z^4 |\g_z^4 \varphi|^2 dz& \leq \|\g_z^2\mathcal{S}\varphi\|_{\delta,2k+2}^2+ \| \frac{\g_z\mathcal{S}\varphi}{z} \|_{\delta,2k+1}^2 + \|  \frac{\mathcal{S}\varphi}{z^2} \|_{\delta,2k}^2  \\
 &\leq (\mathcal{S}^2\varphi,\mathcal{S}^2\varphi)_{\delta,2k+2} +(\mathcal{S}\varphi,\mathcal{S}\varphi)_{\delta,2k} 
\end{split}
\ee

To sum up, we have shown that 
\be\label{inequality1}
\begin{split}
&\|\frac{\mathcal{S}\varphi}{z^2} \|_{\delta, 2k}^2+\|\frac{\g_z \mathcal{S}\varphi}{z} \|_{\delta, 2k+1}^2   +  \|\g_z^2\mathcal{S}\varphi \|_{\delta, 2k+2}^2 \\
&+  \|\frac{\g_z^2 \varphi}{z^2}\|_{\delta, 2k}^2+ \|\frac{\g_z^3 \varphi}{z}\|_{\delta, 2k+1}^2  + \|\g_z^4 \varphi\|_{\delta, 2k+2}^2  \leq \|\mathcal{S}^2\varphi\|^2_{\delta,2k+2} +\|\mathcal{S}\varphi\|^2_{\delta,2k} 
 \end{split}
\ee

Inductively, we can recover all the spatial norms via $\|\mathcal{S}^\ell\varphi\|^2_{\delta, 2k+ 2\ell-2}$. To be more precise, we first apply \eqref{inequality1} to $\varphi \rightarrow \mathcal{S}\varphi$ and $k\rightarrow k+1$, 
\be\label{inequality2}
\begin{split}
&\|\frac{\mathcal{S}^2\varphi}{z^2} \|_{\delta, 2k+2}^2+\|\frac{\g_z \mathcal{S}^2\varphi}{z} \|_{\delta, 2k+3}^2   +  \|\g_z^2\mathcal{S}^2\varphi \|_{\delta, 2k+4}^2 \\
&+  \|\frac{\g_z^2 \mathcal{S}\varphi}{z^2}\|_{\delta, 2k+2}^2+ \|\frac{\g_z^3 \mathcal{S}\varphi}{z}\|_{\delta, 2k+3}^2  + \|\g_z^4\mathcal{S} \varphi\|_{\delta, 2k+4}^2  \leq \|\mathcal{S}^3\varphi\|^2_{\delta,2k+4} +\|\mathcal{S}^2\varphi\|^2_{\delta,2k+2} 
 \end{split}
\ee
We next note that
\be
\begin{split}
\g_z\mathcal{S}\varphi &= \g_z^3 \varphi +\frac{4}{z}\g_z^2\varphi -\frac{4}{z^2}\g_z\varphi \\
\g_z^2\mathcal{S}\varphi&=  \g_z^4 \varphi +\frac{4}{z}\g_z^3\varphi -\frac{8}{z^2}\g_z^2\varphi +\frac{8}{z^3}\g_z\varphi \\
 \g_z^3\mathcal{S}\varphi&=  \g_z^5 \varphi +\frac{4}{z}\g_z^4\varphi -\frac{12}{z^2}\g_z^3\varphi +\frac{24}{z^3}\g_z^2\varphi - \frac{24}{z^4}\g_z\varphi \\
  \g_z^4\mathcal{S}\varphi&=  \g_z^6 \varphi +\frac{4}{z}\g_z^5\varphi -\frac{16}{z^2}\g_z^4\varphi +\frac{48}{z^3}\g_z^3\varphi - \frac{96}{z^4}\g_z^2\varphi + \frac{96}{z^5}\g_z\varphi.   
\end{split}
\ee
It follows that 
\[
\g_z^3 \mathcal{S}\varphi = \g_z^5\varphi + \frac{7}{z}\g_z^4\varphi - \frac{3}{z}\g_z^2\mathcal{S}\varphi, 
\]
and 
\[
\g_z^4\mathcal{S}\varphi =  \g_z^6 \varphi +\frac{8}{z}\g_z^5\varphi -\frac{4}{z}\g_z^3\mathcal{S}\varphi. 
\]

Then from the boundedness of $\g_z^4\mathcal{S}\varphi$, $\frac{\g_z^3 \mathcal{S}\varphi }{z}$ and $\frac{\g_z^2 \mathcal{S}\varphi }{z^2}$ in \eqref{inequality2}, we deduce that 
\be\label{inequality3}
\| {\g_z^6 \varphi }\|_{\delta, 2k+4}^2+ \| \frac{\g_z^5 \varphi }{z}\|_{\delta, 2k+3}^2 +\| \frac{\g_z^4\varphi }{z^2}\|_{\delta, 2k+2}^2 \leq  \|\mathcal{S}^3\varphi\|^2_{\delta,2k+4} +\|\mathcal{S}^2\varphi\|^2_{\delta,2k+2} . 
\ee

Now applying \eqref{inequality3} to $\mathcal{S}\varphi$ and $k+1$, we also deduce 
\be\label{inequality4}
\| {\g_z^8 \varphi }\|_{\delta, 2k+6}^2+ \| \frac{\g_z^7 \varphi }{z}\|_{\delta, 2k+5}^2 +\| \frac{\g_z^6\varphi }{z^2}\|_{\delta, 2k+4}^2 \leq  \|\mathcal{S}^4\varphi\|^2_{\delta,2k+6} +\|\mathcal{S}^3\varphi\|^2_{\delta,2k+4} . 
\ee
This concludes the proof. 
\end{proof}

\subsection{Nonlinear energy estimates and proof of Theorem~\ref{T:HIGHORDER}}

In the estimates below we shall be freely using~\eqref{E:DELTA1}--\eqref{E:TILDELAMBDABOUND}.
 
We commute the equation~\eqref{E:PHIEQUATION2} with the operator $\S^j,$ $j=0,1,2,3,4.$
As a result we obtain the following equation satisfied by $\S^j\phi:$
\begin{align}
& \tilde\lambda \g_{\tau\tau}\S^j\phi + \tilde\lambda _\tau \g_\tau\S^j\phi + 3\delta\S^j\phi + \mathcal{L}_{\delta,2j}\S^j\phi \notag \\
& \ \ + \sum_{\ell=0}^{j-1}\S^\ell\left(a_{j-\ell}\S^{j-\ell}\phi+ b_{j-\ell}\g_z\S^{j-1-\ell}\phi\right) \label{E:COMMUTED}
= \S^j N_\delta[\phi].
\end{align}
In the derivation of~\eqref{E:COMMUTED} we successively used Lemma~\ref{L:SLCOMMUTATOR}.
Taking the  $(\cdot,\cdot)_{\delta,2j}$ inner product with $\mathcal{S}^j\phi_\tau$ we obtain the following identity:
\begin{align}
&\frac{1}{2}\frac{d}{d\tau} \left[ \tilde\lambda  (\mathcal{S}^j\phi_\tau,\mathcal{S}^j\phi_\tau)_{\delta,2j} + 3\delta(\mathcal{S}^j\phi,\mathcal{S}^j\phi)_{\delta,2j} +\frac43 (\g_z\mathcal{S}^j\phi,\g_z\mathcal{S}^j\phi)_{\delta,2j+1} \right] \notag \\
&+\frac{\tilde\lambda _\tau}{2}  (\mathcal{S}^j\phi_\tau,\mathcal{S}^j\phi_\tau)_{\delta,2j} 
 = \mathcal Q_j +\mathcal C_j, \label{E:ENERGYIDENTITYPART2}
\end{align}
where
\begin{align}
\mathcal Q_j : = 
& - \sum_{\ell=0}^{j-1}\left(\S^\ell\left(a_{j-\ell}\S^{j-\ell}\phi+ b_{j-1-\ell}\g_z\S^{j-1-\ell}\phi\right)\,,\,\S^j\phi_\tau\right)_{\delta,2j} 
\label{E:QJ}
\end{align}
and
\begin{align}\label{E:CJ}
\mathcal C_j : = (\mathcal{S}^j N_\delta[\phi],\mathcal{S}^j\phi_\tau )_{\delta,2j} 
\end{align}
 for any $j=0,1,2,3,4$. 
We recall that the functions $a_k(\cdot),b_k(\cdot)$, $k\in\mathbb N$, are defined by~\eqref{E:AK}--\eqref{E:BK}. 

For any $j\in\{0,1,2,3,4\}$ the quantities $\mathcal Q_j$ represent the quadratic error terms and they are therefore very dangerous as they are a priori of the same order of magnitude 
as the energy itself. The (at least) cubic error terms $\mathcal C_j$ are better behaved from the point of view of the order of magnitude, but they possess an intricate quasilinear structure and different ideas 
are needed to control them.


\subsubsection{Energy estimates for $\mathcal Q_j$}

We recall that for any given $\delta\ge0$ or $\delta^*<\delta<0$ with $\tilde e=\lambda_1^2+\frac{2\delta}{\lambda_0}>0$, the homogeneous solution $\tilde\lambda $ behaves like
\[
\tilde\lambda (\tau)\sim_{\tau\to\infty} e^{\beta \tau}, \ \ \beta = \sqrt{\tilde e}.
\]


\begin{lemma}[Estimate for $\mathcal Q_j$]\label{L:QJ}
Assume the same as in Theorem~\ref{T:HIGHORDER}. 
Then the following energy bound holds:
\begin{align}\label{E:QJBOUND}
\Big|\int_0^\tau\mathcal Q_j\,d\sigma\Big|  \le \nu \int_0^\tau\tilde\D(\sigma)\,d\sigma + C_\nu  \int_0^\tau e^{-\beta_2\sigma}\tilde\E (\sigma)d\sigma , \ \ \tau\in[0,T]
\end{align}
for any $\nu>0$. Here $\beta_2>0$ is a constant defined in \eqref{E:TILDELAMBDABOUND}. 

\end{lemma}


\begin{proof}
Applying the product rule~\eqref{E:PRODUCTRULE}
it is easy to see that the top order error term appearing in the first sum on the right-hand side of~\eqref{E:QJ} is of the form
\be\label{E:TOPORDER}
-\sum_{\ell=1}^j\left(a_{j-\ell}\S^j\phi,\S^j\phi_\tau\right)_{\delta,2j}.
\ee
We apply the Young inequality to estimate~\eqref{E:TOPORDER}:
\begin{align*}
\sum_{\ell=0}^{j-1}\Big|(a_{j-\ell}(z)\mathcal{S}^j\phi,\mathcal{S}^j\phi_\tau)_{\delta,2j} \Big|
& \le   \nu \frac{\tilde\lambda _\tau}2 \|\S^j\phi_\tau\|_{\delta,2j}^2 + C\frac{\sum_{\ell=1}^j\|a_{j-\ell}\|_{L^\infty([0,1])}}{\nu \tilde\lambda _\tau }\|\S^{j}\phi\|_{\delta,2j}^2  \\
& \le  \nu\tilde{ \mathcal D} + C_\nu e^{-\beta_2 \tau} \|\S^{j}\phi\|_{\delta,2j}^2. 
\end{align*} 
The remaining terms appearing in the first sum on the right-hand side of~\eqref{E:QJ} are all below-top order. To illustrate this, consider the case $\ell=1$ in the first sum on the right-hand side of~\eqref{E:QJ}:
\begin{align}
\left(\mathcal S(a_{j-1}\mathcal S^{j-1}\phi)\, , \, \mathcal S^j\phi_\tau\right)_{\delta,2j} = & 
\left(\mathcal Sa_{j-1}\mathcal S^{j-1}\phi\, , \, \mathcal S^j\phi_\tau\right)_{\delta,2j} \notag \\
& + 2\left(\pa_za_{j-1}\pa_z\mathcal S^{j-1}\phi\,,\,\mathcal S^j\phi_\tau\right)_{\delta,2j} 
+\left(a_{j-1}\S^j\phi \, , \, \S^j\phi_\tau\right)_{\delta,2j}. \label{E:BELOWTOP}
\end{align} 
Observe that we used the product rule~\eqref{E:PRODUCTRULE} on the right-hand side above.
The last term on the right-hand side of~\eqref{E:BELOWTOP} has already been bounded above. 
The first term on the right-hand side of~\eqref{E:BELOWTOP} is bounded using the H\"older ineqaulity
\begin{align*}
\Big|\left(\mathcal Sa_{j-1}\mathcal S^{j-1}\phi\, , \, \mathcal S^j\phi_\tau\right)_{\delta,2j}\Big| 
& \lesssim \|\S a_{j-1} w_\delta \|_{\infty}\|\S^{j-1}\phi\|_{\delta, 2j-2}\|\S^j\phi_\tau\|_{\delta,2j} \\
&\le \nu\tilde{ \mathcal D} + C_\nu e^{-\beta_2 \tau}  \tilde\E,
\end{align*}
where we used Lemma~\ref{L:AKSMOOTH} to infer that $ \|\S a_{j-1} w_\delta \|_{\infty}\lesssim1$ and the definition of $\tilde\E$ and $\tilde\D$.
To bound the second term on right-hand side of~\eqref{E:BELOWTOP}, we use~\eqref{E:COERCIVITYFIRST} with $\varphi = \S^{j-1}\phi$. We thus obtain
\begin{align*}
\Big|\left(\pa_za_{j-1}\pa_z\mathcal S^{j-1}\phi\,,\,\mathcal S^j\phi_\tau\right)_{\delta,2j} \Big|
&\lesssim \|\pa_za_{j-1}z{w_\delta}^{1/2}\|_{\infty}\| \frac{\pa_z\S^{j-1}\phi}{z}\|_{\delta,2j-1}\|\S^j\phi_\tau\|_{\delta,2j}\\
&
\le \nu\tilde{ \mathcal D} + C_\nu e^{-\beta_2 \tau}  \tilde\E,
\end{align*}
just like in the previous estimate. Combining the last three bounds we conclude that
 \begin{align}
\Big|\left(\mathcal S(a_{j-1}\mathcal S^{j-1}\phi)\, , \, \mathcal S^j\phi_\tau\right)_{\delta,2j} \Big|
\le \nu\tilde{ \mathcal D} + C_\nu e^{-\beta_2 \tau}  \tilde\E.
\end{align}
Analogously to the above estimate, using Lemmas~\ref{L:AKSMOOTH}, \ref{L:NORMENERGY}, the H\"older inequality and the Young inequality we finally infer
that 
\[
\begin{split}
\Big|\sum_{\ell=0}^{j-1}\left(\S^\ell\left(a_{j-\ell}\S^{j-\ell}\phi\right)\,,\,\S^j\phi_\tau\right)_{\delta,2j}\Big|
&\le  \nu\tilde{ \mathcal D} + C_\nu e^{-\beta_2 \tau}\sum_{\ell=0}^j\|\S^{\ell}\phi\|_{\delta,2\ell}^2 \\
&\le \nu\tilde{ \mathcal D} + C_\nu e^{-\beta_2 \tau}  \tilde\E . 
\end{split}
\]

Analogously, for any $\ell\in\{1,\dots,j\}$ 
we first estimate the top-order term 
\begin{align*}
\Big|(b_{j-1-\ell}(z)\g_z\S^{j-1}\phi , \mathcal{S}^j\phi_\tau)_{\delta,2j}\Big| & \le\nu \frac{\tilde\lambda _\tau}2\|\S^j\phi_\tau\|_{\delta,2j}^2 + C_\nu\frac{\|zb_{j-1-\ell}w_\delta^{1/2}\|^2_{L^\infty([0,1])}}{\tilde\lambda _\tau }
\|\frac{\pa_z\S^{j-1}\phi}{z}\|_{\delta,2j-1}^2 \\
& \le \nu \tilde{\mathcal D} + C_\nu e^{-\beta_2 \tau} \tilde\E  .
\end{align*}
Using the Young inequality and Lemmas~\ref{L:AKSMOOTH}, \ref{L:NORMENERGY}, it is straightforward to check that the remaining below-top-order terms in the first sum 
on the right-hand side of~\eqref{E:QJ} are also bounded by  $\nu \tilde{\mathcal D} + C_\nu e^{-\beta_2 \tau} \tilde\E $.
Therefore, using the above bounds and integrating with respect to $\tau$, we obtain the bound~\eqref{E:QJBOUND}.
\end{proof}



\subsubsection{Convenient representation of $N_\delta$ and commutator identities}

We shall rewrite the quadratic nonlinearity $N_\delta[\cdot]$ defined in~\eqref{E:NONLINEARITY} in a slightly different form, more conducive to our energy estimates.
To that end, for any $k\in\mathbb N\cup\{0\}$ we define the following linear operator:
\[
M_{\delta,k}\phi : =  \frac{1}{w_\delta^{3+k}z}\pa_z\left(\frac{w_\delta^{4+k}}{z^2}\pa_z\left(z^3\phi\right)\right).
\]
It is straightforward to check that the operator $M_{\delta,k}$ is symmetric with respect to the inner product $(\cdot,\cdot)_{\delta,k}$: 
\begin{align*}
\left(M_{\delta,k} \phi_1\, , \, \phi_2\right)_{\delta, k} = - \int_0^1 \frac{w_\delta^{4+k}}{z^2}\pa_z(z^3 \phi_1)\pa_z(z^3\phi_2)\,dz, \ \ \phi_1,\phi_2 \in {H}_{\delta,k}^1. 
\end{align*}

With the aid of $M_{\delta,k}$ the nonlinearity $N_\delta[\phi]$~\eqref{E:NONLINEARITY} can be rewritten in the following form:
\begin{align}
N_\delta[\phi] = &p_0(\phi) + \frac{w_\delta'}{z} p_1(\phi) + p_2(\phi)M_{\delta,0}\phi + p_3(\phi) M_{\delta,0}(p_4(\phi)) + Q[\phi,\phi] \label{E:NEWSTRUCTUREN}.
\end{align}
Here 
\[
Q[\phi,\phi]: =\frac{1}{w_\delta^3z}p_5(\phi)\pa_z\left[w_\delta^4\left(\frac1{z^2}\pa_z\left(z^3p_6(\phi)\right)\right)^2 \int_0^1\left(1-\theta +\frac\theta{z^2}\pa_z\left(z^3p_6(\phi)\right)\right)^{\frac{10}3}\,d\theta\right] 
\]
and rational polynomials $p_i,$ $i=0,\dots,6$ are  defined just before~\eqref{E:POLYNOMIALS}.


\begin{lemma}[Commutator identities for $M_{k,\delta}$ and $\S$]
For any $k\in\mathbb N$ the following commutation property holds:
\begin{align}
\S M_{\delta,k} \phi = M_{\delta,k+2}\S \phi + \left(\alpha_k \S + \beta_k\pa_z + \gamma_k\right)\phi,
\end{align}
where the functions $\alpha_k, \beta_k, \gamma_k$ are given by
\begin{align}
\alpha_k &: = (9+2k)w_\delta'' + 4\frac{w_\delta'}{z} \label{E:AK2}\\
\beta_k& : = -3(4+k)\pi w_\delta^2 w_\delta' \label{E:BK2}\\
\gamma_k &: = -9(4+k)\pi w_\delta^2 w_\delta' \label{E:CK2}
\end{align}
\end{lemma}


\begin{proof}
The proof is a direct calculation analogous to the proof of Lemma~\ref{L:SLCOMMUTATOR} 
and the relationship~\eqref{E:LE}.
\end{proof}


\begin{lemma}[Uniform boundedness of commutator coefficients]
For any $k\in\mathbb N$ functions $\alpha_k, \beta_k, \gamma_k$ from the previous lemma are smooth on $[0,1]$ and for any $j\in\mathbb N$ there exists a $C_{jk}>0$ such that 
\[
\sum_{\ell=0}^j\left(\|\pa_z^{2\ell}\zeta\|_{L^\infty([0,1])} + \|\S^\ell \zeta\|_{L^\infty([0,1])} \right) \le C_{jk}, \ \ \zeta= \alpha_k, \beta_k, \gamma_k.
\]
\end{lemma}


\begin{proof}
The proof is a direct consequence of the formulas~\eqref{E:AK2}--\eqref{E:CK2}, the generalized Lane-Emden equation~\eqref{E:LE} and the smoothness properties of $w_\delta$ near $0$ stated in
Lemma~\ref{lem:w}.
\end{proof}


\begin{lemma}[High-order commutators]\label{L:COMMUTATORSHIGH}
For any $j\in\mathbb N$ the following commutation property holds:
\begin{align}
\S^j M_{\delta,0} \phi = M_{\delta,2j}\S^j \phi +\sum_{\ell=0}^{j-1}\S^\ell\left[\left(\alpha_{j-\ell}\S+ \beta_{j-\ell}\g_z + \gamma_{j-\ell}\right)\S^{j-1-\ell}\phi\right]
\end{align}
where the functions $\alpha_k, \beta_k, \gamma_k$ are given by~\eqref{E:AK2}--\eqref{E:CK2}.
\end{lemma}




\subsubsection{Energy estimates for $\mathcal C_j$}


\begin{lemma}[Estimate for $\mathcal C_j$]\label{L:CJ}
Assume the same as in Theorem~\ref{T:HIGHORDER}. Then the following energy bound holds:
\begin{align}\label{E:CJESTIMATE}
\Big|\int_0^\tau\mathcal C_j\,d\sigma\Big|  \lesssim \tilde\E(0)+\left(\nu+\sup_{0\le\sigma\le \tau}\sqrt{\tilde\E}(\sigma)\right)\int_0^\tau\tilde\D(\sigma)\,d\sigma + \sup_{0\le \sigma\le \tau}\left(  \tilde\E^\frac{3}{2}(\sigma)  + \tilde\E^2(\sigma) \right), \ \ \tau\in[0,T]
\end{align}
for any $\nu>0$. 
\end{lemma}


\begin{proof}
Recall the definition~\eqref{E:CJ} of the cubic error term $\mathcal C_j=
 (\mathcal{S}^j N_\delta[\phi],\mathcal{S}^j\phi_\tau )_{\delta,2j} 
$. 
We can rewrite $(\mathcal{S}^j N_\delta[\phi],\mathcal{S}^j\phi_\tau )_{\delta,2j}$ using~\eqref{E:NEWSTRUCTUREN} in the form
\begin{align}
&(\mathcal{S}^j N_\delta[\phi],\mathcal{S}^j\phi_\tau )_{\delta,2j}
 = \left(\mathcal{S}^j\left(p_0(\phi)\right)\, , \, \S^j\phi_\tau\right)_{\delta,2j} + \left(\S^j\left(\frac{w_\delta'}{z} p_1(\phi)\right)\,,\S^j\phi_\tau \right)_{\delta,2j} \notag \\
& \ \  + \left(\mathcal{S}^j\left(p_2(\phi)M_{\delta,0}\phi \right)\,,\S^j\phi_\tau \right)_{\delta,2j}
+ \left(\mathcal{S}^j\left( p_3(\phi) M_{\delta,0}(p_4(\phi))\right)\,,\S^j\phi_\tau \right)_{\delta,2j} 
+ \left(\mathcal{S}^jQ[\phi,\phi] \,, \, \S^j\phi_\tau \right)_{\delta,2j} \label{E:NONLINEARSIMPLER}
\end{align}

\noindent
{\em Step 1: Estimates for the first and second term on the right-hand side of~\eqref{E:NONLINEARSIMPLER}.}
Using~\eqref{E:CHAINRULE} and~\eqref{E:PRODUCTRULE} for any $j\ge1$ we may rewrite
\begin{align*}
\S^j(p_0(\phi)) =& \sum_{\ell=0}^{j-1}{j-1\choose \ell}\left(\S^\ell(p_0'(\phi))\S^{j-\ell}\phi + \S^{\ell}\left(p_0''(\phi)\right)\S^{j-1-\ell}\left(|\pa_z\phi|^2\right)\right)+ B_j[\phi] 
\end{align*}
where $B_j[\phi]$ consists of remaining terms after applying the product rule \eqref{E:PRODUCTRULE}. Each factor in $B_j[\phi]$ contains at least one $\g_z$ derivative and in general a combination of $\g_z$ and $\S$. For instance, 
\begin{align*}
B_1[\phi]&=0, \quad B_2[\phi]= 2 \g_z(p_0'(\phi)) \g_z \S\phi +2 \g_z (p_0''(\phi)) \g_z|\g_z\phi|^2, \\
B_3[\phi]&= \S (B_2[\phi]) + 2 \g_z(p_0'(\phi)) \g_z \S^2\phi   +2 \g_z\S( p_0'(\phi)) \g_z \S\phi \\
& \quad+ 2 \g_z(p_0''(\phi)) \g_z \S (|\g_z\phi|^2) + 2 \g_z \S (p_0''(\phi)) \g_z(|\g_z\phi|^2)
\end{align*}
 and so on.    
Since $p_0'(0) = 0$ the above expression is at least  quadratic in $\phi.$ A routine application of Lemma~\ref{L:NORMENERGY}, the Hardy and Sobolev embedding inequalities of Lemmas~\ref{hard}--\ref{L:LINFINITYBOUNDS}
yields the estimate
$$
\|\S^j(p_0(\phi))\|_{\delta,2j} \lesssim \tilde\E.
$$
Together with $\|\S^j\phi_\tau\|_{\delta,2j}^2\lesssim e^{-\beta_2 \tau} \tilde\D$ we conclude that 
\begin{align}\label{E:ESTIMATE1}
\Big|\left(\mathcal{S}^j\left(p_0(\phi)\right)\, , \, \S^j\phi_\tau\right)_{\delta,2j}\Big| \lesssim e^{-\beta_2 \tau/2}\tilde\E \sqrt{\tilde\D} \lesssim \nu \tilde\D + \frac{e^{-\beta_2 \tau}\tilde\E}{\nu} \tilde\E.
\end{align}

From~\eqref{E:LE} and Lemma~\ref{lem:w} it follows that $\frac{w_\delta'}{z}$ is in fact a smooth function  and all of its $\S$-derivatives are uniformly bounded on the interval $[0,1].$
By the same argument as above we conclude that
\begin{align}\label{E:ESTIMATE2}
\Big|\left(\S^j\left(\frac{w_\delta'}{z} p_1(\phi)\right)\,,\S^j\phi_\tau \right)_{\delta,2j}\Big| \lesssim e^{-\beta_2 \tau/2}\tilde\E \sqrt{\tilde\D} \lesssim \nu \tilde\D + \frac{e^{-\beta_2 \tau}\tilde\E}{\nu} \tilde\E.
\end{align}

\noindent
{\em Step 2: Estimates for the third, fourth,  and the fifth term on the right-hand side of~\eqref{E:NONLINEARSIMPLER}.}
Applying the product rule~\eqref{E:PRODUCTRULE} and the chain rule~\eqref{E:CHAINRULE} we obtain
\begin{align}
&\mathcal{S}^j\left(p_2(\phi)M_{\delta,0}\phi \right)
= \sum_{k=0}^j{j \choose k} \S^{j-k}p_2(\phi)\S^{k}M_{\delta,0}\phi + A_j[\phi]
\end{align}
where $ A_j[\phi]$ consists of remaining nonlinear terms after applying the product rule \eqref{E:PRODUCTRULE} and each factor in $ A_j[\phi]$ contains at least one $\g_z$ derivative. We focus on the first summation. We can rewrite it as 
 \begin{align}
 & \sum_{k=0}^j{j \choose k} \S^{j-k}p_2(\phi)\S^{k}M_{\delta,0}\phi \notag \\
&=\sum_{k=0}^j{j \choose k} \S^{j-k}p_2(\phi) \left(M_{\delta,2k}\S^k\phi  +\sum_{\ell=0}^{k-1}\S^\ell\left[\left(\alpha_{k-\ell}\S+ \beta_{k-\ell}\g_z + \gamma_{k-\ell}\right)\S^{k-1-\ell}\phi\right]\right) \notag \\
& = p_2(\phi)M_{\delta,2j}\S^j\phi  + \mathcal R_j \label{E:THIRD},
\end{align}
where the remainder term $\mathcal R_j$ is given by
\begin{align}
\mathcal R_j & =  p_2(\phi)\sum_{\ell=0}^{j-1}\S^\ell\left[\left(\alpha_{k-\ell}\S+ \beta_{k-\ell}\g_z + \gamma_{k-\ell}\right)\S^{k-1-\ell}\phi\right] \notag \\
& \ \ + \sum_{k=0}^{j-1}{j \choose k} \S^{j-k}p_2(\phi) \left(M_{\delta,2k}\S^k\phi  +\sum_{\ell=0}^{k-1}\S^\ell\left[\left(\alpha_{k-\ell}\S+ \beta_{k-\ell}\g_z + \gamma_{k-\ell}\right)\S^{k-1-\ell}\phi\right]\right). \label{E:REMAINDER}
\end{align}
We isolated the top order term $p_2(\phi)M_{\delta,2j}\S^j\phi $ on the right-most side of~\eqref{E:THIRD} while the lower-order remainder term $\mathcal R_j$ is estimated by the Sobolev and Hardy inequalities on Lemmas~\ref{hard}--\ref{L:LINFINITYBOUNDS} and Lemma~\ref{L:NORMENERGY}.
To handle the top-order term, we note that 
\begin{align}
&(p_2(\phi)M_{\delta,2j}\S^j\phi   ,\mathcal{S}^j\phi_\tau )_{\delta,2j} 
= \int_0^1p_2(\phi) \pa_z\left(w_\delta^{4+2j}\frac1{z^2}\pa_z(z^3\S^j\phi)\right) z^3\S^j\phi_\tau\,dz \notag \\
& =-\frac12\pa_\tau\int_0^1p_2(\phi)\frac{1}{z^2}w_\delta^{4+2j}\big|\pa_z(z^3\S^j\phi)\big|^2\,dz + \frac12\int_0^1p_2'(\phi)\phi_\tau\frac1{z^2}\big|\pa_z(z^3\S^j\phi)\big|^2w_\delta^{4+j}\,dz \label{E:TOPORDER1}
\end{align}
Just like in~\eqref{E:TOPORDERONE} using the Hardy inequality we have
\begin{align*}
\int_0^1\frac{1}{z^2}w_\delta^{4+2j}\big|\pa_z(z^3\S^j\phi)\big|^2\,dz 
\lesssim \int_0^1 w_\delta^{4+2j} z^2 (\S^j\phi)^2\,dz + \int_0^1 w_\delta^{4+2j} z^4 (\pa_z\S^j\phi)^2\,dz \lesssim \tilde\E.
\end{align*}
Using~\eqref{E:POLYNOMIALS} and Lemma~\ref{L:LINFINITYBOUNDS} it follows that $\|p_2(\phi)\|_{\infty}\lesssim \sqrt{\tilde\E}$.
Therefore
\begin{align}
\Big|\int_0^\tau-\frac12\pa_\tau\int_0^1p_2(\phi)\frac{1}{z^2}w_\delta^{4+2j}\big|\pa_z(z^3\S^j\phi)\big|^2\,dz\,d\sigma\Big|
\lesssim \tilde\E(0) + \tilde\E^{\frac32}(\tau).
\end{align}
The second term on the right-most side of~\eqref{E:TOPORDER1} is estimated similarly:
\begin{align}
&\Big|\int_0^\tau\int_0^1p_2'(\phi)\phi_\tau\frac1{z^2}\big|\pa_z(z^3\S^j\phi)\big|^2w_\delta^{4+j}\,dz\,d\sigma\Big| \notag \\
&\lesssim 
\int_0^\tau \|p_2'(\phi)\phi_\tau\|_{L^\infty([0,1])}\int_0^1\frac{1}{z^2}w_\delta^{4+2j}\big|\pa_z(z^3\S^j\phi)\big|^2\,dz d\sigma \notag \\
& \lesssim \int_0^\tau e^{-\beta_2 \sigma/2}\sqrt{\tilde\D}(\sigma)\tilde\E(\sigma)\,d\sigma \lesssim \nu\int_0^\tau\tilde\D(\sigma)\,d\sigma + \sup_{0\le \sigma\le \tau}\tilde\E^2(\sigma).
\end{align}
Using Lemmas~\ref{hard}--\ref{L:LINFINITYBOUNDS} and Lemma~\ref{L:NORMENERGY} and the $L^\infty-L^2-L^2$ H\"older inequality the remainder term~\eqref{E:REMAINDER} is easily shown 
to satisfy the energy bound
\begin{align}\label{E:REMAINDERBOUND}
\Big|\int_0^\tau\left(\mathcal R_j\,,\S^j\phi_\tau \right)_{\delta,2j}\,d\sigma\Big| \lesssim \left(\nu+\sup_{0\le\sigma\le \tau}\sqrt{\tilde\E}(\sigma)\right)\int_0^\tau\tilde\D(\sigma)\,d\sigma + \sup_{0\le \sigma\le \tau}\tilde\E^2(\sigma).
\end{align}
From~\eqref{E:TOPORDER1}--\eqref{E:REMAINDERBOUND} we obtain the inequality
\begin{align}\label{E:THIRDBOUND}
&\Big|\int_0^\tau\left(\mathcal{S}^j\left(p_2(\phi)M_{\delta,0}\phi \right)\,,\S^j\phi_\tau \right)_{\delta,2j}\,d\sigma\Big| \notag \\
& \ \ \lesssim \tilde\E(0)+\left(\nu+\sup_{0\le\sigma\le \tau}\sqrt{\tilde\E}(\sigma)\right)\int_0^\tau\tilde\D(\sigma)\,d\sigma + \sup_{0\le \sigma\le \tau}\tilde\E^2(\sigma).
\end{align}
The fourth and the fifth term on the right-hand side of~\eqref{E:NONLINEARSIMPLER} are estimated in an analogous manner:
\begin{align}\label{E:FOURTHFIFTHBOUND}
&\Big|\int_0^\tau\left[\left(\mathcal{S}^j\left( p_3(\phi) M_{\delta,0}(p_4(\phi))\right)\,,\S^j\phi_\tau \right)_{\delta,2j} 
 + \left(\mathcal{S}^jQ[\phi,\phi] \,, \, \S^j\phi_\tau \right)_{\delta,2j}\right] \,d\sigma\Big| \notag \\
 & \ \  \lesssim \tilde\E(0)+\left(\nu+\sup_{0\le\sigma\le \tau}\sqrt{\tilde\E}(\sigma)\right)\int_0^\tau\tilde\D(\sigma)\,d\sigma + \sup_{0\le \sigma\le \tau}\E^2(\sigma).
\end{align}
Summing~\eqref{E:ESTIMATE1}--\eqref{E:ESTIMATE2} and~\eqref{E:THIRDBOUND}--\eqref{E:FOURTHFIFTHBOUND} we obtain~\eqref{E:CJESTIMATE}.
\end{proof}

\subsubsection{Proof of Theorem~\ref{T:HIGHORDER}}

Integrating  the energy identity \eqref{E:ENERGYIDENTITYPART2} with respect to $\tau$ we obtain 
\begin{align}
&\frac{1}{2} \left( \tilde\lambda  \| \mathcal{S}^j\phi_\tau\|^2_{\delta,2j} + 3\delta\|\mathcal{S}^j\phi \|^2_{\delta,2j} +\frac43 \|\g_z\mathcal{S}^j\phi\|^2_{\delta,2j+1} \right) \Big|^\tau_0\notag  
+\int_0^\tau \frac{\tilde\lambda _\tau}{2}  \| \mathcal{S}^j\phi_\tau\|^2_{\delta,2j} d\sigma  \\
&\ \ \ \  = \int_0^\tau \mathcal Q_jd\sigma  +\int_0^\tau \mathcal C_j d\sigma \label{energytilde}
\end{align}
for $0\le j\le 4$. If $-\tilde\varepsilon<\delta\leq 0$, we apply Lemma \ref{L:POSDEF} to the first term of the linear combination of \eqref{energytilde} and use the estimates for $\mathcal{Q}_j$ \eqref{E:QJBOUND} and $\mathcal{C}_j$ \eqref{E:CJESTIMATE}  to obtain 
\begin{align*}
\sup_{0\le\tau'\le \tau} \tilde{\E} (\tau') + C_1 \int_0^\tau \tilde{\mathcal{D}} d\sigma \le C_2\tilde{\E} (0) +C_3 |\tilde{\mathcal{J}} [\phi]|^2 +C_4 \sup_{0\le \sigma\le \tau}\left(  \tilde\E^\frac{3}{2}(\sigma)  + \tilde\E^2(\sigma) \right) \\
 +C_5\left(\nu+\sup_{0\le\sigma\le \tau}\sqrt{\tilde\E}(\sigma)\right) 
  \int_0^\tau \tilde{\mathcal{D}} d\sigma 
+C_\nu  \int_0^\tau e^{-\beta_2\sigma}\tilde\E (\sigma)d\sigma  . 
\end{align*}
From \eqref{E:TILDEJDEFINITION}, it is straightforward to check $|\tilde{\mathcal{J}}[\phi]|^2 \lesssim  \tilde\E^2$. Therefore, by choosing sufficiently small $\nu>0$ and $\tilde M>0$ if necessary, we deduce  the energy-dissipation bound  \eqref{E:MAINBOUND2}. 
If $\delta>0$, by using the energy positivity \eqref{Epos}, the estimates for $\mathcal{Q}_j$ \eqref{E:QJBOUND} and $\mathcal{C}_j$ \eqref{E:CJESTIMATE}, and the smallness of $\tilde M$, one can easily deduce \eqref{E:MAINBOUND2}.

\subsection{Proof of Theorem~\ref{T:LINEAREXPANSION}}

By Remark~\ref{R:ALTERNATIVE} the local-in-time well-posedness theory implies that the unique solution to the initial value problem~\eqref{E:PHIEQUATION2}--\eqref{E:PHIINITIAL2} exists on a time interval $[0,T]$
where $T\sim\frac1{\tilde{\mathcal E}(0)}\gtrsim\frac1\varepsilon$. Choose $\varepsilon>0$ so small that the time of existence $T$ satisfies 
\be\label{E:LOCALTIME}
e^{-\beta_2T/2}\le\frac{\beta_2}{2\tilde C_3}, \ \ \sup_{\tau\le T}\tilde\E(\tau)\le C\tilde\E(0).
\ee
Define 
\[
\mathcal T : = \sup_{\tau\ge0}\{\text{ solution to~\eqref{E:PHIEQUATION2}--\eqref{E:PHIINITIAL2} exists on $[0,\tau)$ and} \ \sup_{0\le\tau'\le \tau}\tilde\E(\tau') \le 4\tilde C_2C\E(0)\} . 
\]
Observe that $\mathcal T\geq T$. 
From~\eqref{E:MAINBOUND2} we obtain 
\be\label{E:ENERGYBOUNDONE}
\sup_{\tau'\in[\frac T2,\tau]} \tilde\Energy(\tau') + \tilde C_1\int_{\frac T2}^\tau\tilde\D(\tau')\,d\tau' \le \tilde C_2\tilde\Energy(\frac T2) 
+ \tilde C_3\int_{\frac T2}^\tau e^{-\beta_2\tau'}\tilde\Energy(\tau')\,d\tau', \ \ \tau\in[\frac T2,\mathcal T].
\ee
Therefore, using~\eqref{E:LOCALTIME} we conclude that
\be
\sup_{\tau'\in[\frac T2,\tau]} \tilde\Energy(\tau') + \tilde C_1\int_{\frac T2}^\tau\tilde\D(\tau')\,d\tau' \le \tilde C_2\tilde\Energy(\frac T2) + \frac{\tilde C_3}{\beta_2}e^{-\beta_2 T/2}\sup_{\tau'\in[\frac T2,\tau]} \tilde\Energy(\tau')
\le \tilde C_2\tilde\Energy(\frac T2 ) +\frac12 \sup_{\tau'\in[\frac T2,\tau]} \tilde\Energy(\tau').
\ee 
From this inequality we conclude that 
\be
\sup_{\tau'\in[\frac T2,\tau]} \tilde\Energy(\tau') \le 2\tilde C_2\tilde\Energy(\frac T2)\le 2\tilde C_2C\tilde\E(0), \ \ \tau\in[\frac T2,\mathcal T].
\ee
From the continuity of the map $\tau\mapsto \sup_{\tau'\in[T,\tau]} \tilde\Energy(\tau')$ and the definition of $\mathcal T$ we conclude that 
$\mathcal T = \infty$ and the solution to~\eqref{E:PHIEQUATION2}--\eqref{E:PHIINITIAL2} exists globally-in-time. From the proved global-in-time boundedness of 
$\tilde\E$, the estimate $\sum_{j=0}^4\tilde\lambda  \|\S^j\phi_\tau\|_{\delta,2j}^2\lesssim \tilde\E$, and  the bound~\eqref{E:TILDELAMBDABOUND}, the second claim of~\eqref{E:ESTIMATESFINAL}
follows.  

It now remains to prove the first part of~\eqref{E:ESTIMATESFINAL}. From the global-in-time boundedness of $\tilde \E$, there exists a weak limit $\phi_\infty$ independent of $\tau$ such that $\tilde\E(\phi_\infty,0)= \sum_{j=0}^4 \|\S^j \phi_\infty\|_{\delta,2j}^2 + 
\|\g_z\S^j \phi_\infty\|_{\delta,2j+1}^2 \lesssim \tilde\E \le C\varepsilon$, which in particular implies $\S^j \phi_\infty \in H^1_{\delta, 2j}$ for $0\le j\le 4$. To show the strong convergence of $\phi(\tau)$  as $\tau \rightarrow\infty$ in  $\mathfrak{H}^4_{\delta,0}=\{\phi : \S^j \phi_\infty \in L^2_{\delta, 2j}, \text{ for }0\le j\le 4\}$ (low-regularity space), we observe that for any $0<\tau_2<\tau_1$, 
\[
\begin{split}
\|  \S^j \phi (\tau_1) -  \S^j \phi(\tau_2) \|^2_{L^2_{\delta, 2j}}&  = \int_0^1z^4w^{3+2j} \left| \int_{\tau_2}^{\tau_1} \S^j\phi_\tau d\tau 
\right|^2dz\\
&\le \left( \int_{\tau_2}^{\tau_1} \frac{1}{\tilde\lambda _\tau} d\tau \right) \left( \int_{\tau_2}^{\tau_1}  \tilde \lambda_\tau \int_0^1z^4w^{3+2j} |\S^j\phi_\tau |^2dz d\tau \right)\\
&  \lesssim ( e^{- \beta_2\tau_2} -e^{- \beta_2\tau_1}  )  \int_{\tau_2}^{\tau_1} \tilde \D d\tau    \lesssim ( e^{- \beta_2\tau_2} -e^{- \beta_2\tau_1}  ) \, \varepsilon.
\end{split}
\]
Therefore, given a strictly increasing sequence $\tau_n \rightarrow \infty$, the sequence $\{\S^j \phi (\tau_n)\}_{n=1}^\infty$  is Cauchy in $\mathfrak{H}^4_{\delta,0}$. This completes the proof. 


\section*{Acknowledgments}

The authors thank Dr. Y. Zwols for his kind help with the drawing of Figure~\ref{F:BIFURCATION}. JJ is supported in part by NSF grants DMS-1212142, DMS-1608494, and a von Neumann fellowship of the Institute for Advanced Study through the NSF grant DMS-1128155. MH acknowledges the support of the EPSRC Grant EP/N016777/1.

\section*{Appendix}

\appendix

\section{Spectral-theoretic properties of $\mathcal L_\delta$ and $\mathcal L_{\delta,k}$}\label{A:A}

 In this section, we discuss the spectral theory of the linearized operators $\mathcal L_\delta$ and $\mathcal L_{\delta,k}$ defined in \eqref{E:LDELTA} and Definition~\ref{D:OPERATORS} respectively  for $-\varepsilon<\delta\leq 0$ where $\varepsilon>0$ is sufficiently small. Recall that for any $k\in\mathbb N \cup \{0\}$, 
\begin{align}
\mathcal L_{\delta,k}\psi&= - \frac{4}{3w_\delta^{3+k}z^4}\g_z (w_\delta^{4+k} z^4\g_z\psi)  . 
\end{align}
Here we identify $\mathcal{L}_{\delta,0}=\mathcal L_\delta$. Notice that $\mathcal L_{\delta,k}$ is a singular Sturm-Liouville operator, and it is nonnegative and symmetric with respect to $(,)_{\delta,k}$: 
\[
(\mathcal L_{\delta,k}\psi,\psi)_{\delta,k}= (\psi, \mathcal L_{\delta,k}\psi)=\frac43\|\g_z\psi\|_{\delta,k}^2
\]
 for any $\psi\in C^\infty_c([0,1])$. Since $C^\infty_c([0,1])$ is dense in the weighted Sobolev spaces, it is natural to consider the Friedrichs extension in $L^2_{\delta,k}$. 
 
 When $w_\delta$ is given by the Lane-Emden solution, namely when $\delta=0$, it was shown in \cite{Be1995, Makino2015} that $\mathcal{L}_{0,0}$ has the Friedrichs extension in $L^2_{0,0}$, which is a well-defined self-adjoint operator whose spectrum is purely discrete with $0$ as the smallest isolated eigenvalue. The proof relies on the qualitative boundary behavior of $w$: $w'(0)=0$, $w(1)=0$ and $w(z)\sim 1-z$ for $z<1$ sufficiently close to 1. Since $w_\delta$ behaves in the same way near the boundaries, the same conclusion holds for $\mathcal L_\delta$ for any $\delta<0$ sufficiently close to 0. The same argument works for $\mathcal L_{\delta,k}$ for any $k\in \mathbb{N}$. In order to state the result, let us introduce appropriate function spaces: 
 \be\label{Hkj}
 H_{\delta,k}^j:=\{\psi\in L^2_{\delta,k}| \g_z^l \psi \in L^2_{\delta,k+l}, \ \  \text{ for all } \ \ 0\leq l\leq j\}. 
 \ee
When $k=0$, $H_{\delta,k}^j = H_{\delta}^j$ where $ H_{\delta}^j=H_{w_\delta}^j$ and $L^2_{\delta,k}$ were defined in Definition \ref{Def:H^j}.

\begin{proposition} $\mathcal L_{\delta,k}$ has the Friedrichs extension in $L^2_{\delta,k}$ with the domain $H_{\delta,k}^2$. It is a self-adjoint operator  with respect to $(\cdot,\cdot)_{\delta,k}$ whose spectrum consists of simple eigenvalues $\mu_0<\mu_1<  \dots< \mu_n <\mu_{n+1} <\cdots \rightarrow \infty$. The smallest eigenvalue is $\mu_0=0$ and the corresponding eigenspace consists of constant functions. 
\end{proposition}

\begin{proof} The case of $\mathcal{L}_0$ directly follows from Proposition 1 of \cite{Makino2015} and other cases can be treated similarly. For completeness, we will describe the argument in \cite{Makino2015} for other cases. 

In order to apply the classical theory, we first perform the so-called Liouville transformation: 
\[
y:=\int_0^z \sqrt{\frac{w_\delta^{3+k}}{\tfrac43 w_\delta^{4+k}}}d\tilde z, \quad \eta:= z^2(\tfrac43w_\delta^{7+2k})^\frac14 \psi. 
\]
Here $y\in [0,y_+]$ where $y_+=\int_0^1 \sqrt{\frac{w_\delta^{3+k}}{\tfrac43 w_\delta^{4+k}}}d z<\infty$. 
Then 
\[
\mathcal L_{\delta,k}\psi =  \frac{1}{z^2(\tfrac43w_\delta^{7+2k})^\frac14}\left(- \g_y^2 \eta + q \eta \right)
\]
where 
\[
q = q(y) = \frac{\g_y^2(w_\delta^{\frac{7+2k}{4}})}{w_\delta^{\frac{7+2k}{4}}} + 4  \frac{\g_y z}{z} \frac{\g_y(w_\delta^{\frac{7+2k}{4}})}{w_\delta^{\frac{7+2k}{4}} } + 2 \frac{z\g_y^2 z + (\g_y z)^2}{z^2} . 
\]

Let us consider the operator $\mathcal{T}_0\eta= - \g_y^2 \eta + q \eta$ for $\eta\in C_c^\infty(0,y_+)$. 
For $y$ sufficiently small, $y\sim \sqrt{\frac{3}{4w_\delta(0)}} z$ and the last term in $q$ dominates. Using $\g_y z= \sqrt{\frac{4w_\delta}{3}}$, we thus obtain 
\[
q \sim 2 \frac{\frac{4w_\delta}{3}}{z^2} \sim \frac{2}{y^2} \ \  \text{ as } \ \ y\rightarrow 0. 
\]
For $y$ close to $y_+$, we have $\frac{c}{3}(y_+-y)^2\sim 1-z$ where $c=-\g_zw_\delta(1)$. The biggest contribution of $q$ in this regime comes from the first term. It is easier to compare in $z$ variable. Using $\g_y=\sqrt{\frac{4w_\delta}{3}}\g_z$, we get 
\[
q\sim \frac{(7+2k)(5+2k)}{12}\frac{(\g_zw_\delta)^2}{w_\delta} \sim \frac{(7+2k)(5+2k)}{4}\frac{1}{(y_+-y)^2}. 
\]
Hence, the boundary points $y=0, \; y_+$ are of limit point type for all $k\geq 0$ and for all $\delta$. Therefore, by Theorem X.10 of  \cite{ReSi}, the operator $\mathcal{T}_0$ with the domain $C_c^\infty(0,y_+)$ has the Friedrichs extension $\mathcal{T}$ in $L^2(0,y_+)$, a self-adjoint operator with simple eigenvalues $\mu_0<\mu_1<  \dots< \mu_n <\mu_{n+1} <\cdots \rightarrow \infty$. The domain of $\mathcal{T}$ is characterized by $\{\eta\in C[0,y_+] : \eta(0)=\eta(y_+)=0, \;\; -\g_y^2\eta+ q \eta \in L^2(0,y_+) \}$. By unwinding these results in terms of $\mathcal L_{\delta,k}$ and $\psi$ in $z$ variable, we obtain the desired result on the spectral theory of $\mathcal L_{\delta,k}$ in the weighted spaces $L^2_{\delta,k}$. Notice that $\eta\in L^2(0,y_+)$ corresponds to $\psi\in L^2_{\delta,k}$. Furthermore, we deduce that the domain of $L^2_{\delta,k}$ is $H_{\delta,k}^2$ from the elliptic regularity (for instance, see Lemma 5.2 in \cite{JaMa2015}): for each $f\in L^2_{\delta,k}$, there exists a unique solution  $\psi\in H^2_{\delta,k}$ solving $\mathcal L_{\delta,k}\psi =f$. 

To finish the proof of the proposition, it suffices to show that $0$ is an eigenvalue of $\mathcal L_{\delta,k}$. It is easy to see that the set of constant functions belongs to kernel of $\mathcal L_{\delta,k}$. Since $\mathcal L_{\delta,k}$ is non-negative and by the spectral theory, the last assertion follows. 
\end{proof}

We are now ready to prove Lemma  \ref{lem:core} and Lemma \ref{lem:core2}. Since the proofs of two lemmas are almost identical, we only provide the proof of Lemma \ref{lem:core}. 
 
\begin{proof}[Proof of Lemma \ref{lem:core}]
  
From the spectral theory of $\mathcal L_\delta$, the operator
\[
\mathcal L_\delta\phi = -\frac43 \partial_z\left(w_\delta^4z^4\partial_z\phi\right)
\]
is completely diagonalizable and the smallest eigenvalue is precisely 0, i..e:
\[
\mathcal{L} _\delta\phi_k = \mu_k w_\delta^3z^4\phi_k, \ \ k\in\mathbb N\cup\{0\}, 
\]
and
\[
0=\mu_0<\mu_1< \mu_2 < \dots
\]
This implies the first statement. For the second claim, note that 
\[
\left((\mathcal L_\delta +3\delta)\varphi,\varphi\right)_\delta = \frac43 \|\partial_z\varphi\|_{\delta,1}^2 - \frac{3}{2}b^2\|\varphi\|_\delta^2 \ge (\mu_1-\frac32 b^2)\|\varphi\|_\delta^2
\]
Hence for sufficiently small $b$ satisfying $b^2 <\frac23\mu_1$, we deduce that 
\[
\left((\mathcal L_\delta +3\delta)\varphi,\varphi\right)_\delta \gtrsim \|\partial_z\varphi\|_{\delta,1}^2+\|\varphi\|_\delta^2. 
\]
\end{proof}

\section{Hardy inequalities and embeddings of weighted Sobolev spaces}\label{A:B}

The $L^\infty$ bounds on $\phi$ and its derivatives with suitable weights can be obtained by using the Hardy inequalities and embedding inequalities. 
Since our energy norms involve different weights near the origin and near the boundary, we will utilize localized Hardy inequalities as in \cite{Jang2014, Jang2015}. In order to state the results, we introduce 
suitable functions spaces. Let  $Z_0$ be a completion of $C_c^\infty([0,1])$ with respect to the norm $\|\cdot\|_{Z_0}$ generated by the inner product $(u,v)_{Z_0}:=\int_0^{1}z^4 u v  dz$. 
For any $i\in \mathbb{N}$ we define $Z_i$   by 
\[
Z_i:=\{u\in Z_0 : \g_z^l u \in Z_0 \ \text{ for all } \ 0\leq  l\leq i\}
\]
with $\|u\|^2_{Z_i}=\sum_{l=0}^i \|\g_z^l u \|^2_{Z_0}$. 
Similarly, we introduce $X_a$ as a completion of  $C_c^\infty([0,1])$ with respect to the norm $\|\cdot\|_{X_a}$ generated by the inner product $(u,v)_{X_a}:=\int_0^{1}w_\delta^a u v dz$ for $a>1$. We let 
\[
X_a^i:= \{ v\in X_a : \g_z^l v \in X_{a} \  \text{ for all }\ 0\leq  l\leq i\}
\]
with $\|v\|^2_{X_a^1} =\|v\|^2_{X_a} + \|\g_z v\|^2_{X_a}$.

\begin{lemma}[Hardy inequalities]\label{hard}  

\begin{itemize}
\item[(1)] For any $u\in Z_1$, we have 
\begin{equation}\label{hardy0}
\int_0^{\frac{1}{2}} z^2 |u|^2 dz \lesssim \int_0^{\frac{3}{4}} z^4 |u_z|^2  dz +\int_{0}^{\frac{3}{4}} z^4 | u|^2 dz. 
\end{equation}

\item[(2)] For any $u\in Z_2$, we have 
\begin{equation}\label{bts}
\int_0^{\frac{1}{2}} |u|^2 dz \lesssim \int_{0}^{\frac{3}{4}} z^4 | u_{zz}|^2 dz + \int_0^{\frac{3}{4}} z^4 |u_z|^2  dz+\int_{0}^{\frac{3}{4}} z^4 | u|^2 dz. 
\end{equation}
\item[(3)]   Let $a>1$ be given. For any $v\in X_a^1$, we have 
\begin{equation}\label{Hardy-gw}
\int_{\frac{1}{2}}^1 w_\delta^{a-2}| v|^2 dz \lesssim  \int_{\frac{1}{4}}^1 w_\delta^a | v_z|^2 dz +\int_{\frac{1}{4}}^1 w_\delta^a | v|^2 dz  . 
\end{equation}
\end{itemize}
\end{lemma}

As a consequence of the previous lemma, we have Hardy embedding inequalities. 

\begin{lemma}[Embedding inequalities] Let $m$ be any nonnegative integer. 
\begin{itemize}
\item[(1)] For any $u\in Z_2\cap X_{2m}^m$, we have 
\begin{equation}\label{L1}
\|u\|_{L^1}^2 \lesssim \sum_{k=0}^2\int_0^{\frac{3}{4}} z^4 |\partial_z^k u|^2 dz +  \sum_{k=0}^{m}\int_{\frac{1}{4}}^1w_\delta^{2m}   |\partial_z^k  u|^2 dz . 
\end{equation}

\item[(2)] For any $u\in Z_3\cap X_{2m}^{m+1}$, we have 
\begin{equation}\label{Linfty}
\|u\|_{\infty}^2\lesssim \sum_{k=0}^3\int_0^{\frac{3}{4}} z^4 |\partial_z^k u|^2 dz +  \sum_{k=0}^{m+1}\int_{\frac{1}{4}}^1w_\delta^{2m}  |\partial_z^k  u|^2 dz. 
\end{equation}
\end{itemize}
\end{lemma}

A direct consequence of the above lemmas are the $L^\infty$ bounds presented in the lemma below.  
We shall explain how to control the $L^\infty$ norms by two different notions of energy $\E$ and $\tilde{\mathcal E}$ introduced in~\eqref{energy1} and~\eqref{E:ENERGYLINEAR} respectively.

\begin{lemma}\label{L:LINFINITYBOUNDS} 
\begin{enumerate} 
\item 
Let $\phi=\phi(s,z)\in H^8_\delta$ for each fixed $s\in [0,S]$ be given so that $\g_s^i\phi\in H^{8-i}_\delta$ for each $0\leq i\leq 8$ and that the corresponding energy $\E(s) <\infty$ for all $s\in [0,S]$. We recall $H^j_\delta$ is defined in Definition \ref{Def:H^j} and $\E$ is given in \eqref{energy1}. Then the following estimates hold: 
\begin{enumerate}
\item[(a)] 
\[
\big|\phi\big|+\big|\g_s\phi\big|   +\big|\g_s^2\phi \big|+\sum_{q=1}^{5} \big|z^{\delta(q)}w^{\frac{q-1}{2}} \partial_s^q\partial_s^2\phi  \big|
 \lesssim \sqrt{{ \Energy}} 
\]
where  ${\delta(q)}=0$ for $q\leq 3$, $\delta(q)=1$ for $q=4$ and $\delta(q)=2$ for $q=5$. 
\item [(b)]
\[
\big|\g_z\phi\big| +\big|\g_{zs}\phi\big|+ \sum_{q=1}^{5}\big| z^{\delta(q)} w^{\frac{q}{2}}\partial_s^q\g_{zs}\phi\big|  \lesssim \sqrt{{ \Energy}} 
\]
where  ${\delta(q)}=0$ for $q\leq 3$, $\delta(q)=1$ for $q=4$ and $\delta(q)=2$ for $q=5$. 
\item [(c)]
\begin{equation}\label{rrbound}
\sum_{q=0}^{5}\left|  w^{\frac{q+1}{2}} z^{\delta(q)} \partial_s^q \partial_z^2 \phi \right|\lesssim \sqrt{{ \Energy}}  
\end{equation}
where ${\delta(q)}=0$ for $q\leq 3$, $\delta(q)=1$ for $q=4$ and $\delta(q)=2$ for $q=5$. 
\end{enumerate}

\item Let $\phi=\phi(\tau,z)\in \mathfrak{H}^4_{\delta,0}$ for each fixed $\tau\in [0,T]$ be given so that $\g_\tau \phi \in \mathfrak{H}^4_{\delta,0}$, $\S^4 \phi \in H^1_{\delta, 8}$, and that the corresponding energy $\tilde\E<\infty$ and dissipation $\tilde\D<\infty$ for all $\tau\in [0,T]$. We recall $\mathfrak{H}^j_{\delta,0}$ is defined in Definition \ref{D:DELTAJSPACES}, $H^j_{\delta,k}$ is introduced in \eqref{Hkj}, and $\tilde\E$, $\tilde\D$ in \eqref{E:ENERGYLINEAR} and \eqref{dissipation} respectively. Then the following estimates hold:

\begin{enumerate}
\item[(a)] 
\[
\big| \phi \big|  + \big|\g_z \phi \big| + \big|\g_z^2 \phi \big|  + \big| \S\phi  \big|  \lesssim \sqrt{{\tilde\Energy}} 
\]
\item [(c)]
\[
\big|z^{\delta(q)}  w^q \g_z^{2+q} \phi  \big|+  \big| z^{\delta(q)}  w^q \g_z^{q} \S\phi  \big|  \lesssim \sqrt{{\tilde\Energy}} 
\]
for $1\le q\le 5$. ${\delta(q)}=0$ for $q\leq 3$, $\delta(q)=1$ for $q=4$ and $\delta(q)=2$ for $q=5$. We also have 
\[
\big| w^2 \S^2 \phi \big| + \big| z w^4 \S^3 \phi \big|  \lesssim \sqrt{{\tilde\Energy}}
\]
\item [(c)]
\[
\sqrt{\tilde\lambda }\big| \g_\tau \phi \big|  \lesssim \sqrt{{ \tilde\D}} 
\]
\end{enumerate}

\end{enumerate}
\end{lemma}

 For the proofs of the above lemmas, we refer to \cite{Jang2014,Jang2015}.

\end{document}